\documentclass[leqno,11pt,a4paper,twoside]{article}

\usepackage{moreverb}
\usepackage{amsthm,amsmath,amsfonts,amssymb,amsbsy,gensymb}
\usepackage{graphicx,color,euscript,epstopdf}

\newcommand\BibTeX{{\rmfamily B\kern-.05em \textsc{i\kern-.025em b}\kern-.08em
T\kern-.1667em\lower.7ex\hbox{E}\kern-.125emX}}

\newtheorem{theorem}{Theorem}[section]
\newtheorem{proposition}[theorem]{Proposition}
\newtheorem{lemma}[theorem]{Lemma}
\newtheorem{assumption}[theorem]{Assumption}
\newtheorem{definition}[theorem]{Definition}
\newtheorem{remark}[theorem]{Remark}
\newtheorem{example}[theorem]{Example}

\newtheorem*{lemmaaux}{Lemma~2.12}
\newtheorem*{propaux}{Proposition~2.29}
\newtheorem*{theoremaux}{Theorem~2.28}

\newcounter{count}

\begin{document}

\title{Synthesis of control Lyapunov functions and stabilizing feedback strategies using exit-time optimal control}

\author{Ivan Yegorov \thanks{Department of Mathematics, North Dakota State University,
			North Dakota, USA}
\and
        Peter M. Dower \thanks{Department of Electrical and Electronic Engineering, University of Melbourne, Victoria, Australia}
\and
        Lars Gr\"une \thanks{Chair of Applied Mathematics, University of Bayreuth,
			Germany}
}

\maketitle


\begin{abstract}
This paper studies the problem of constructing control Lyapunov functions (CLFs) and feedback stabilization strategies
for deterministic nonlinear control systems described by ordinary differential equations. Many numerical methods for
solving the Hamilton--Jacobi--Bellman partial differential equations specifying CLFs typically require dense state
space discretizations and consequently suffer from the curse of dimensionality. A relevant direction of attenuating
the curse of dimensionality concerns reducing the computation of the values of CLFs and associated feedbacks at any selected
states to finite-dimensional nonlinear programming problems. In this work, exit-time optimal control is used for that purpose.
First, we state an exit-time optimal control problem with respect to a sublevel set of an appropriate local CLF and
establish that, under a number of reasonable conditions, the concatenation of the corresponding value function and
the local CLF is a global CLF in the whole domain of asymptotic null-controllability. This leads to a curse-of-dimensionality-free
approach to feedback stabilization. We also investigate the formulated optimal control problem. A modification of these
constructions for the case when one does not find a suitable local CLF is provided as well. Supporting numerical simulation
results that illustrate our development are subsequently presented and discussed. Furthermore, it is pointed out that
the curse of complexity may cause significant issues in practical implementation even if the curse of dimensionality is
mitigated.
\end{abstract}



\section{Introduction}

In control theory and engineering, feedback stabilization methods for nonlinear dynamical systems are of both
theoretical and practical importance, and control Lyapunov functions (CLFs) constitute a fundamental tool
there~\cite{Isidori1995,Sepulchre1997,Nikitin1994,CamilliGruneWirth2008,GieslHafstein2015,Afanasiev1996,
Fantoni2002,Malisoff2009,ChoukchouBraham2014}. As was established in \cite{CamilliGruneWirth2008} for
a relatively wide subclass of deterministic control systems described by ordinary differential equations
(ODEs) without state constraints, the value functions of appropriate infinite-horizon optimal control
problems are CLFs and also the unique viscosity solutions of boundary value problems for the corresponding
Hamilton--Jacobi--Bellman (HJB) partial differential equations (PDEs) of first order. This is in fact
an extension of the classical Zubov method for finding Lyapunov functions \cite{Zubov1964} to problems of
weak asymptotic null-controllability. Moreover, the framework of \cite{CamilliGruneWirth2008} can be extended
to some state-constrained problems (see \cite[Example~5.2]{GruneZidani2015}).

Exact solutions of boundary value, initial value, and mixed problems for HJB equations are known only in very
special cases. Many broadly used numerical approaches to solving these problems, including
semi-Lagrangian schemes
\cite{BardiCapuzzoDolcetta2008,FalconeFerretti1998,Falcone2006,CristianiFalcone2007,ROCHJ2017},
finite-difference schemes
\cite{FlemingSoner2006,CrandallLions1984,OsherShu1991,JiangPeng2000,ZhangShu2003,BokanForcadelZidani2010,
BokanCristianiZidani2010,ROCHJ2017},
finite element methods
\cite{JensenSmears2013},
and level set methods
\cite{OsherSethian1988,Osher1993,Sethian1999,OsherFedkiw2003,MitchellBayenTomlin2005,MitchellLevelSetToolbox2012},
typically rely on dense state space discretizations. With the increase of the state space dimension,
the computational cost of such grid based techniques grows exponentially. Their practical implementation is
in general extremely difficult (even on supercomputers) if the state space dimension is greater than $ 3 $,
which leads to what R.~Bellman called the curse of dimensionality~\cite{Bellman1957,Bellman1961}. Possible
ways to attenuate the curse of dimensionality for various classes of HJB equations and also more general
Hamilton--Jacobi (HJ) equations, such as Hamilton--Jacobi--Isaacs (HJI) equations for zero-sum two-player
differential games, have therefore become an important research area. A number of related approaches have been
developed for particular classes of problems (see, e.\,g., the corresponding overview in
\cite[Introduction and Section~4]{YegorovDower2017}). It has to be emphasized that, even when the curse of
dimensionality is mitigated, the so-called curse of complexity may still cause significant issues in
numerical implementation~\cite{McEneaney2007,YegorovDower2017,KangWilcox2017}.

A relevant direction of attenuating the curse of dimensionality for certain classes of first-order HJ equations
is reducing the evaluation of their solutions at any selected states to finite-dimensional optimization
(nonlinear programming) problems~\cite{YegorovDower2017,DarbonOsher2016,ChowDarbonOsherYin2018,
ChowDarbonOsherYin2017,KangWilcox2017}. In contrast with the aforementioned grid based techniques,
this leads to the following advantages:
\begin{itemize}
\setlength\itemsep{0em}
\item  the solutions can be evaluated independently at different states, which allows for mitigating
the curse of dimensionality;
\item  since different states are separately treated, one can choose arbitrary bounded regions and grids for
computations and arrange parallelization;
\item  when obtaining the value functions, i.\,e., the solutions of HJB or HJI equations, at selected states
by solving the related finite-dimensional optimization problems, one can usually retrieve the corresponding control
actions as well, without requiring possibly unstable approximations of the partial derivatives of the value functions.
\end{itemize}
However, the curse of complexity still takes place if the considered nonlinear programming problems are
essentially multi-extremal or if one wants to construct global solution approximations in high-dimensional
regions.

The finite-dimensional optimization problems describing the values at arbitrary isolated states of
the solutions of first-order HJB equations in optimal control problems may build on the (generalized)
method of characteristics for such PDEs \cite{YegorovDower2017,ChowDarbonOsherYin2018,KangWilcox2017}
(related also to Pontryagin's principle~\cite{Pontryagin1964,Cesari1983,Geering2007}), or on so-called direct
approximation techniques~\cite{PattersonRao2014,GPOPS_II_User_guide,ICLOCS2,Becerra2010,PSOPT_Manual,
BonnansMartinon2017,Houska2011,ACADO_Manual,Benson2006,Garg2011}. The latter involve direct transcriptions (approximations) of
infinite-dimensional optimal open-loop control problems to finite-dimensional nonlinear programming problems
via discretizations in time applied to state and control variables, as well as to dynamical state equations.
In this context, the frameworks based on Pontryagin's principle and the method of characteristics are called
indirect. In comparison, the direct numerical approaches are in principle less precise and less justified
from a theoretical perspective, but often more robust with respect to initialization and more straightforward
to use.

For designing a curse-of-dimensionality-free approach to feedback stabilization in one of the ways discussed
above, it is crucial first to bridge the gap between the infinite-horizon Zubov type setting of
\cite{CamilliGruneWirth2008} and numerical optimization frameworks handling only finite terminal (exit) times.
To that end, one can impose an appropriate terminal condition leading to an exit-time optimal control problem.
Such a formulation is in particular involved in the work~\cite{MichalskaMayne1993} developing model predictive
control (MPC) schemes for stabilization, while some other MPC studies, such as
\cite{ChenAllgower1998,Fontes2001,Jadbabaie2005}, use terminal conditions with fixed horizon length.
In general, the works~\cite{MichalskaMayne1993,ChenAllgower1998,Fontes2001,Jadbabaie2005} adopt local asymptotic
controllability conditions and establish the existence of sufficiently small sampling times and sufficiently large
prediction horizons such that systems driven by the corresponding MPC algorithms become asymptotically stable for
given initial states.

In comparison, this paper establishes global characterizations of CLFs via exit-time optimal control, serving as
a theoretical basis for curse-of-dimensionality-free approaches to feedback stabilization. It in fact extends the results of
our conference papers~\cite{YegorovDowerGrune2018_1,YegorovDowerGrune2018_2} and provides detailed proofs, discussions, and
some practical developments.

This paper is organized as follows. In Section~2, we state an exit-time optimal control problem with respect to
a sublevel set of an appropriate local CLF similarly to \cite{MichalskaMayne1993}. It is then shown that, under a number of
reasonable conditions, the concatenation of the corresponding value function and the local CLF is a global CLF in the whole
domain of asymptotic null-controllability. We also investigate the formulated problem and derive a characteristics based
representation of the value function. Section~3 presents a modification of these constructions for the case when a suitable
local CLF is not found. Namely, the terminal set in the exit-time optimal control problem is taken as a sufficiently small
closed ball centered at the origin, the terminal cost is chosen as zero, and we in particular establish sufficient conditions
for the uniform convergence of the associated value function to the original CLF from the infinite-horizon setting on compact
subsets of the domain of asymptotic null-controllability as the radius of the target ball tends to zero. The results of Sections~2
and 3 form a theoretical foundation for curse-of-dimensionality-free approaches to feedback stabilization. Some related
computational aspects are discussed in Section~4, while further development of widely applicable numerical schemes and their
software implementation is left for future works. Supporting numerical simulation results are presented and discussed in
Section~5, and Section~6 contains concluding remarks. Possible issues related to the curse of complexity are pointed out
as well. The paper is also supported by \hyperlink{Online_App}
{an appendix} including some proofs and auxiliary
considerations.

The following notation is adopted throughout the paper:
\begin{itemize}
\setlength\itemsep{0em}
\item  given integer numbers~$ j_1 $ and $ j_2 \geqslant j_1 $, we write $ i = \overline{j_1, j_2} $ instead of
$ \: i \, = \, j_1, j_1 + 1, \ldots, j_2 $;
\item  the Minkowski sum of two sets~$ \Xi_1, \Xi_2 $ in some linear space is defined as
$$
\Xi_1 + \Xi_2 \:\: \stackrel{\mathrm{def}}{=} \:\: \{ \xi_1 + \xi_2 \: \colon \: \xi_1 \in \Xi_1, \:\: \xi_2 \in \Xi_2 \},
$$
and, if $ \, \Xi_1 = \{ \xi \} \, $ is singleton, we write $ \xi + \Xi_2 $ instead of $ \, \{ \xi \} + \Xi_2 $;
\item  given $ j \in \mathbb{N} $ and $ \Xi \subseteq \mathbb{R}^j $, the interior, closure, and boundary of
$ \Xi $ are denoted by $ \, \mathrm{int} \, \Xi $, $ \, \bar{\Xi}, \, $ and $ \, \partial \Xi, \, $ respectively;
\item  given $ j \in \mathbb{N} $, the origin in $ \mathbb{R}^j $ is written as $ 0_j $, $ \| \cdot \| $
is the Euclidean norm in $ \mathbb{R}^j $ (we avoid any confusions when considering the norms of vectors of
different dimensions together), the open Euclidean ball with center~$ \xi \in \mathbb{R}^j $ and
radius~$ r > 0 $ is denoted by $ \mathrm{B}_r(\xi) $, and its closure is $ \bar{\mathrm{B}}_r(\xi) $;
\item  given $ \: j_1, j_2 \in \mathbb{N}, \: $ the zero matrix of size $ j_1 \times j_2 $ is written as $ 0_{j_1 \times j_2} $, and
the $ j_1 \times j_1 $ identity matrix is $ I_{j_1 \times j_1} $;
\item  given $ j \in \mathbb{N} $, a vector~$ \xi \in \mathbb{R}^j $ and a nonempty set~$ \Xi \subseteq \mathbb{R}^j $,
the Euclidean distance from $ \xi $ to $ \Xi $ is denoted by $ \mathrm{dist} \, (\xi, \Xi) $;
\item  given $ \: j_1, j_2 \in \mathbb{N} $, $ \, \Xi_1 \subseteq \mathbb{R}^{j_1}, \, $ and
$ \, \Xi_2 \subseteq \mathbb{R}^{j_2}, \: $ the class of all essentially bounded functions
$ \, \varphi \colon \Xi_1 \to \Xi_2 \, $ is denoted by $ \, L^{\infty}(\Xi_1, \Xi_2), \, $ while
$ \, L^{\infty}_{\mathrm{loc}}(\Xi_1, \Xi_2) \, $ is the wider class of all locally essentially bounded functions
$ \, \varphi \colon \Xi_1 \to \Xi_2 $;
\item  given a function $ \, \varphi \colon \Xi_1 \to \mathbb{R}, \, $ the set of all its minimizers on
$ \Xi \subseteq \Xi_1 $ is denoted by $ \: \mathrm{Arg} \min_{\xi \, \in \, \Xi} \, \varphi(\xi), \: $ while the criterion
for the corresponding minimization problem is written as $ \: \varphi(\xi)  \, \longrightarrow \, \inf_{\xi \, \in \, \Xi} \: $
(or $ \: \varphi(\xi)  \, \longrightarrow \, \min_{\xi \, \in \, \Xi} \: $ if the minimum exists);
\item  $ \mathcal{K} $ is the class of all strictly increasing continuous
functions~$ \: \varphi \, \colon \, [0, +\infty) \to [0, +\infty) \: $ satisfying $ \varphi(0) = 0 $;
\item  $ \mathcal{K}_{\infty} $ is the class of all functions~$ \varphi(\cdot) \in \mathcal{K} $
satisfying $ \: \lim_{\rho \, \to \, +\infty} \, \varphi(\rho) \, = \, +\infty $;
\item  $ \mathcal{L} $ is the class of all nonincreasing continuous
functions~$ \: \varphi \, \colon \, [0, +\infty) \to [0, +\infty) \: $ for which
$ \: \lim_{\rho \, \to \, +\infty} \, \varphi(\rho) \, = \, 0 $;
\item  $ \mathcal{K} \mathcal{L} $ is the class of all continuous
functions~$ \: \varphi \, \colon \, [0, +\infty)^2 \to [0, +\infty) \: $ such that
$ \, \varphi(\cdot, \rho) \in \mathcal{K} \, $ and $ \, \varphi(\rho, \cdot) \in \mathcal{L} \, $ for every
$ \rho \geqslant 0 $;
\item  if a vector variable~$ \xi $ consists of some arguments of a map $ \: \varphi \, = \, \varphi(\ldots, \xi, \ldots), \: $
then $ \mathrm{D}_{\xi} \varphi $ denotes the standard (Fr\'echet) partial derivative of $ \varphi $ with respect to $ \xi $, and
$ \mathrm{D} \varphi $ is the standard derivative with respect to the vector of all arguments (the exact definitions of
the derivatives depend on the domain and range of $ \varphi $);
\item  given a real Hilbert space~$ X $, a nonempty set $ \Xi \subseteq X $ and a point~$ \xi \in \Xi $, the proximal
normal cone to $ \Xi $ at $ \xi $ is written as $ \mathrm{N}_{\mathrm{P}}(\xi; \Xi) $, and, if $ \Xi $ is closed,
$ \mathrm{N}(\xi; \Xi) $ denotes the normal cone to $ \Xi $ at $ \xi $, which is polar to the related tangent cone
(see, e.\,g., \cite[\S 1.1, \S 2.5]{ClarkeLedyaev1998});
\item  given $ \, j \in \mathbb{N} $, $ \, \Xi \subseteq \mathbb{R}^j $, $ \, \xi \, \in \, \mathrm{int} \, \Xi $,
$ \, \zeta \in \mathbb{R}^j \, $ and $ \, \varphi \colon \Xi \to \mathbb{R}, \, $ the lower Dini derivative
(or the directional subderivate) of $ \varphi $ at the point~$ \xi $ in the direction~$ \zeta $ is written as
$ \, \partial^- \varphi(\xi; \zeta), \, $ the directional subdifferential (that is, the set of all directional
subgradients) of $ \varphi $ at $ \xi $ is denoted by $ \mathrm{D}^- \varphi(\xi) $, and
$ \mathrm{D}^-_{\mathrm{P}} \varphi(\xi) $ is the proximal subdifferential (that is, the set of all proximal
subgradients) of $ \varphi $ at $ \xi $ (see, e.\,g., \cite[\S 0.1, \S 3.4]{ClarkeLedyaev1998}).
\end{itemize}

We also use the following definitions:
\begin{itemize}
\setlength\itemsep{0em}
\item  given $ j \in \mathbb{N} $ and a set~$ \Xi \subseteq \mathbb{R}^j $ containing the origin~$ 0_j $,
a function~$ \: \varphi \colon \: \Xi \: \to \: \mathbb{R} \, \cup \, \{ +\infty \} \: $ is called positive definite if
$ \varphi(0_j) = 0 $ and $ \varphi(\xi) > 0 $ for all $ \, \xi \, \in \, \Xi \setminus \{ 0_j \} $;
\item  given $ j \in \mathbb{N} $ and a set~$ \Xi \subseteq \mathbb{R}^j $, a function
$ \: \varphi \colon \: \Xi \: \to \: \mathbb{R} \, \cup \, \{ -\infty, +\infty \} \: $ is called proper
if the preimage $ \, \varphi^{-1} (M) \subseteq \Xi \, $ of any compact set~$ M \subset \mathbb{R} $ is
also compact.
\end{itemize}

\section{Global extension of a local CLF via exit-time optimal control}

\subsection{Problem statement and preliminary considerations}

Let the state and control variables be denoted by $ x \in \mathbb{R}^{n \times 1} $ and
$ u \in \mathbb{R}^{m \times 1} $, respectively. Consider the time-invariant system
\begin{equation}
\left\{ \begin{aligned}
& \dot{x}(t) \: = \: f(x(t), u(t)), \quad t \geqslant 0, \\
& x(0) \, = \, x_0 \, \in \, G, \\
& u(\cdot) \: \in \: \mathcal{U} \: \stackrel{\mathrm{def}}{=} \: L^{\infty}_{\mathrm{loc}}([0, +\infty), \, U).
\end{aligned} \right.  \label{Eq_1}
\end{equation}

\begin{assumption}  \label{Ass_1}
The following conditions concerning {\rm (\ref{Eq_1})} hold{\rm :}
\begin{list}{\rm \arabic{count})}%
{\usecounter{count}}
\setlength\itemsep{0em}
\item  $ U \subset \mathbb{R}^m $ is compact{\rm ,} $ G \subseteq \mathbb{R}^n $ and $ G_1 \subseteq \mathbb{R}^n $ are
open domains{\rm ,} $ \bar{G} \subset G_1 ${\rm ,} and $ 0_n \in G ${\rm ;}
\item  $ \: G_1 \times U \, \ni \, (x, u) \: \longmapsto \: f(x, u) \, \in \, \mathbb{R}^n \: $ is a continuous function{\rm ;}
\item  any state trajectory of {\rm (\ref{Eq_1})} defined on an interval~$ [0, T) $ with
$ \: T \, \in \, (0, +\infty) \cup \{ +\infty \} \: $ and corresponding to $ x_0 \in G $ and $ u(\cdot) \in \mathcal{U} $
stays inside $ G $ and does not reach the boundary~$ \partial G ${\rm ,} that is{\rm ,} $ G $ is a strongly invariant domain
in the state space {\rm (}see{\rm ,} e.\,g.{\rm , \cite[Chapter~4, \S 3]{ClarkeLedyaev1998}} and note that $ G = \mathbb{R}^n $
is a trivial case{\rm );}
\item  for any $ R > 0 ${\rm ,} there exists $ C_{1, R} > 0 $ satisfying
$$
\| f(x, u) \, - \, f(x', u) \| \:\, \leqslant \:\, C_{1, R} \, \| x - x' \| \quad
\forall \: x, x' \, \in \, \bar{\mathrm{B}}_R(0_n) \, \cap \, \bar{G} \quad \forall u \in U;
$$
\item  there exist a continuously differentiable proper function~$ \, Y \colon G_1 \to [0, +\infty) \, $ and
a constant~$ C_2 > 0 $ such that
$$
\sup_{u \, \in \, U} \, \left< \mathrm{D} Y(x), \, f(x, u) \right> \:\, \leqslant \:\, C_2 \, Y(x) \quad
\forall x \in G_1.
$$
\end{list}
\end{assumption}

\begin{remark}  \label{Rem_2}  \rm
For any $ x_0 \in G $ and $ u(\cdot) \in \mathcal{U} $, let
$$
\left[ 0, \, T_{\mathrm{ext}}(x_0, u(\cdot)) \right) \: \ni \: t \:\: \longmapsto \:\: x(t; \, x_0, u(\cdot)) \: \in \: G
$$
be a solution of the Cauchy problem~(\ref{Eq_1}) defined on the maximum extendability interval with the right
endpoint $ \: T_{\mathrm{ext}}(x_0, u(\cdot)) \: \in \: (0, +\infty) \, \cup \, \{ +\infty \} $. The local existence and
uniqueness of the solutions follow from Items~1--4 of Assumption~\ref{Ass_1}, while Item~5 is included in order to guarantee
their extendability to the whole time interval~$ [0, +\infty) $. For verifying these properties, it suffices to recall
basic results on Carath\'eodory ordinary differential equations~\cite[\S 1]{Filippov1988} and to note that Item~5 is
related to the forward completeness property~\cite{AngeliSontag1999} and implies the boundedness of the reachable set
$$
\{ x(t; \, x_0, u(\cdot)) \: \colon \:\: x_0 \in X_0, \:\:\: u(\cdot) \in \mathcal{U}, \:\:\:
t \:\, \in \:\, [ 0, \: \min \, \{ T_{\mathrm{ext}}(x_0, u(\cdot)), \, T \} ) \}
$$
for any finite time~$ T \in (0, +\infty) $ and any compact set~$ X_0 \subset G $ of initial states. For example, if
Items~1--3 of Assumption~\ref{Ass_1} hold and there exists a constant~$ C_1 > 0 $ satisfying
\begin{equation}
\| f(x, u) \, - \, f(x', u) \| \:\, \leqslant \:\, C_1 \, \| x - x' \| \quad
\forall \: x, x' \, \in \, G \quad \forall u \in U,  \label{Eq_2}
\end{equation}
then Item~5 is fulfilled with $ \: Y(x) \, = \, 1 + \| x \|^2 \: $ (while Item~4 is a trivial corollary to (\ref{Eq_2})). \qed
\end{remark}

Items~1 and 2 of Assumption~\ref{Ass_1} ensure the compactness of the sets $ \: \{ f(x, u) \, \colon \, u \in U \} \: $
for all $ x \in \bar{G} $. We also need their convexity.

\begin{assumption}  \label{Ass_3}
The set $ \: \{ f(x, u) \, \colon \, u \in U \} \: $ is convex for every $ x \in \bar{G} $.
\end{assumption}

Now recall two underlying definitions (see, e.\,g., \cite{CamilliGruneWirth2008,Clarke2000}).

\begin{definition}  \label{Def_4}
The global region of asymptotic null-controllability for the system~{\rm (\ref{Eq_1})} is given by
$$
\mathcal{D}_0 \:\: \stackrel{\mathrm{def}}{=} \:\: \left\{ x_0 \in G \: \colon \:
\mbox{\rm there exists $ u(\cdot) \in \mathcal{U} $ such that} \:\:\,
\lim\limits_{t \, \to \, +\infty} \, \| x(t; \, x_0, u(\cdot)) \| \: = \: 0 \right\}.
$$
\end{definition}

\begin{definition}  \label{Def_5}
A continuous{\rm ,} proper and positive definite function $ \: V \colon \, \mathcal{D}_0 \to [0, +\infty) \: $ is called
a {\rm (}global{\rm )} control Lyapunov function {\rm (}CLF{\rm )} for the system~{\rm (\ref{Eq_1})} in the region of
asymptotic null-controllability~$ \mathcal{D}_0 $ if there exists a continuous and positive definite function
$ \: W \colon \, \mathcal{D}_0 \to [0, +\infty) \: $ such that the following infinitesimal decrease condition
{\rm (}involving lower Dini derivatives{\rm )} holds{\rm :}
\begin{equation}
\inf_{u \, \in \, U} \: \partial^{-} V(x; \, f(x, u)) \:\, \leqslant \:\, -W(x) \quad \forall x \in \mathcal{D}_0.
\label{Eq_3}
\end{equation}
\end{definition}

\begin{remark}  \label{Rem_6}  \rm
Let Items~1,2 of Assumption~\ref{Ass_1} and Assumption~\ref{Ass_3} hold. Suppose that $ E \subseteq G $ is an open domain,
$ 0_n \in E $, and $ \: \Pi_i \colon E \to \mathbb{R} $, $ \, i = 1,2, \: $ are continuous and positive definite functions.
At a state~$ x \in E $, consider the infinitesimal decrease conditions
\begin{equation}
\inf_{u \, \in \, U} \: \partial^{-} \Pi_1(x; \, f(x, u)) \:\, \leqslant \:\, -\Pi_2(x),  \label{Eq_4}
\end{equation}
\begin{equation}
\max_{u \, \in \, U} \: \{ -\left< \zeta, \, f(x, u) \right> \} \:\, \geqslant \:\, \Pi_2(x)
\quad \forall \, \zeta \, \in \, \mathrm{D}_{\mathrm{P}}^- \Pi_1(x),  \label{Eq_5}
\end{equation}
\begin{equation}
\max_{u \, \in \, U} \: \{ -\left< \zeta, \, f(x, u) \right> \} \:\, \geqslant \:\, \Pi_2(x)
\quad \forall \, \zeta \, \in \, \mathrm{D}^- \Pi_1(x)  \label{Eq_6}
\end{equation}
in the Dini, proximal and viscosity forms, respectively. If (\ref{Eq_4}) holds at a state~$ x \in E $, then
(\ref{Eq_5}) and (\ref{Eq_6}) also hold at this state (see \cite[pp.~136,~138]{ClarkeLedyaev1998}). Furthermore,
the following three statements are equivalent (see \cite[p.~27]{Clarke2000}, \cite[Chapter~3, Theorem~4.2]{ClarkeLedyaev1998}
and \cite[Theorem~9.2]{Clarke1995}):
(i) (\ref{Eq_4}) holds for all $ x \in E $;
(ii) (\ref{Eq_5}) holds for all $ x \in E $;
(iii) (\ref{Eq_6}) holds for all $ x \in E $.
Thus, the Dini, proximal and viscosity decrease conditions lead to equivalent definitions of a CLF. \qed
\end{remark}

The next assumption plays a significant role and states the existence of a function that locally satisfies
the CLF conditions and some other technical properties.

\begin{assumption}  \label{Ass_7}
The following conditions hold{\rm :}
\begin{list}{\rm \arabic{count})}%
{\usecounter{count}}
\setlength\itemsep{0em}
\item  $ \Omega \subseteq G $ is an open domain{\rm ,} and $ 0_n \in \Omega ${\rm ;}
\item  $ V_{\mathrm{loc}} \colon \, \bar{\Omega} \to [0, +\infty) \: $ is a continuous{\rm ,} proper and
positive definite function{\rm ,} whose restriction to $ \Omega $ satisfies the infinitesimal decrease condition
\begin{equation}
\inf_{u \, \in \, U} \: \partial^{-} V_{\mathrm{loc}}(x; \, f(x, u)) \:\, \leqslant \:\,
-W_{\mathrm{loc}}(x) \quad \forall x \in \Omega  \label{Eq_7}
\end{equation}
with some continuous and positive definite function $ \: W_{\mathrm{loc}} \colon \, \Omega \to [0, +\infty) ${\rm ;}
\item  $ V_{\mathrm{loc}}(\cdot) $ is locally Lipschitz continuous in $ \Omega $ {\rm (}and hence
Lipschitz continuous on any compact subset of $ \Omega $ {\rm \cite[Theorem~1.14]{Markley2004});}
\item  there exist positive constants~$ c $ and $ C_3 $ such that the set
$ \: \left\{ x \in \bar{\Omega} \, \colon \, V_{\mathrm{loc}}(x) < c \right\} \: $ is an open domain in
$ \mathbb{R}^n ${\rm ,} whose closure coincides with the set
\begin{equation}
\Omega_c \:\, \stackrel{\mathrm{def}}{=} \:\, \left\{ x \in \bar{\Omega} \, \colon \,
V_{\mathrm{loc}}(x) \leqslant c \right\}  \label{Eq_8}
\end{equation}
and fulfills the inclusion
\begin{equation}
\Omega_c \, + \, \mathrm{B}_{C_3}(0_n) \: \subseteq \: \Omega,  \label{Eq_9}
\end{equation}
while the boundary~$ \partial \Omega_c $ coincides with
\begin{equation}
l_c \:\, \stackrel{\mathrm{def}}{=} \:\, \left\{ x \in \bar{\Omega} \, \colon \,
V_{\mathrm{loc}}(x) = c \right\}  \label{Eq_10}
\end{equation}
and is a connected piecewise regular hypersurface in $ \mathbb{R}^n ${\rm ;}
\item  $ \lim_{\varepsilon \, \to \, +0} \: \sup \, \left\{ \| x \| \: \colon \: x \in \bar{\Omega}, \:\:
V_{\mathrm{loc}}(x) \leqslant \varepsilon \right\} \:\, = \:\, 0 $.
\end{list}
\end{assumption}

\begin{remark}  \label{Rem_8}  \rm
Due to Remark~\ref{Rem_6}, the condition~(\ref{Eq_7}) in Item~2 of Assumption~\ref{Ass_7} can also be
written in the proximal and viscosity forms. \qed
\end{remark}

\begin{remark}  \label{Rem_9}  \rm
Let Assumptions~\ref{Ass_1} and \ref{Ass_7} hold. Since $ V_{\mathrm{loc}}(\cdot) $ is a proper function,
$ \Omega_c $ is a compact set. If Assumption~\ref{Ass_3} also holds, then, with the help of
\cite[Theorem~2.1 and Remark~2.1]{Clarke2000}, one can establish the inclusion $ \Omega_c \subseteq \mathcal{D}_0 $,
which yields local asymptotic null-controllability for the system~(\ref{Eq_1}). \qed
\end{remark}

\begin{remark}  \label{Rem_10}  \rm
Let Assumption~\ref{Ass_1} and Items~1,2 of Assumption~\ref{Ass_7} hold. It is easy to verify that a sufficient condition for
Item~5 of Assumption~\ref{Ass_7} is the existence of a function~$ \alpha(\cdot) \in \mathcal{K} $ satisfying
$ \, \alpha(\| x \|) \leqslant V_{\mathrm{loc}}(x) \, $ for all $ x \in \bar{\Omega} $. \qed
\end{remark}

The following proposition indicates that, under the adopted assumptions, the right-hand side of the system~(\ref{Eq_1})
satisfies the Petrov condition on $ l_c = \partial \Omega_c $ in the sense of \cite[Definition~8.2.2]{Cannarsa2004}.
This condition strengthens the property that, at any state~$ x \in l_c $, there exists a velocity of (\ref{Eq_1})
pointing strictly inside $ \Omega_c $.

\begin{proposition}  \label{Pro_11}
Let Assumptions~{\rm \ref{Ass_1}, \ref{Ass_3}} and {\rm \ref{Ass_7}} hold. There exists a constant~$ C_4 > 0 $ satisfying
\begin{equation}
\min_{u \, \in \, U} \: \left< \nu, \, f(x, u) \right> \: \leqslant \: -C_4 \quad
\forall \: \nu \: \in \: \{ \nu' \, \in \, \mathrm{N}_{\mathrm{P}}(x; \Omega_c) \: \colon \: \| \nu' \| = 1 \} \quad
\forall x \in l_c,  \label{Eq_11}
\end{equation}
that is{\rm ,} the Petrov condition holds for the right-hand side of {\rm (\ref{Eq_1})} on $ l_c $.
\end{proposition}

The proof of Proposition~\ref{Pro_11} requires two auxiliary results from nonsmooth analysis. The proof of the first of them
(Lemma~\ref{Thm_A_1}) is rather straightforward and given in Subsection~A.1.1 of \hyperlink{Online_App}
{the appendix},
while the proof of the second result (Lemma~\ref{Thm_A_2}) is essentially more difficult and can be found in \cite{Schirotzek2007}.

\begin{lemma}  \label{Thm_A_1}
If $ E \subseteq \mathbb{R}^n $ is an open set and a function $ \: \varphi \, \colon \, E \to \mathbb{R} \: $
is Lipschitz continuous with constant~$ C > 0 ${\rm ,} then
$$
\| \zeta \| \, \leqslant \, C \sqrt{n} \quad \forall \, \zeta \, \in \, \mathrm{D}^-_{\mathrm{P}} \varphi(x) \quad
\forall x \in E.
$$
\end{lemma}

\begin{lemma}{\rm \cite[Theorem~11.6.3]{Schirotzek2007}}  \label{Thm_A_2}
Assume that $ X $ is a real Hilbert space{\rm ,}
$ \: \varphi \colon \: X \: \to \: \mathbb{R} \, \cup \, \{ -\infty, +\infty \} \: $ is a proper and lower semicontinuous
function{\rm ,} $ \: M \: \stackrel{\mathrm{def}}{=} \: \{ \xi \in X \, \colon \, \varphi(\xi) \leqslant 0 \} ${\rm ,}
$ \: x \in M ${\rm ,} and $ \, \nu \in \mathrm{N}_{\mathrm{P}}(x; M) $. Then at least one of the following two
properties holds{\rm :}
\begin{list}{\rm \arabic{count})}%
{\usecounter{count}}
\setlength\itemsep{0em}
\item  for any $ \varepsilon > 0 ${\rm ,} there exist $ x' \in X $ and
$ \, \zeta' \in \mathrm{D}^-_{\mathrm{P}} \varphi(x') \, $ such that
$$
\| x' - x \| \: < \: \varepsilon, \quad \| \varphi(x') \, - \, \varphi(x) \| \: < \: \varepsilon, \quad
\| \zeta' \| \, < \, \varepsilon;
$$
\item  for any $ \varepsilon > 0 ${\rm ,} there exist $ x' \in X ${\rm ,}
$ \, \zeta' \in \mathrm{D}^-_{\mathrm{P}} \varphi(x') \, $ and $ \lambda > 0 $ such that
$$
\| x' - x \| \: < \: \varepsilon, \quad \| \varphi(x') \, - \, \varphi(x) \| \: < \: \varepsilon, \quad
\| \nu \, - \, \lambda \zeta' \| \: < \: \varepsilon.
$$
\end{list}
\end{lemma}

\begin{proof}[Proof of Proposition~{\rm \ref{Pro_11}}]
Since $ l_c $ is compact and $ V_{\mathrm{loc}}(\cdot) $, $ W_{\mathrm{loc}}(\cdot) $ are continuous and positive
definite, there exist constants $ \: \eta_1 > 0 $, $ \, \eta_2 \in (0, C_3) $, $ \, \eta_3 > 0 \: $ such that
$ \: \mathrm{B}_{\eta_1}(0_n) \, \subseteq \, \mathrm{int} \, \Omega_c \: $ and
\begin{equation}
W_{\mathrm{loc}}(x) \, \geqslant \, \eta_3 \quad \forall \: x \: \in \: l_c \, + \, \mathrm{B}_{\eta_2}(0_n).
\label{Eq_12}
\end{equation}
In line with Remark~\ref{Rem_8}, the infinitesimal decrease condition on $ V_{\mathrm{loc}}(\cdot) $ can be written
in the proximal form:
\begin{equation}
\min_{u \, \in \, U} \: \left< \zeta, \, f(x, u) \right> \:\, \leqslant \:\, -W_{\mathrm{loc}}(x) \quad
\forall \zeta \, \in \, \mathrm{D}_{\mathrm{P}}^- V_{\mathrm{loc}}(x) \quad \forall x \in \Omega.  \label{Eq_13}
\end{equation}
From the relations~(\ref{Eq_12}), (\ref{Eq_13}), (\ref{Eq_9}) and $ \eta_2 \in (0, C_3) $, one obtains
\begin{equation}
\min_{u \, \in \, U} \: \left< \zeta, \, f(x, u) \right> \:\, \leqslant \:\, -\eta_3 \quad
\forall \zeta \, \in \, \mathrm{D}_{\mathrm{P}}^- V_{\mathrm{loc}}(x) \quad
\forall \: x \: \in \: l_c \, + \, \mathrm{B}_{\eta_2}(0_n).  \label{Eq_14}
\end{equation}
The property~(\ref{Eq_14}), continuity of $ f(\cdot, \cdot) $, and compactness of $ U $ and
$ \, l_c \, + \, \bar{\mathrm{B}}_{\eta_2}(0_n) \, $ yield the existence of a constant~$ \eta_4 > 0 $ satisfying
\begin{equation}
\| \zeta \| \, \geqslant \, \eta_4 \quad \forall \zeta \, \in \, \mathrm{D}_{\mathrm{P}}^- V_{\mathrm{loc}}(x) \quad
\forall \: x \: \in \: l_c \, + \, \mathrm{B}_{\eta_2}(0_n).  \label{Eq_15}
\end{equation}
Moreover, the Lipschitz continuity of $ V_{\mathrm{loc}}(\cdot) $ on compact subsets of $ \Omega $ and Lemma~\ref{Thm_A_1}
guarantee the existence of a constant~$ \eta_5 > 0 $ such that
\begin{equation}
\| \zeta \| \, \leqslant \, \eta_5 \quad \forall \zeta \, \in \, \mathrm{D}_{\mathrm{P}}^- V_{\mathrm{loc}}(x) \quad
\forall \: x \: \in \: l_c \, + \, \mathrm{B}_{\eta_2}(0_n).  \label{Eq_16}
\end{equation}

Now let us apply Lemma~\ref{Thm_A_2} to the zero sublevel set~$ \Omega_c $ of the proper and lower semicontinuous
function that equals $ \, V_{\mathrm{loc}}(x) - c \, $ for $ x \in \bar{\Omega} $ and $ +\infty $ for
$ \, x \, \in \, \mathbb{R}^n \setminus \bar{\Omega} $.

Take $ x \in l_c $ and $ \, \nu \in \mathrm{N}_{\mathrm{P}}(x; \Omega_c) \, $ with $ \| \nu \| = 1 $.

By virtue of (\ref{Eq_15}), Item~1 of Lemma~\ref{Thm_A_2} does not hold in the considered situation. Then Item~2 of
Lemma~\ref{Thm_A_2} holds and implies that, for any $ \varepsilon > 0 $, there exist $ \, x' \in \mathrm{B}_{\varepsilon}(x) $,
$ \, \zeta' \in \mathrm{D}_{\mathrm{P}}^- V_{\mathrm{loc}}(x') \, $ and $ \lambda > 0 $ satisfying
\begin{equation}
\left\| \nu \, - \, \lambda \zeta' \right\| \: < \: \varepsilon.  \label{Eq_17}
\end{equation}
By assuming $ \varepsilon \in (0, 1) $ without loss of generality, and by using (\ref{Eq_17}) with $ \| \nu \| = 1 $, it is
easy to derive $ \: \left| \lambda \, \| \zeta' \| \, - \, 1 \right| \: < \: \varepsilon $, $ \: \| \zeta' \| > 0, \: $ and
therefore
\begin{equation}
\left\| \lambda \zeta' \, - \, \frac{\zeta'}{\| \zeta' \|} \right\| \:\, = \:\,
\left| \lambda \, - \, \frac{1}{\| \zeta' \|} \right| \: \left\| \zeta' \right\| \:\, = \:\,
\left| \lambda \, \| \zeta' \| \, - \, 1 \right| \:\, < \:\, \varepsilon.  \label{Eq_18}
\end{equation}
The inequalities (\ref{Eq_17}) and (\ref{Eq_18}) lead to
\begin{equation}
\left\| \nu \, - \, \frac{\zeta'}{\| \zeta' \|} \right\| \: < \: 2 \varepsilon.  \label{Eq_19}
\end{equation}

Thus, for any $ \varepsilon > 0 $, there exist $ \, x' \in \mathrm{B}_{\varepsilon}(x) \, $ and
$ \, \zeta' \in \mathrm{D}_{\mathrm{P}}^- V_{\mathrm{loc}}(x') \, $ such that (\ref{Eq_19}) holds. Together with
the relations~(\ref{Eq_14}), (\ref{Eq_16}) and continuity of the function
$ \: \mathbb{R}^n \times G \, \ni \, (\xi_1, \xi_2) \,\, \longmapsto \,\,
\min_{u \, \in \, U} \, \left< \xi_1, \, f(\xi_2, u) \right>, \: $
this ensures that, for any $ \varepsilon > 0 $, there exist $ x' \in \mathrm{B}_{\eta_2}(x) $ and
$ \, \zeta' \in \mathrm{D}_{\mathrm{P}}^- V_{\mathrm{loc}}(x') \, $ for which
$$
\min_{u \, \in \, U} \: \left< \nu, \, f(x, u) \right> \:\: \leqslant \:\:
\min_{u \, \in \, U} \: \left< \frac{\zeta'}{\| \zeta' \|}, \, f(x', u) \right> \:\, + \:\, \varepsilon \:\:
\leqslant \:\: -\frac{\eta_3}{\| \zeta' \|} \: + \: \varepsilon \:\: \leqslant \:\:
-\frac{\eta_3}{\eta_5} \: + \: \varepsilon.
$$
Since $ \varepsilon > 0 $ can be taken arbitrarily small, the Petrov condition~(\ref{Eq_11}) holds with
$ \, C_4 = \eta_3 / \eta_5 $.
\end{proof}

Other important properties are the opennes, connectedness and weak invariance of the region of asymptotic
null-controllability (recall Definition~\ref{Def_4}).

\begin{proposition}  \label{Pro_12}
Under Assumptions~{\rm \ref{Ass_1}, \ref{Ass_3}} and {\rm \ref{Ass_7},} $ \mathcal{D}_0 $ is an open domain
{\rm (}that is{\rm ,} an open connected set{\rm )} containing $ \Omega_c ${\rm ,} and it is weakly invariant
in the sense that{\rm ,} for any $ x_0 \in \mathcal{D}_0 ${\rm ,} there exists $ \, u(\cdot; x_0) \in \mathcal{U} \, $
satisfying $ \: x(t; \, x_0, \, u(\cdot; x_0)) \, \in \, \mathcal{D}_0 \: $ for all $ t \geqslant 0 $. 
\end{proposition}

\begin{proof}
The inclusion $ \Omega_c \subseteq \mathcal{D}_0 $ was justified in Remark~\ref{Rem_9}. The connectedness and weak
invariance of $ \mathcal{D}_0 $ can be established similarly to \cite[Proposition~2.3, (ii)]{CamilliGruneWirth2008}.
By using Proposition~\ref{Pro_11}, \cite[Remark~8.1.6]{Cannarsa2004}, and the reasonings in
\cite[the proofs of Theorems~8.2.1 and 8.2.3]{Cannarsa2004}, one can show that $ \mathcal{D}_0 $ is open.
In \cite[Chapter~8]{Cannarsa2004}, the global Lipschitz condition is imposed on $ f(\cdot, u) $ uniformly with respect to
$ u \in U $, but it can in fact be relaxed to Items~4 and 5 of Assumption~\ref{Ass_1} when verifying the openness of
$ \mathcal{D}_0 $.
\end{proof}

Next, let us adopt the convention $ \, \inf \, \emptyset = +\infty \, $ and introduce the minimum times of reaching
$ \Omega_c $:
\begin{equation}
T_{\Omega_c}(x_0, u(\cdot)) \:\: \stackrel{\mathrm{def}}{=} \:\: \inf \:
\{ T \in [0, +\infty) \: \colon \: x(T; \, x_0, u(\cdot)) \: \in \: \Omega_c \} \quad
\forall x_0 \in G \quad \forall u(\cdot) \in \mathcal{U}.  \label{Eq_20}
\end{equation}

\begin{figure}
\begin{center}
\includegraphics[width=5.7cm,height=4cm]{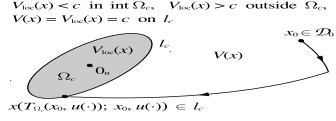}
\end{center}
\bf \caption{\rm The exit-time optimal control problem~{\rm (\ref{Eq_21}),} whose target set is a level set of a local CLF.}
\label{Fig_1}
\end{figure}

A key point of this section is to represent a sought-after CLF outside the sublevel set~$ \Omega_c $ as the value function
in an exit-time optimal control problem, stated with respect to the target set~$ l_c $ and the constant terminal cost
$ \, V_{\mathrm{loc}}(x) = c \, $ for $ x \in l_c $ (see Fig.~\ref{Fig_1}):
\begin{equation}
V(x_0) \:\: \stackrel{\mathrm{def}}{=} \:\: \inf_{\substack{u(\cdot) \, \in \, \mathcal{U} \: \colon \\
T_{\Omega_c}(x_0, \, u(\cdot)) \: < \: +\infty}} \: \left\{ \int\limits_0^{T_{\Omega_c}(x_0, \, u(\cdot))}
g(x(t; \, x_0, u(\cdot)), \: u(t)) \: \mathrm{d} t \:\, + \:\, c \right\} \quad
\forall \, x_0 \, \in \, G \setminus \Omega_c.  \label{Eq_21}
\end{equation}

\begin{assumption}  \label{Ass_13}
The following conditions concerning the running cost~$ g(\cdot, \cdot) $ hold{\rm :}
\begin{list}{\rm \arabic{count})}%
{\usecounter{count}}
\setlength\itemsep{0em}
\item  $ \bar{G} \times U \, \ni \, (x, u) \,\, \longmapsto \,\, g(x, u) \: \in \: [0, +\infty) \: $ is
a nonnegative continuous function{\rm ;}
\item  for any $ R > 0 ${\rm ,} there exists $ C_{5, R} > 0 $ such that
\begin{equation}
|g(x, u) \, - \, g(x', u)| \:\, \leqslant \:\, C_{5, R} \, \| x - x' \| \quad
\forall \: x, x' \, \in \, \bar{\mathrm{B}}_R(0_n) \, \cap \, \bar{G} \quad \forall u \in U;  \label{Eq_68}
\end{equation}
\item  $ g(x, u) > 0 $ for all $ \, x \, \in \, G \setminus \{ 0_n \} \, $ and $ u \in U ${\rm ;}
\item  $ C_6 \:\, \stackrel{\mathrm{def}}{=} \:\,
\inf \: \{ g(x, u) \: \colon \: x \, \in \, G \, \setminus \, \mathrm{int} \, \Omega_c, \:\: u \in U \} \:\, > \:\, 0 $.
\end{list}
\end{assumption}

\begin{proposition}  \label{Pro_14}
Under Assumptions~{\rm \ref{Ass_1}, \ref{Ass_3}, \ref{Ass_7}} and {\rm \ref{Ass_13},} the following relations hold
for the value function~{\rm (\ref{Eq_21}):}
\begin{equation}
V(x_0) \, = \, +\infty \quad \forall \, x_0 \, \in \, G \setminus \mathcal{D}_0,  \label{Eq_22}
\end{equation}
\begin{equation}
V(x_0) \, > \, c \quad \forall \, x_0 \, \in \, \mathcal{D}_0 \setminus \Omega_c.  \label{Eq_23}
\end{equation}
\end{proposition}

\begin{proof}
The property~(\ref{Eq_22}) is clear due to the definition~(\ref{Eq_21}) and Proposition~\ref{Pro_12}.
For establishing (\ref{Eq_23}), let us take $ \, x_0 \, \in \, \mathcal{D}_0 \setminus \Omega_c \, $ and show that
\begin{equation}
\mathcal{T}_{\Omega_c}(x_0) \:\, \stackrel{\mathrm{def}}{=} \:\,
\inf_{u(\cdot) \, \in \, \mathcal{U}} \: T_{\Omega_c}(x_0, u(\cdot)) \:\, > \:\, 0.  \label{Eq_24}
\end{equation}
Assume $ \, \mathcal{T}_{\Omega_c}(x_0) = 0 $. Then there exist a number~$ T > 0 $ and
a sequence $ \, \left\{ u^{(k)}(\cdot) \right\}_{k = 1}^{\infty} \subset \mathcal{U} \, $ such that
$ \: T^{(k)} \: \stackrel{\mathrm{def}}{=} \: T_{\Omega_c} \left( x_0, \, u^{(k)}(\cdot) \right) \: \leqslant \: T \: $
for all $ k \in \mathbb{N} $ and $ \: \lim_{k \, \to \, \infty} \, T^{(k)} \, = \, 0 $. According to Remark~\ref{Rem_2},
the reachable set
$$
X_T(x_0) \:\: \stackrel{\mathrm{def}}{=} \:\:
\{ x(t; \, x_0, u(\cdot)) \: \colon \: t \in [0, T], \:\: u(\cdot) \in \mathcal{U} \} \:\: \subseteq \:\: G
$$
is bounded. Since $ f(\cdot, \cdot) $ is continuous and $ U $ is compact, one has
$$
M_T(x_0) \:\, \stackrel{\mathrm{def}}{=} \:\,
\max_{x \, \in \, \overline{X_T(x_0)}, \:\: u \, \in \, U} \: \| f(x, u) \| \:\, < \:\, +\infty.
$$
Hence,
$$
\arraycolsep=1.5pt
\def\arraystretch{1.5}
\begin{array}{c}
x \left( T^{(k)}; \: x_0, \, u^{(k)}(\cdot) \right) \:\, \in \:\, \Omega_c \quad \forall k \in \mathbb{N}, \\
0 \:\: \leqslant \:\: \lim\limits_{k \, \to \, \infty} \:
\left\| x \left( T^{(k)}; \: x_0, \, u^{(k)}(\cdot) \right) \,\, - \,\, x_0 \right\| \:\: \leqslant \:\:
M_T(x_0) \: \lim\limits_{k \, \to \, \infty} \, T^{(k)} \:\, = \:\, 0,
\end{array}
$$
which contradicts with $ x_0 \notin \Omega_c $. This implies (\ref{Eq_24}). From (\ref{Eq_21}), (\ref{Eq_24}) and Item~4 of
Assumption~\ref{Ass_13}, one obtains $ \: V(x_0) - c \: \geqslant \: C_6 \, \mathcal{T}_{\Omega_c}(x_0) \: > \: 0 $.
\end{proof}

It is reasonable to extend the function~(\ref{Eq_21}) to $ \Omega_c $ by
\begin{equation}
V(x_0) \, \stackrel{\mathrm{def}}{=} \, V_{\mathrm{loc}}(x_0) \quad \forall x_0 \in \Omega_c  \label{Eq_25}
\end{equation}
(see Fig.~\ref{Fig_1}).

\begin{proposition}  \label{Pro_15}
Let Assumptions~{\rm \ref{Ass_1}, \ref{Ass_3}, \ref{Ass_7}} and {\rm \ref{Ass_13}} hold{\rm ,} and consider
the function~$ V(\cdot) $ defined by {\rm (\ref{Eq_21})} and {\rm (\ref{Eq_25})}. The following properties
hold{\rm :}
$$
\arraycolsep=1.5pt
\def\arraystretch{1.3}
\begin{array}{c}
V(x_0) \, < \, +\infty \quad \forall x_0 \in \mathcal{D}_0, \\
V(x_0) \, = \, +\infty \quad \forall \, x_0 \, \in \, G \setminus \mathcal{D}_0, \\
V(x_0) \, = \, V_{\mathrm{loc}}(x_0) \, < \, c \quad \forall \, x_0 \, \in \, \mathrm{int} \, \Omega_c, \\
V(x_0) \, = \, V_{\mathrm{loc}}(x_0) \, = \, c \quad \forall \, x_0 \, \in \, l_c = \partial \Omega_c, \\
V(x_0) \, > \, c \quad \forall \, x_0 \, \in \, G \setminus \Omega_c, \\
V_{\mathrm{loc}}(x_0) \, > \, c \quad \forall \, x_0 \, \in \, \Omega \setminus \Omega_c.
\end{array}
$$
\end{proposition}

\begin{proof}
These relations can be directly obtained by using Definition~\ref{Def_4}, Assumption~\ref{Ass_7}, Remark~\ref{Rem_9}, and
Proposition~\ref{Pro_14}.
\end{proof}

One more technical assumption will be required below.

\begin{assumption}  \label{Ass_16}
There exist positive constants~$ C_7, C_8 $ such that
$$
g(x, u) \: \geqslant \: C_8 \, \| f(x, u) \| \quad \forall \, x \, \in \, G \setminus \mathrm{B}_{C_7}(0_n) \quad
\forall u \in U.
$$
\end{assumption}

\subsection{Main result}

The main result of this section (Theorem~\ref{Thm_18}) indicates that, under the adopted assumptions, the concatenation
of the local CLF in $ \Omega_c $ with the value function for the exit-time optimal control problem~(\ref{Eq_21}) is
a global CLF in the whole domain of asymptotic null-controllability. Before verifying the main result, let us establish
some auxiliary properties.

\begin{proposition}  \label{Pro_17}
Under Assumptions~{\rm \ref{Ass_1}, \ref{Ass_3}, \ref{Ass_7}, \ref{Ass_13}} and {\rm \ref{Ass_16},} the following properties
hold for the function~$ V(\cdot) $ defined by {\rm (\ref{Eq_21})} and {\rm (\ref{Eq_25}):}
\begin{list}{\rm \arabic{count})}%
{\usecounter{count}}
\setlength\itemsep{0em}
\item  $ V(\cdot) $ is locally Lipschitz continuous in $ \mathcal{D}_0 ${\rm ;}
\item  the restriction of $ V(\cdot) $ to $ \mathcal{D}_0 \setminus \Omega_c $ solves the HJB equation
$$
\max_{u \, \in \, U} \: \{ -\left< \mathrm{D} V(x), \, f(x, u) \right> \: - \: g(x, u) \} \:\: = \:\: 0, \quad
x \, \in \, \mathcal{D}_0 \setminus \Omega_c,
$$
in the viscosity sense{\rm ;}
\item  for any sequence $ \: \left\{ x^{(k)} \right\}_{k = 1}^{\infty} \, \subset \, \mathcal{D}_0 \: $
satisfying either $ \: \lim_{k \, \to \, \infty} \, x^{(k)} \, = \, x' \, \in \, \partial \mathcal{D}_0 \: $ or
$ \: \lim_{k \, \to \, \infty} \, \left\| x^{(k)} \right\| \, = \, +\infty, \: $ one has
$ \: \lim_{k \, \to \, \infty} \, V \left( x^{(k)} \right) \, = \, +\infty $.
\end{list}
\end{proposition}

\begin{proof}
For verifying Items~1,\,2, as well as Item~3 for
$ \: \lim_{k \, \to \, \infty} \, x^{(k)} \, = \, x' \, \in \, \partial \mathcal{D}_0, \: $ it suffices to use
\cite[Remark~8.1.6]{Cannarsa2004} and the reasonings in \cite[the proofs of Theorem~8.2.5, Theorem~8.1.8 and
Proposition~8.2.6]{Cannarsa2004}. As in the proof of Proposition~\ref{Pro_12}, Items~4 and 5 of Assumption~\ref{Ass_1}
replace the requirement that $ f(\cdot, u) $ should satisfy the global Lipschitz condition uniformly with respect to
$ u \in U $.

It remains to prove Item~3 in case $ \: \lim_{k \, \to \, \infty} \, \left\| x^{(k)} \right\| \, = \, +\infty $.
Consider such a sequence $ \: \left\{ x^{(k)} \right\}_{k = 1}^{\infty} \, \subset \, \mathcal{D}_0 $.
In line with Assumption~\ref{Ass_16} and the compactness of $ \Omega_c $, there exists a constant~$ C'_7 \geqslant C_7 $
satisfying
$$
\arraycolsep=1.5pt
\def\arraystretch{1.3}
\begin{array}{c}
\Omega_c \, \subseteq \, \mathrm{B}_{C'_7}(0_n), \\
g(x, u) \: \geqslant \: C_8 \, \| f(x, u) \| \quad \forall \, x \, \in \, G \setminus \mathrm{B}_{C'_7}(0_n) \quad
\forall u \in U.
\end{array}
$$
Denote
$$
T_k(u(\cdot)) \:\: \stackrel{\mathrm{def}}{=} \:\: \inf \: \left\{ T \in [0, +\infty) \: \colon \:\,
x \left( T; \: x^{(k)}, \, u(\cdot) \right) \: \in \: \bar{\mathrm{B}}_{C'_7}(0_n) \right\} \quad
\forall u(\cdot) \in \mathcal{U} \quad \forall k \in \mathbb{N}.
$$
Then one has
$$
\begin{aligned}
V \left( x^{(k)} \right) \:\: > \:\: & \inf_{\substack{u(\cdot) \, \in \, \mathcal{U} \: \colon \\
T_k(u(\cdot)) \: < \: +\infty}} \: \left\{ \int\limits_0^{T_k(u(\cdot))}
g \left( x \left( t; \: x^{(k)}, \, u(\cdot) \right), \: u(t) \right) \: \mathrm{d} t \right\}  \\
\geqslant \:\: & C_8 \: \inf_{\substack{u(\cdot) \, \in \, \mathcal{U} \: \colon \\
T_k(u(\cdot)) \: < \: +\infty}} \: \left\{ \int\limits_0^{T_k(u(\cdot))}
\left\| f \left( x \left( t; \: x^{(k)}, \, u(\cdot) \right), \: u(t) \right) \right\| \: \mathrm{d} t \right\} \\
\geqslant \:\: & C_8 \: \inf_{\substack{u(\cdot) \, \in \, \mathcal{U} \: \colon \\
T_k(u(\cdot)) \: < \: +\infty}} \: \left\| x \left( T_k(u(\cdot)); \: x^{(k)}, \, u(\cdot) \right) \: - \:
x^{(k)} \right\| \\
\geqslant \:\: & C_8 \: \left( \left\| x^{(k)} \right\| \, - \, C'_7 \right)
\end{aligned}
$$
for all $ k \in \mathbb{N} $. Together with $ \: \lim_{k \, \to \, \infty} \, \left\| x^{(k)} \right\| \, = \, +\infty, \: $
this leads to $ \: \lim_{k \, \to \, \infty} \, V \left( x^{(k)} \right) \, = \, +\infty $.
\end{proof}

\begin{theorem}  \label{Thm_18}
Let Assumptions~{\rm \ref{Ass_1}, \ref{Ass_3}, \ref{Ass_7}, \ref{Ass_13}} and {\rm \ref{Ass_16}} hold. The function~$ V(\cdot) $
defined by {\rm (\ref{Eq_21})} and {\rm (\ref{Eq_25})} is a CLF for the system~{\rm (\ref{Eq_1})} in $ \mathcal{D}_0 ${\rm ,}
i.\,e.{\rm ,} the restriction of this function to $ \mathcal{D}_0 $ is a continuous{\rm ,} proper{\rm ,} positive definite and
such that the infinitesimal decrease condition
\begin{equation}
\inf_{u \, \in \, U} \: \partial^- V(x; \, f(x, u)) \:\, \leqslant \:\, -W(x) \quad \forall x \in \mathcal{D}_0
\label{Eq_26}
\end{equation}
holds with some continuous and positive definite function~$ \: W \, \colon \, \mathcal{D}_0 \to [0, +\infty) $.
Furthermore{\rm ,} $ V(\cdot) $ is locally Lipschitz continuous in $ \mathcal{D}_0 $ and therefore differentiable almost
everywhere in $ \mathcal{D}_0 $ {\rm (}with respect to the Lebesgue measure in $ \mathbb{R}^n ${\rm )}.
\end{theorem}

\begin{proof}
In line with Item~1 of Proposition~\ref{Pro_17}, $ V(\cdot) $ is locally Lipschitz continuous in $ \mathcal{D}_0 $,
and it is differentiable almost everywhere in $ \mathcal{D}_0 $ due to Rademacher's theorem. The positive definiteness of
$ V(\cdot) $ directly follows from Proposition~\ref{Pro_14} and the positive definiteness of $ V_{\mathrm{loc}}(\cdot) $.

Let us show that $ V(\cdot) $ is proper. According to the relation~(\ref{Eq_22}) and Item~3 of Proposition~\ref{Pro_17},
it suffices to verify the properness of the restriction of $ V(\cdot) $ to $ \mathcal{D}_0 $. The continuity of the latter
implies that the preimages of closed sets are closed. Again due to Item~3 of Proposition~\ref{Pro_17}, the considered
restriction is also such that the preimages of bounded sets are bounded. One consequently obtains the compactness of
the preimages of compact sets, which means properness.

It remains to establish the infinitesimal decrease condition~(\ref{Eq_26}) with an appropriate function~$ W(\cdot) $.

Since the function~$ \: W_{\mathrm{loc}} \, \colon \, \Omega \to [0, +\infty) \: $ (introduced in Item~2 of
Assumption~\ref{Ass_7}) is continuous and the set~$ \Omega_c \subset \Omega $ is compact, Tietze's extension theorem
(see, e.\,g., \cite[Theorem~5.2.1]{Krantz1999}) ensures the existence of a continuous function
$ \: W_1 \colon \, \mathbb{R}^n \to \mathbb{R} \: $ satisfying $ \, W_1(x) = W_{\mathrm{loc}}(x) \, $ for all
$ x \in \Omega_c $. Bearing in mind also the positive definiteness of $ W_{\mathrm{loc}}(\cdot) $ and the compactness of
the boundary~$ l_c = \partial \Omega_c $ that does not contain $ 0_n $, one concludes
$ \: \min_{x \, \in \, l_c} \, W_{\mathrm{loc}}(x) \: > \: 0 $. Hence, the function
$$
W_2(x) \:\: \stackrel{\mathrm{def}}{=} \:\: \begin{cases}
W_{\mathrm{loc}}(x), & x \in \Omega_c, \\
\max \, \left\{ W_1(x), \,\, \min\limits_{\xi \, \in \, l_c} \, W_{\mathrm{loc}}(\xi) \right\}, &
x \, \in \, \mathbb{R}^n \setminus \Omega_c,
\end{cases}
$$
is continuous and positive definite. Now take
\begin{equation}
W(x) \:\, \stackrel{\mathrm{def}}{=} \:\, \min \, \left\{ W_2(x), \,\, \min_{u \, \in \, U} \, g(x, u) \right\} \quad
\forall x \in \mathcal{D}_0.  \label{Eq_27}
\end{equation}
The compactness of $ U $ and Items~1,\,3 of Assumption~\ref{Ass_13} yield that the function
$ \: G \, \ni \, x \: \longmapsto \: \min_{u \, \in \, U} \, g(x, u) \: $ is continuous everywhere in $ G $ and
positive for all $ \, x \, \in \, G \setminus \{ 0_n \} $. Thus, (\ref{Eq_27}) is a continuous and positive definite
function.

In order to establish the condition~(\ref{Eq_26}) with the selected $ W(\cdot) $, it suffices to verify this Dini form
for $ x \in \Omega_c $ and the related viscosity form for $ \, x \, \in \, \mathcal{D}_0 \setminus \Omega_c \, $
(recall Remark~\ref{Rem_6}).

For $ \, x \, \in \, \mathrm{int} \, \Omega_c, \, $ the inequality in (\ref{Eq_26}) holds due to Assumption~\ref{Ass_7}.
For $ \, x \, \in \, \mathcal{D}_0 \setminus \Omega_c, \, $ Item~2 of Proposition~\ref{Pro_17} implies the viscosity form
of the infinitesimal decrease condition:
$$
\begin{aligned}
& \max_{u \, \in \, U} \: \{ -\left< \zeta, \, f(x, u) \right> \} \:\, - \:\, W(x) \\
& \qquad
\geqslant \:\: \max_{u \, \in \, U} \: \{ -\left< \zeta, \, f(x, u) \right> \} \:\, - \:\,
\min_{u \, \in \, U} \, g(x, u) \\
& \qquad
\geqslant \:\: \max_{u \, \in \, U} \: \{ -\left< \zeta, \, f(x, u) \right> \: - \: g(x, u) \} \:\: \geqslant \:\: 0 \\
& \forall \, \zeta \, \in \, \mathrm{D}^- V(x).
\end{aligned}
$$
It therefore remains to prove the inequality in (\ref{Eq_26}) for $ \, x \, \in \, l_c = \partial \Omega_c $.

Let $ x \in l_c $. Due to the local Lipschitz continuity of $ V(\cdot) $ and $ V_{\mathrm{loc}}(\cdot) $ in
$ \mathcal{D}_0 $ and $ \Omega $, respectively, the following representations for the lower Dini derivatives hold
(see, e.\,g., \cite[Remark~3.1.4]{Cannarsa2004}):
\begin{equation}
\begin{aligned}
& \partial^- V(x; \zeta) \:\, = \:\, \liminf_{\lambda \, \to \, +0} \:
\frac{V(x + \lambda \zeta) \, - \, V(x)}{\lambda} \, , \\
& \partial^- V_{\mathrm{loc}}(x; \zeta) \:\, = \:\, \liminf_{\lambda \, \to \, +0} \:
\frac{V_{\mathrm{loc}}(x + \lambda \zeta) \, - \, V_{\mathrm{loc}}(x)}{\lambda} \\
& \forall \zeta \in \mathbb{R}^n.
\end{aligned}  \label{Eq_28}
\end{equation}
Introduce the control subset
$$
\begin{aligned}
U_x \:\: \stackrel{\mathrm{def}}{=} \:\: \left\{ u \in U \:\, \colon \:\, \mbox{there exists a sequence
$ \, \{ \lambda_k \}_{k = 1}^{\infty} \subset (0, +\infty) \, $ such that} {}^{{}^{{}^{{}^{}}}} \right. & \\
\left. \mbox{$ \lim\limits_{k \, \to \, \infty} \, \lambda_k \, = \, 0 \, $ and
$ \: x \, + \, \lambda_k \, f(x, u) \: \in \: \Omega_c \: $ for all $ k \in \mathbb{N} $} \right\}, &
\end{aligned}
$$
which is nonempty by virtue of Proposition~\ref{Pro_11}. With the help of Proposition~\ref{Pro_15} and
the property~(\ref{Eq_28}), one obtains
$$
\arraycolsep=1.5pt
\def\arraystretch{1.5}
\begin{array}{c}
\inf\limits_{u \, \in \, U_x} \: \partial^- V(x; \, f(x, u)) \:\, \leqslant \:\, 0, \quad
\inf\limits_{u \: \in \: U \, \setminus \, U_x} \: \partial^- V(x; \, f(x, u)) \:\, \geqslant \:\, 0, \\
\inf\limits_{u \, \in \, U_x} \: \partial^- V_{\mathrm{loc}}(x; \, f(x, u)) \:\, \leqslant \:\, 0, \quad
\inf\limits_{u \: \in \: U \, \setminus \, U_x} \: \partial^- V_{\mathrm{loc}}(x; \, f(x, u)) \:\, \geqslant \:\, 0,
\end{array}
$$
and
$$
\begin{aligned}
& \inf_{u \, \in \, U} \: \partial^- V(x; \, f(x, u)) \:\, = \:\,
\inf_{u \, \in \, U_x} \: \partial^- V(x; \, f(x, u)) \\
& \qquad
= \:\, \inf_{u \, \in \, U_x} \: \partial^- V_{\mathrm{loc}}(x; \, f(x, u)) \:\, = \:\,
\inf_{u \, \in \, U} \: \partial^- V_{\mathrm{loc}}(x; \, f(x, u)).
\end{aligned}
$$
Together with (\ref{Eq_3}), this leads to
$$
\inf_{u \, \in \, U} \: \partial^- V(x; \, f(x, u)) \:\, \leqslant \:\, -W_{\mathrm{loc}}(x)
\:\, \leqslant \:\, -W(x)
$$
and thereby completes the proof.
\end{proof}

\subsection{Investigation of the exit-time optimal control problem}

As was shown in the previous subsection, if one can find a suitable local CLF~$ V_{\mathrm{loc}}(\cdot) $ and the conditions of
Theorem~\ref{Thm_18} are fulfilled, the value function in the exit-time optimal control problem~(\ref{Eq_21}) extends
the local CLF outside the sublevel set~$ \Omega_c $, so that the resulting function~$ V(\cdot) $ becomes a global CLF
in the whole domain of asymptotic null-controllability~$ \mathcal{D}_0 $.

In order to verify the existence of optimal control strategies and to use necessary optimality conditions (Pontryagin's
principle) for the exit-time problem~(\ref{Eq_21}) with $ \, x_0 \, \in \, \mathcal{D}_0 \setminus \Omega_c, \, $ let us
reformulate it as
\begin{equation}
V(x_0) \:\: \stackrel{\mathrm{def}}{=} \:\: \inf_{\substack{u(\cdot) \, \in \, \mathcal{U}, \:\:
T \, \in \, [0, +\infty) \: \colon \\ x(T; \: x_0, \, u(\cdot)) \,\, \in \,\, \Omega_c}} \:
\left\{ \int\limits_0^T g(x(t; \, x_0, u(\cdot)), \: u(t)) \: \mathrm{d} t \:\, + \:\, c \right\} \quad
\forall \, x_0 \, \in \, G \setminus \Omega_c.  \label{Eq_29}
\end{equation}
It is easy to see that (\ref{Eq_21}) and (\ref{Eq_29}) are equivalent under Assumptions~\ref{Ass_1}, \ref{Ass_3},
\ref{Ass_7} and \ref{Ass_13}. Some additional conditions also need to be imposed.

\begin{assumption}  \label{Ass_19}
The set
$$
\left\{ (f(x, u), \, y) \: \in \: \mathbb{R}^n \times \mathbb{R} \: \colon \: u \in U, \:\: y \, \geqslant \, g(x, u) \right\}
$$
is convex for every~$ x \in \bar{G} $.
\end{assumption}

\begin{remark}  \label{Rem_20}  \rm
Assumption~\ref{Ass_19} strengthens Assumption~\ref{Ass_3}. One can easily verify that a sufficient condition for the fulfillment of
Assumption~\ref{Ass_19} is the convexity of the set
$$
\left\{ (f(x, u), \, g(x, u)) \: \in \: \mathbb{R}^n \times \mathbb{R} \: \colon \: u \in U \right\}
$$
for all~$ x \in \bar{G} $. \qed
\end{remark}

\begin{assumption}  \label{Ass_21}
The functions $ \: G \, \ni \, x \:\, \longmapsto \:\, f(x, u) \, \in \, \mathbb{R}^n \: $ and
$ \: G \, \ni \, x \:\, \longmapsto \:\, g(x, u) \, \in \, [0, +\infty) \: $ are continuously differentiable for
every~$ u \in U $.
\end{assumption}

\begin{theorem}  \label{Thm_22}
Let Assumptions~{\rm \ref{Ass_1}, \ref{Ass_7}, \ref{Ass_13}} and {\rm \ref{Ass_19}} hold. For any fixed initial
state~$ \, x_0 \, \in \, \mathcal{D}_0 \setminus \Omega_c, \, $ there exists an optimal control strategy for
the exit-time problem~{\rm (\ref{Eq_21})} or{\rm ,} equivalently{\rm ,} for~{\rm (\ref{Eq_29})}.
\end{theorem}

\begin{proof}
Consider the optimal control problem~(\ref{Eq_21}) or (\ref{Eq_29}) with a fixed initial
state~$ \, x_0 \, \in \, \mathcal{D}_0 \setminus \Omega_c $. In line with Proposition~\ref{Pro_15}, one has
$ \, V(x_0) < +\infty $. Fix an arbitrary~$ \varepsilon > 0 $. By $ \mathcal{U}_{\varepsilon}(x_0) $,
denote the set of all $ u(\cdot) \in \mathcal{U} $ for which the cost is not greater than $ \, V(x_0) + \varepsilon $.
The control subclass~$ \mathcal{U}_{\varepsilon}(x_0) $ obviously contains a minimizing sequence. Recall also
the notation~(\ref{Eq_20}). By the definition of $ \mathcal{U}_{\varepsilon}(x_0) $, one has
$ \: T_{\Omega_c}(x_0, u(\cdot)) \, < \, +\infty \: $ for all $ \, u(\cdot) \, \in \, \mathcal{U}_{\varepsilon}(x_0) $.
If one proves that the integral funnel
\begin{equation}
\{ (t, \: x(t; \, x_0, u(\cdot))) \:\, \colon \:\, t \: \in \: [0, \, T_{\Omega_c}(x_0, u(\cdot))], \:\:\,
u(\cdot) \, \in \, \mathcal{U}_{\varepsilon}(x_0) \}  \label{Eq_30}
\end{equation}
is contained in some compact set~$ K \subset \mathbb{R}^{n + 1} $, then including the constraint that admissible
integral trajectories should lie in $ K $ will not change the infimum in the considered optimal control problem,
while this will allow for using the general existence theorem of \cite[\S 9.3]{Cesari1983}. Thus, it remains
to establish the boundedness of (\ref{Eq_30}). According to Remark~\ref{Rem_2}, it suffices to verify that
the set of exit times
$ \: \{ T_{\Omega_c}(x_0, u(\cdot)) \: \colon \: u(\cdot) \, \in \, \mathcal{U}_{\varepsilon}(x_0) \} \: $
is bounded. Due to the definition of $ \mathcal{U}_{\varepsilon}(x_0) $ and Item~4 of Assumption~\ref{Ass_13},
any $ \, u(\cdot) \, \in \, \mathcal{U}_{\varepsilon}(x_0) \, $ satisfies
$$
C_6 \: T_{\Omega_c}(x_0, \, u(\cdot)) \:\, + \:\, c \:\: \leqslant \:\:
\int\limits_0^{T_{\Omega_c}(x_0, \, u(\cdot))} g(x(t; \, x_0, u(\cdot)), \: u(t)) \: \mathrm{d} t \:\, + \:\, c \:\:
\leqslant \:\: V(x_0) \, + \, \varepsilon
$$
with a constant~$ C_6 > 0 $, which leads to the estimate
$$
T_{\Omega_c}(x_0, \, u(\cdot)) \: \leqslant \: \frac{V(x_0) \, + \, \varepsilon \, - \, c}{C_6}
$$
and therefore completes the proof.
\end{proof}

\begin{theorem}{\rm (Pontryagin's principle; see, e.\,g.,
					 \cite[\S 5.1, \S 4.2 (emphasize Remark~10), \S 4.4.B]{Cesari1983},
					 \cite[\S 2.4]{Geering2007})}  \label{Thm_23}
Let Assumptions~{\rm \ref{Ass_1}, \ref{Ass_7}, \ref{Ass_13}, \ref{Ass_19}} and {\rm \ref{Ass_21}} hold.
Consider an optimal control strategy~$ u^*(\cdot) \in \mathcal{U} $ in the exit-time
problem~{\rm (\ref{Eq_21})} or{\rm ,} equivalently{\rm ,} in {\rm (\ref{Eq_29})} for a fixed initial
state~$ \, x_0 \, \in \, \mathcal{D}_0 \setminus \Omega_c $. Denote
$ \: T^* \, \stackrel{\mathrm{def}}{=} \, T_{\Omega_c}(x_0, \, u^*(\cdot)) \, < \, +\infty, \: $ and let
$$
[0, T^*] \: \ni \: t \:\: \longmapsto \:\: x^*(t) \: \stackrel{\mathrm{def}}{=} \:
x(t; \, x_0, u^*(\cdot)) \: \in \: G
$$
be the corresponding optimal state trajectory. Moreover{\rm ,} introduce the Hamiltonian{\rm :}
\begin{equation}
\begin{aligned}
& H(x, u, p, \tilde{p}) \:\, \stackrel{\mathrm{def}}{=} \:\, \left< p, f(x, u) \right> \: + \: \tilde{p} \, g(x, u), \\
& \mathcal{H}(x, p, \tilde{p}) \:\, \stackrel{\mathrm{def}}{=} \:\, \min_{u' \, \in \, U} \, H(x, u', p, \tilde{p}) \\
& \forall \: (x, u, p, \tilde{p}) \: \in \: G \times U \times \mathbb{R}^n \times \mathbb{R}.
\end{aligned}  \label{Eq_31}
\end{equation}
Then there exist a function $ \: p^* \colon \, [0, T^*] \to \mathbb{R}^n \: $ and
a constant~$ \tilde{p}^* \geqslant 0 $ such that the following properties hold{\rm :}
\begin{itemize}
\setlength\itemsep{0em}
\item  $ (p^*(t), \tilde{p}^*) \, \neq \, 0_{n + 1} \: $ for every $ t \in [0, T^*] ${\rm ;}
\item  $ (x^*(\cdot), p^*(\cdot)) \, $ is an absolutely continuous solution of the characteristic boundary value problem
\begin{equation}
\left\{ \begin{aligned}
& \dot{x^*}(t) \:\: = \:\: \mathrm{D}_p H(x^*(t), \, u^*(t), \, p^*(t), \, \tilde{p}^*) \:\: = \:\:
f(x^*(t), \, u^*(t)), \\
& \dot{p^*}(t) \:\: = \:\: -\mathrm{D}_x H(x^*(t), \, u^*(t), \, p^*(t), \, \tilde{p}^*) \\
& \qquad \:\:
= \:\: -(\mathrm{D}_x f(x^*(t), \, u^*(t)))^{\top} \: p^*(t) \:\, - \:\,
\tilde{p}^* \: \mathrm{D}_x g(x^*(t), \, u^*(t)), \\
& t \, \in \, [0, T^*], \\
& x^*(0) \, = \, x_0, \\
& x^*(T^*) \, \in \, l_c, \quad p^*(T^*) \, \in \, \mathrm{N}(x^*(T^*); \, \Omega_c)
\end{aligned} \right.  \label{Eq_32}
\end{equation}
{\rm (}the notation for normal cones was described in the introduction{\rm );}
\item  the Hamiltonian minimum condition
\begin{equation}
H(x^*(t), \, u^*(t), \, p^*(t), \, \tilde{p}^*) \:\, = \:\, \mathcal{H}(x^*(t), \, p^*(t), \, \tilde{p}^*)
\label{Eq_33}
\end{equation}
is satisfied for almost all $ \, t \in [0, T^*] \, $ {\rm (}with respect to the Lebesgue measure in $ \mathbb{R} ${\rm );}
\item  the Hamiltonian vanishes along the optimal characteristic trajectory{\rm ,} i.\,e.{\rm ,}
\begin{equation}
\mathcal{H}(x^*(t), \, p^*(t), \, \tilde{p}^*) \: \equiv \: 0 \quad \forall t \in [0, T^*].  \label{Eq_34}
\end{equation}
\end{itemize}
\end{theorem}

\begin{remark}  \label{Rem_24}  \rm
Since the Hamiltonian~(\ref{Eq_31}) is positive homogeneous of degree~$ 1 $ with respect to $ (p, \tilde{p}) $,
it suffices to consider only the two cases~$ \tilde{p}^* = 0  $ and $ \tilde{p}^* = 1 $ in Theorem~\ref{Thm_23}.
The case~$ \tilde{p}^* = 0 $ is called abnormal. \qed
\end{remark}

For handling the infinite value~$ +\infty $, consider the Kruzhkov transformed function
\begin{equation}
v(x_0) \:\, \stackrel{\mathrm{def}}{=} \:\, 1 \, - \, e^{-V(x_0)} \:\, \in \:\, [0, 1] \quad \forall x_0 \in G
\label{Eq_35}
\end{equation}
with the convention $ \, e^{-(+\infty)} \stackrel{\mathrm{def}}{=} 0 $. Note that the function
$ \: \mathbb{R} \ni \xi \: \longmapsto \: 1 - e^{-\xi} \: $ vanishes at $ \xi = 0 $, tends to $ 1 $ as $ \xi \to +\infty $,
strictly increases, and is infinitely differentiable.

\begin{theorem}  \label{Thm_25}
Let Assumptions~{\rm \ref{Ass_1}, \ref{Ass_3}, \ref{Ass_7}} and {\rm \ref{Ass_13}} hold{\rm ,} and consider
the functions~$ V(\cdot) ${\rm ,} $ v(\cdot) $ defined by {\rm (\ref{Eq_21}), (\ref{Eq_25}), (\ref{Eq_35})}.
The domain of asymptotic null-controllability can be represented as
$$
\mathcal{D}_0 \:\: = \:\: \{ x_0 \in G \: \colon \: V(x_0) \, < \, +\infty \} \:\: = \:\:
\{ x_0 \in G \: \colon \: v(x_0) \, < \, 1 \}.
$$
\end{theorem}

\begin{proof}
It suffices to recall Proposition~\ref{Pro_15}.
\end{proof}

Introduce also the set-valued extremal control map:
\begin{equation}
U^*(x, p, \tilde{p}) \:\, \stackrel{\mathrm{def}}{=} \:\, \mathrm{Arg} \min_{u \, \in \, U} \,
H(x, u, p, \tilde{p}) \quad
\forall \: (x, p, \tilde{p}) \: \in \: G \times \mathbb{R}^n \times \mathbb{R}.  \label{Eq_36}
\end{equation}

As was discussed in \cite{YegorovDower2017,YegorovDowerGrune2018_1,YegorovDowerGrune2018_2}, characteristic
boundary value problems, such as (\ref{Eq_32}), may admit multiple solutions, some of which may not be optimal,
and it is therefore relevant to parametrize the characteristic fields with respect to the extended initial adjoint
vector ($ (p_0, \tilde{p}^*) $ in case of (\ref{Eq_32})) and to solve the related Cauchy problems. Solutions of
the latter are unique if, for example, the absence of the abnormal case~$ \tilde{p}^* = 0 $ is verified and
the running cost is regularized by adding an appropriate control-dependent term, so that the extremal control map
takes only singleton values along the characteristic trajectories.

Taking that into account, the next theorem reduces the computation of the transformed value function~(\ref{Eq_35}) at
any selected state~$ \, x_0 \, \in \, \mathcal{D}_0 \setminus \Omega_c \, $ to a finite-dimensional optimization problem with
respect to the unknown initial data~$ (p_0, \tilde{p}^*) $ for the characteristic system. For certain classes of
optimal control problems with fixed finite horizons and free terminal states, some related techniques were previously
proposed and tested in \cite{DarbonOsher2016,ChowDarbonOsherYin2018,YegorovDower2017}. Theoretical results regarding
the construction of global CLFs via exit-time optimal control and Pontryagin's characteristics were initially formulated in
the conference papers~\cite{YegorovDowerGrune2018_1,YegorovDowerGrune2018_2}, while the current work provides their
extension with detailed proofs, discussions, and practical developments.

\begin{theorem}  \label{Thm_26}
Let Assumptions~{\rm \ref{Ass_1}, \ref{Ass_7}, \ref{Ass_13}, \ref{Ass_19}} and {\rm \ref{Ass_21}} hold. For any
initial state~$ \, x_0 \, \in \, \mathcal{D}_0 \setminus \Omega_c, \, $ the Kruzhkov transformed value~$ v(x_0) $
defined by {\rm (\ref{Eq_20}),~(\ref{Eq_21}),~(\ref{Eq_35})} is the minimum of
\begin{equation}
1 \:\: - \:\: \exp \left\{ -\int\limits_0^{T_{\Omega_c}(x_0, \, u^*(\cdot))} g(x^*(t), \, u^*(t)) \: \mathrm{d} t \:\, - \:\, c \right\}
\label{Eq_37}
\end{equation}
over the solutions of the characteristic Cauchy problems
\begin{equation}
\left\{ \begin{aligned}
& \dot{x^*}(t) \:\: = \:\: \mathrm{D}_p H(x^*(t), \, u^*(t), \, p^*(t), \, \tilde{p}^*) \:\: = \:\:
f(x^*(t), \, u^*(t)), \\
& \dot{p^*}(t) \:\: = \:\: -\mathrm{D}_x H(x^*(t), \, u^*(t), \, p^*(t), \, \tilde{p}^*) \\
& \qquad \:\:
= \:\: -(\mathrm{D}_x f(x^*(t), \, u^*(t)))^{\top} \: p^*(t) \:\, - \:\,
\tilde{p}^* \: \mathrm{D}_x g(x^*(t), \, u^*(t)), \\
& u^*(t) \: \in \: U^*(x^*(t), \, p^*(t), \, \tilde{p}^*), \\
& t \:\, \in \:\, I(x_0, \, u^*(\cdot)) \:\, \stackrel{\mathrm{def}}{=} \:\, \begin{cases}
[0, \, T_{\Omega_c}(x_0, \, u^*(\cdot))], & T_{\Omega_c}(x_0, \, u^*(\cdot)) \: < \: +\infty, \\
[0, +\infty), & T_{\Omega_c}(x_0, \, u^*(\cdot)) \: = \: +\infty,
\end{cases} \\
& x^*(0) \, = \, x_0, \quad p^*(0) \, = \, p_0,
\end{aligned} \right.  \label{Eq_38}
\end{equation}
for all extended initial adjoint vectors
\begin{equation}
(p_0, \tilde{p}^*) \:\, \in \:\, \{ (p, \tilde{p}) \: \colon \: p \in \mathbb{R}^n, \:\: \tilde{p} \in \{ 0, 1 \} \}.
\label{Eq_39}
\end{equation}
Moreover{\rm ,} the same value is obtained when minimizing over the bounded set
\begin{equation}
(p_0, \tilde{p}^*) \:\, \in \:\, \{ (p, \tilde{p}) \: \in \: \mathbb{R}^n \times \mathbb{R} \: \colon \:
\| (p, \tilde{p}) \| \: = \: 1, \:\:\, \tilde{p} \geqslant 0 \},  \label{Eq_40}
\end{equation}
or even over its subset
\begin{equation}
(p_0, \tilde{p}^*) \:\, \in \:\, \{ (p, \tilde{p}) \: \in \: \mathbb{R}^n \times \mathbb{R} \: \colon \:
\| (p, \tilde{p}) \| \: = \: 1, \:\:\, \tilde{p} \geqslant 0, \:\:\,
\mathcal{H}(x_0, p, \tilde{p}) \: = \: 0 \}.  \label{Eq_41}
\end{equation}
\end{theorem}

\begin{proof}
The first statement directly follows from Theorems~\ref{Thm_22}, \ref{Thm_23}, Remark~\ref{Rem_24} and
the fact that, compared to the boundary value problems~(\ref{Eq_32}), (\ref{Eq_33}), the Cauchy
problems~(\ref{Eq_38}),~(\ref{Eq_39}) generate a wider characteristic field (due to the absence of
the transversality condition on the terminal adjoint vector).

Since the Hamiltonian is positive homogeneous of degree~$ 1 $ with respect to $ \left( p, \tilde{p} \right) $,
the extremal control map~(\ref{Eq_36}) satisfies
$$
U^* \left( x, p, \tilde{p} \right) \:\, = \:\, U^* \left( x, \, \lambda \, p, \, \lambda \, \tilde{p} \right) \quad
\forall \lambda > 0 \quad \forall \: \left( x, p, \tilde{p} \right) \: \in \:
G \times \mathbb{R}^n \times \mathbb{R},
$$
and the state components of the characteristic trajectories do not change after multiplying
$ \left( p_0, \tilde{p}^* \right) $ by any positive number. Together with the Hamiltonian vanishing condition~(\ref{Eq_34})
in Theorem~\ref{Thm_23}, this yields the second statement.
\end{proof}

Besides, let us separately formulate the well-known Hamiltonian conservation property as applied to (\ref{Eq_38}). For convenience,
its proof is given in Subsection~A.1.2 of \hyperlink{Online_App}
{the appendix}.

\begin{proposition}  \label{Pro_27}
Under the conditions of Theorem~{\rm \ref{Thm_26},} the Hamiltonian is conserved along any solution of
the characteristic Cauchy problem~{\rm (\ref{Eq_38})} with
$ \: (x_0, p_0, \tilde{p}^*) \: \in \: G \, \times \, \mathbb{R}^n \, \times \, [0, +\infty) $.
\end{proposition}

Theorems~\ref{Thm_18}, \ref{Thm_22}, \ref{Thm_23}, \ref{Thm_25} and \ref{Thm_26} form the theoretical basis of
a curse-of-dimensionality-free approach to approximating global CLFs and feedback stabilization. A number of related
practical aspects are discussed in Section~4 below and also in Section~A.2 of \hyperlink{Online_App}
{the appendix},
while the next section modifies the theoretical constructions of the current section for the case when an appropriate local
CLF is not available.

\section{Approximation of a global CLF in case when an appropriate local CLF is not available}

If linearization based techniques for building quadratic local CLFs (see Subsection~A.2.1 of
\hyperlink{Online_App}
{the appendix}) cannot be applied to a particular system, it may be very difficult to obtain
a suitable local CLF. Even if the conditions of Theorems~\ref{Thm_18}, \ref{Thm_22}, \ref{Thm_25} and \ref{Thm_26} hold with
an explicitly found nonsmooth local CLF and some software implementation of a direct approximation method (e.\,g., one of
the toolkits mentioned in Subsection~A.2.3 of \hyperlink{Online_App}
{the appendix}) can be launched for
the corresponding exit-time optimal open-loop control problems, the nonsmoothness may still cause significant numerical issues.
This section describes the theoretical constructions that were previously introduced in the conference
paper~\cite{YegorovDowerGrune2018_1} and could help to approximate global CLFs for some classes of nonlinear control systems
without using any local CLFs. Here we also discuss the proofs omitted in \cite{YegorovDowerGrune2018_1} and provide a qualitative
comparison with the constructions of Section~2.

Let us consider the control system~(\ref{Eq_1}) and indicate the required assumptions.

Since it is not asserted that a suitable local CLF can be obtained, some new conditions on the running
cost~$ g(\cdot, \cdot) $ have to be imposed (they do not appear in Section~2). Furthermore, it is convenient
to assume the boundedness of the set of pointwise control constraints~$ U $ from the very beginning
(in \cite{YegorovDowerGrune2018_1}, this was supposed just after Proposition~2.8 stating that
the region of asymptotic null-controllability~$ \mathcal{D}_0 $ is an open domain, but before Assumption~2.12
introducing the running cost).

First, Assumption~\ref{Ass_1} is adopted. It is also supposed that a local asymptotic null-controllability property
holds in a weak or strong form as follows (see \cite[Section~2]{CamilliGruneWirth2008}).

\begin{assumption}  \label{Ass_32}
$ 0_m \in U ${\rm ,} $ f(0_n, 0_m) = 0_n ${\rm ,} and one of the following two conditions holds {\rm (}the second
condition strengthens the first one and is called the small control property{\rm ):}
\begin{list}{\rm \arabic{count})}%
{\usecounter{count}}
\setlength\itemsep{0em}
\item  there exist positive constants~$ r, \bar{u} $ and
a function~$ \, \beta(\cdot, \cdot) \in \mathcal{K} \mathcal{L} \, $ such that $ \bar{\mathrm{B}}_r(0_n) \subset G $
and{\rm ,} for any $ x_0 \in \mathrm{B}_r(0_n) ${\rm ,} there is a control strategy~$ u_{x_0}(\cdot) \in \mathcal{U} $
satisfying
\begin{equation}
\arraycolsep=1.5pt
\def\arraystretch{2}
\begin{array}{c}
\left\| u_{x_0}(\cdot) \right\|_{L^{\infty}([0, +\infty), \, U)} \:\, \leqslant \:\, \bar{u}, \\
\left\| x \left( t; \, x_0, u_{x_0}(\cdot) \right) \right\| \:\, \leqslant \:\,
\beta(\| x_0 \|, t) \quad \forall t \geqslant 0;
\end{array}  \label{Eq_69}
\end{equation}
\item  there exists a constant~$ r > 0 $ and a function~$ \, \beta(\cdot, \cdot) \in \mathcal{K} \mathcal{L} \, $
such that $ \bar{\mathrm{B}}_r(0_n) \subset G $ and{\rm ,} for any $ x_0 \in \mathrm{B}_r(0_n) ${\rm ,} there is
a control strategy~$ u_{x_0}(\cdot) \in \mathcal{U} $ satisfying
$$
\left\| x \left( t; \, x_0, u_{x_0}(\cdot) \right) \right\| \: + \:
\left\| u_{x_0}(t) \right\| \:\: \leqslant \:\: \beta(\| x_0 \|, t) \quad
\forall t \geqslant 0
$$
{\rm (}this implies {\rm (\ref{Eq_69})} with $ \bar{u} = \beta(r, 0) ${\rm )}.
\end{list}
\end{assumption}

\begin{remark}  \label{Rem_33}  \rm
If $ \, 0_m \in \mathrm{int} \, U $, $ \, f(0_n, 0_m) = 0_n, \, $ the function~$ f(\cdot, \cdot) $ is continuously
differentiable, and the linearization
\begin{equation}
\arraycolsep=1.5pt
\def\arraystretch{1.5}
\begin{array}{c}
\dot{x}(t) \: = \: A \, x(t) \, + \, B \, u(t), \quad t \geqslant 0, \quad u(\cdot) \, \in \, \mathcal{U}, \\
A \: \stackrel{\mathrm{def}}{=} \: \mathrm{D}_x f(0_n, 0_m) \: \in \: \mathbb{R}^{n \times n}, \quad
B \: \stackrel{\mathrm{def}}{=} \: \mathrm{D}_u f(0_n, 0_m) \: \in \: \mathbb{R}^{n \times m},
\end{array}  \label{Eq_60}
\end{equation}
of the system~(\ref{Eq_1}) is asymptotically null-controllable, then (\ref{Eq_1}) admits a locally stabilizing
linear feedback according to \cite[\S 5.8, Theorem~19]{SontagBook1998}, and Item~2 of Assumption~\ref{Ass_32} therefore
holds.  \qed
\end{remark}

\begin{remark}  \label{Rem_34}  \rm
Due to \cite[Proposition~7]{Sontag1998}, $ \, \beta(\cdot, \cdot) \in \mathcal{K} \mathcal{L} \, $ implies
the existence of two functions $ \: \alpha_1(\cdot), \alpha_2(\cdot) \, \in \, \mathcal{K}_{\infty} \: $ satisfying
\begin{equation}
\beta(\rho, t) \: \leqslant \: \alpha_2 \left( \alpha_1(\rho) \, e^{-t} \right) \quad \forall \rho \geqslant 0 \quad
\forall t \geqslant 0.  \label{Eq_73}
\end{equation}
For example, if $ C_9, C_{10} $ are positive constants, $ \nu(\cdot) \in \mathcal{K}_{\infty} $, and
$$
\beta(\rho, t) \: = \: C_9 \, \nu(\rho) \, e^{-C_{10} \, t} \quad \forall \rho \geqslant 0 \quad \forall t \geqslant 0,
$$
then one can choose
$$
\alpha_1(\rho) \: = \: (\nu(\rho))^{1 / C_{10}}, \quad \alpha_2(\rho) \: = \: C_9 \, \rho^{C_{10}} \quad
\forall \rho \geqslant 0 \quad \forall t \geqslant 0
$$
in order to fulfill the estimate~(\ref{Eq_73}) in the equality form. \qed
\end{remark}

\begin{proposition}{\rm \cite[Proposition~2.3]{CamilliGruneWirth2008}}  \label{Pro_35}
Let Assumptions~{\rm \ref{Ass_1}} and {\rm \ref{Ass_32}} hold. The region of asymptotic null-controllability~$ \mathcal{D}_0 $
is an open domain containing the closed ball~$ \bar{\mathrm{B}}_r(0_n) $. Furthermore{\rm ,} $ \mathcal{D}_0 $ is weakly
invariant in the sense that{\rm ,} for any $ x_0 \in \mathcal{D}_0 ${\rm ,} there exists
$ u_{x_0}(\cdot) \in \mathcal{U} $ satisfying $ \: x(t; \, x_0, u_{x_0}(\cdot)) \, \in \, \mathcal{D}_0 \: $ for all
$ t \geqslant 0 $.
\end{proposition}

Next, let us formulate the conditions on the running cost~$ g(\cdot, \cdot) $ in an infinite-horizon optimal control
problem leading to a global CLF in line with the results of \cite[Sections~3 and 4]{CamilliGruneWirth2008}.

\begin{assumption}  \label{Ass_36}
Let $ \alpha_2^{-1}(\cdot) $ be the inverse of the function~$ \alpha_2(\cdot) $ introduced in Remark~{\rm \ref{Rem_34},}
and take the constants~$ r, \bar{u} $ from Assumption~{\rm \ref{Ass_32}}. The following properties hold{\rm :}
\begin{list}{\rm \arabic{count})}%
{\usecounter{count}}
\setlength\itemsep{0em}
\item  $ \bar{G} \times U \, \ni \, (x, u) \: \longmapsto \: g(x, u) \, \in \, [0, +\infty) \: $ is a nonnegative
continuous function{\rm ,} and{\rm ,} for any $ R > 0 ${\rm ,} there exists $ C_{5, R} > 0 $ satisfying
the condition~{\rm (\ref{Eq_68}) (}these are Items~{\rm 1,\,2} of Assumption~{\rm \ref{Ass_13});}
\item  for any $ R > 0 ${\rm ,} one has
$$
\inf \: \{ g(x, u) \: \colon \: (x, u) \, \in \, \bar{G} \times U, \:\: \| x \| \, \geqslant \, R \} \:\: > \:\: 0;
$$
\item  if Item~{\rm 2} of Assumption~{\rm \ref{Ass_32} (}the small control property{\rm )} is not asserted{\rm ,} then
there exist positive constants~$ C_{11}, C_{12} $ such that
\begin{equation}
g(x, u) \:\: \leqslant \:\: C_{11} \: \left( \alpha_2^{-1}(\| x \|) \right)^{C_{12}} \quad
\forall x \in \bar{\mathrm{B}}_r(0_n) \quad \forall \: u \, \in \, \bar{\mathrm{B}}_{\bar{u}}(0_m) \, \cap \, U;  \label{Eq_70}
\end{equation}
\item  if Item~{\rm 2} of Assumption~{\rm \ref{Ass_32}} holds{\rm ,} then the condition~{\rm (\ref{Eq_70})} is weakened to
$$
g(x, u) \:\: \leqslant \:\: C_{11} \: \left( \alpha_2^{-1}(\| x \| + \| u \|) \right)^{C_{12}} \quad
\forall x \in \bar{\mathrm{B}}_r(0_n) \quad \forall \: u \, \in \, \bar{\mathrm{B}}_{\bar{u}}(0_m) \, \cap \, U,
$$
where $ C_{11}, C_{12} $ are positive constants{\rm ;}
\item  there exists a constant~$ C_{13} > 0 $ satisfying
$$
g(x, u) \: \geqslant \: C_{13} \, \| f(x, u) \| \quad
\forall \: (x, u) \: \in \: \{ (x', u') \, \in \, \bar{G} \times U \: \colon \:
\| x' \| \, \geqslant \, 2 r \:\: \mathrm{or} \:\: \| u' \| \, \geqslant \, 2 \bar{u} \}
$$
{\rm (}we take $ \bar{u} = \beta(r, 0) $ if Item~{\rm 2} of Assumption~{\rm \ref{Ass_32}} holds{\rm )}.
\end{list}
\end{assumption}

Note the difference between Items~2--5 of Assumption~\ref{Ass_36} on one hand, and Items~3,\,4 of Assumption~\ref{Ass_13}
together with Assumption~\ref{Ass_16} on the other.

Introduce the infinite-horizon optimal control problem
\begin{equation}
V_0(x_0) \:\: \stackrel{\mathrm{def}}{=} \:\: \inf_{u(\cdot) \, \in \, \mathcal{U}} \:
\left\{ \int\limits_0^{+\infty} g(x(t; \, x_0, u(\cdot)), \: u(t)) \: \mathrm{d} t \right\} \:\: \in \:\:
[0, +\infty) \, \cup \, \{ +\infty \} \quad
\forall x_0 \in G.  \label{Eq_71}
\end{equation}
Similarly to (\ref{Eq_35}), consider the Kruzhkov transformed value function
\begin{equation}
v_0(x_0) \:\, \stackrel{\mathrm{def}}{=} \:\, 1 \, - \, e^{-V_0(x_0)} \:\, \in \:\, [0, 1] \quad \forall x_0 \in G
\label{Eq_72}
\end{equation}
with the convention $ \, e^{-(+\infty)} \stackrel{\mathrm{def}}{=} 0 $.

Similarly to \cite[Propositions~3.3,~3.5,~3.6 and Remark~4.2]{CamilliGruneWirth2008}, one can obtain the following
result that represents the domain of asymptotic null-controllability~$ \mathcal{D}_0 $ via the value
functions~$ V_0(\cdot), v_0(\cdot) $ and indicates the CLF property for $ V_0(\cdot) $ in $ \mathcal{D}_0 $.

\begin{theorem}  \label{Thm_37}
Let Assumptions~{\rm \ref{Ass_1}, \ref{Ass_32}} and Items~{\rm 1--4} of Assumption~{\rm \ref{Ass_36}} hold.
Then the domain of asymptotic null-controllability is represented as
\begin{equation}
\mathcal{D}_0 \:\: = \:\: \{ x_0 \in G \: \colon \: V_0(x_0) \, < \, +\infty \} \:\: = \:\:
\{ x_0 \in G \: \colon \: v_0(x_0) \, < \, 1 \}.  \label{Eq_79}
\end{equation}
If{\rm ,} moreover{\rm ,} Item~{\rm 5} of Assumption~{\rm \ref{Ass_36}} holds{\rm ,} then the restriction of
$ V_0(\cdot) $ to $ \mathcal{D}_0 $ is a CLF for the system~{\rm (\ref{Eq_1}),} and the following statements
in particular hold{\rm :}
\begin{itemize}
\setlength\itemsep{0em}
\item  $ V_0(\cdot) $ is continuous on $ \mathcal{D}_0 ${\rm ,} $ v_0(\cdot) $ is continuous on $ G ${\rm ;}
\item  $ \{ x_0 \in G \: \colon \: V_0(x_0) \, = \, 0 \} \:\, = \:\,
\{ x_0 \in G \: \colon \: v_0(x_0) \, = \, 0 \} \:\, = \:\, \{ 0_n \} ${\rm ;}
\item  for any sequence $ \, \left\{ x^{(k)} \right\}_{k = 1}^{\infty} \subset G \, $ satisfying either
$$
\lim_{k \, \to \, \infty} \, \mathrm{dist} \, \left( x^{(k)}, \partial \mathcal{D}_0 \right) \: = \: 0 \quad
\mathrm{or} \quad \lim_{k \, \to \, \infty} \, \left\| x^{(k)} \right\| \: = \: +\infty,
$$
one also has
$$
\lim_{k \, \to \, \infty} \, V_0 \left( x^{(k)} \right) \: = \: +\infty \quad \mathrm{and} \quad
\lim_{k \, \to \, \infty} \, v_0 \left( x^{(k)} \right) \: = \: 1.
$$
\end{itemize}
\end{theorem}

The dynamic programming principle for the transformed value function~$ v_0(\cdot) $ can be formulated as follows
(see, e.\,g., \cite[Section~3]{CamilliGruneWirth2008} and \cite[Sections~3,\,4]{Soravia1999}).

\begin{proposition}  \label{Pro_38}
Under Assumptions~{\rm \ref{Ass_1}, \ref{Ass_32}} and Items~{\rm 1--4} of Assumption~{\rm \ref{Ass_36},}
one has
\begin{equation}
\begin{aligned}
& v_0(x_0) \:\: = \:\: \inf\limits_{u(\cdot) \, \in \, \mathcal{U}} \: \inf\limits_{T \, \in \, [0, +\infty)} \:
\{ 1 \:\, - \:\, \mu(x_0, u(\cdot), T) \:\,
+ \:\, \mu(x_0, u(\cdot), T) \: v_0(x(T; \, x_0, u(\cdot))) \} \\
& \qquad\quad
= \:\: \inf\limits_{u(\cdot) \, \in \, \mathcal{U}} \: \sup\limits_{T \, \in \, [0, +\infty)} \:
\{ 1 \:\, - \:\, \mu(x_0, u(\cdot), T) \:\,
+ \:\, \mu(x_0, u(\cdot), T) \: v_0(x(T; \, x_0, u(\cdot))) \} \\
& \forall x_0 \in G,
\end{aligned}  \label{Eq_80}
\end{equation}
where
\begin{equation}
\mu(x_0, u(\cdot), T) \:\: \stackrel{\mathrm{def}}{=} \:\: \exp \left\{ -\int\limits_0^T
g(x(t; \, x_0, u(\cdot)), \: u(t)) \: \mathrm{d} t \right\} \quad
\forall x_0 \in G \quad \forall u(\cdot) \in \mathcal{U} \quad \forall T \in [0, +\infty).  \label{Eq_81}
\end{equation}
\end{proposition}

The next theorem can be established similarly to \cite[Theorem~4.4]{CamilliGruneWirth2008} and in fact extends
the classical Zubov method for constructing Lyapunov functions~\cite{Zubov1964} to the problem of weak asymptotic
null-controllability. Due to the compactness of $ U $, there is no need to adopt \cite[Hypothesis~(H6)]{CamilliGruneWirth2008},
which states that, for any $ x \in G $ and $ \, \left\{ u^{(k)} \right\}_{k = 1}^{\infty} \subseteq U \, $ satisfying
$ \: \lim_{k \, \to \, \infty} \, \left\| u^{(k)} \right\| \: = \: +\infty, \: $ one also has
$$
\lim_{k \, \to \, \infty} \: \frac{\left| g \left( x, u^{(k)} \right) \right|}{1 \: + \:
\left\| f \left( x, u^{(k)} \right) \right\|} \:\, = \:\, +\infty.
$$

\begin{theorem}  \label{Thm_39}
Under Assumptions~{\rm \ref{Ass_1}, \ref{Ass_32}} and {\rm \ref{Ass_36},} the transformed value function~$ v_0(\cdot) $
is the unique bounded viscosity solution of the following boundary value problem for the HJB equation{\rm :}
\begin{equation}
\left\{ \begin{aligned}
& \max_{u \, \in \, U} \: \{ -\left< \mathrm{D} v_0(x), \, f(x, u) \right> \: - \:
(1 - v_0(x)) \, g(x, u) \} \:\: = \:\: 0, \quad x \in G, \\
& v_0(0_n) \, = \, 0.
\end{aligned} \right.  \label{Eq_74}
\end{equation}
\end{theorem}

For numerical purposes, it is reasonable to approximate the infinite-horizon optimal control problem~(\ref{Eq_71})
by an exit-time problem. If a local CLF and its level sets are not practically obtained, the approach of Section~2
cannot be used. The exit-time problem is then stated with respect to the closed ball~$ \bar{\mathrm{B}}_{\delta}(0_n) $
with center~$ x = 0_n $ and sufficiently small radius~$ \delta \in (0, r] $ (see Fig.~\ref{Fig_2} and recall that
the constant~$ r $ was introduced in Assumption~\ref{Ass_32}):
\begin{equation}
\begin{aligned}
& V_{\delta}(x_0) \:\: \stackrel{\mathrm{def}}{=} \:\: \inf \: \left\{
\int\limits_0^T g(x(t; \, x_0, u(\cdot)), \: u(t)) \: \mathrm{d} t \:\, \colon \right. \\
& \qquad\qquad\qquad\quad
\left. {}^{{}^{{}^{{}^{{}^{{}^{{}^{{}^{{}^{}}}}}}}}}
u(\cdot) \in \mathcal{U}, \:\,\, x(T; \, x_0, u(\cdot)) \: \in \: \bar{\mathrm{B}}_{\delta}(0_n) \:\,\,
\mbox{at some $ \, T \in [0, +\infty) $} \right\} \\
& \forall x_0 \in G \quad \forall \delta \in (0, r].
\end{aligned}  \label{Eq_75}
\end{equation}
The convention $ \, \inf \, \emptyset = +\infty \, $ is adopted as before. With the help of the notation
\begin{equation}
T_{\delta}(x_0, u(\cdot)) \:\: \stackrel{\mathrm{def}}{=} \:\: \inf \, \left\{ T \in [0, +\infty) \: \colon \:
x(T; \, x_0, u(\cdot)) \: \in \: \bar{\mathrm{B}}_{\delta}(0_n) \right\} \quad
\forall x_0 \in G \quad \forall u(\cdot) \in \mathcal{U}  \label{Eq_76}
\end{equation}
for the exit times, the value function~(\ref{Eq_75}) can also be determined by
\begin{equation}
V_{\delta}(x_0) \:\: = \:\: \inf_{\substack{u(\cdot) \, \in \, \mathcal{U} \: \colon \\
T_{\delta}(x_0, \, u(\cdot)) \: < \: +\infty}} \: \left\{
\int\limits_0^{T_{\delta}(x_0, \, u(\cdot))} g(x(t; \, x_0, u(\cdot)), \: u(t)) \: \mathrm{d} t \right\} \quad
\forall x_0 \in G \quad \forall \delta \in (0, r]  \label{Eq_77}
\end{equation}
(note that the running cost is nonnegative according to Item~1 of Assumption~\ref{Ass_36}). Consider also
the Kruzhkov transformed function
\begin{equation}
v_{\delta}(x_0) \:\, \stackrel{\mathrm{def}}{=} \:\, 1 \, - \, e^{-V_{\delta}(x_0)} \:\, \in \:\, [0, 1] \quad
\forall x_0 \in G \quad \forall \delta \in (0, r].  \label{Eq_78}
\end{equation}

\begin{figure}
\begin{center}
\includegraphics[width=6.5cm,height=3.25cm]{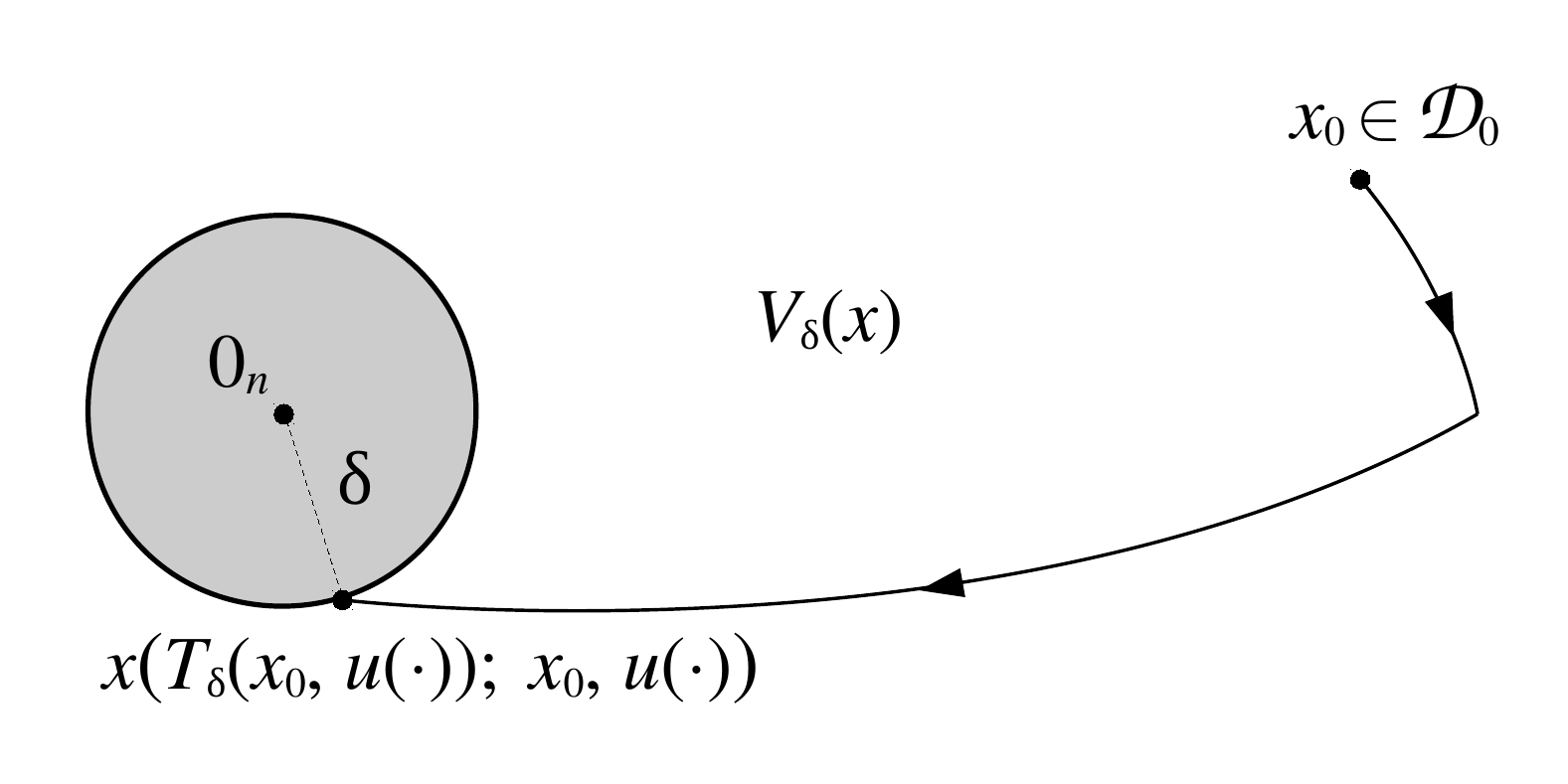}
\end{center}
\bf \caption{\rm The exit-time optimal control problem~{\rm (\ref{Eq_75}) (}or{\rm ,} equivalently{\rm , (\ref{Eq_77})),}
whose target set is the closed ball with center~$ x = 0_n $ and sufficiently small radius~$ \delta \in (0, r] $.}
\label{Fig_2}
\end{figure}

A key result on approximating the infinite-horizon problem~(\ref{Eq_71}) by the exit-time problem~(\ref{Eq_75})
(or, equivalently, by (\ref{Eq_77})) can now be derived.

\begin{theorem}  \label{Thm_40}
Under Assumptions~{\rm \ref{Ass_1}, \ref{Ass_32}} and {\rm \ref{Ass_36},} the following properties hold{\rm :}
\begin{list}{\rm \arabic{count})}%
{\usecounter{count}}
\setlength\itemsep{0em}
\item  the domain of asymptotic null-controllability can be represented as
$$
\mathcal{D}_0 \:\: = \:\: \{ x_0 \in G \: \colon \: V_{\delta}(x_0) \, < \, +\infty \} \:\: = \:\:
\{ x_0 \in G \: \colon \: v_{\delta}(x_0) \, < \, 1 \} \quad
\forall \delta \in [0, r]
$$
{\rm (}according to Definition~{\rm \ref{Def_4},} it is obvious that $ \mathcal{D}_0 $ does not depend on $ \delta ${\rm );}
\item  $ v_{\delta}(x_0) \to v_0(x_0) \, $ uniformly on $ G $ as $ \delta \to +0 ${\rm ;}
\item  $ V_{\delta}(x_0) \to V_0(x_0) \, $ uniformly on every compact subset of $ \mathcal{D}_0 $ as $ \delta \to +0 ${\rm ;}
\item  for every $ \delta \in (0, r] ${\rm ,} $ V_{\delta}(\cdot) $ is locally Lipschitz continuous in $ \mathcal{D}_0 $ {\rm (}and
therefore differentiable almost everywhere in $ \mathcal{D}_0 $ with respect to the Lebesgue measure in $ \mathbb{R}^n ${\rm )} if
the Petrov condition
$$
\min_{u \, \in \, U} \: \left< x, \, f(x, u) \right> \: < \: 0 \quad
\forall \: x \: \in \: \partial \mathrm{B}_{\delta}(0_n) \: = \: \{ x' \in \mathbb{R}^n \: \colon \: \| x' \| = \delta \}
$$
holds.
\end{list}
\end{theorem}

\begin{proof}
For any $ \delta \in (0, r] $, the condition~$ \, x_0 \, \in \, G \setminus \mathcal{D}_0 \, $ yields the absence of state
trajectories $ \: x(\cdot; \, x_0, u(\cdot)) \: $ corresponding to $ u(\cdot) \in \mathcal{U} $ and reaching the target
ball $ \, \bar{B}_{\delta}(0_n) \subseteq \bar{B}_r(0_n) \, $ in finite time, while such trajectories exist if
$ x_0 \in \mathcal{D}_0 $ (recall Definition~\ref{Def_4}, Assumption~\ref{Ass_32}, and the inclusion
$ \, \bar{\mathrm{B}}_r(0_n) \subset \mathcal{D}_0 \, $ from Proposition~\ref{Pro_35}). This and the property~(\ref{Eq_79})
lead to Item~1.

Note that Item~3 would follow from Item~2, because
$$
V_{\delta}(x_0) \: = \: - \ln \, (1 \, - \, v_{\delta}(x_0)) \quad
\forall x_0 \in \mathcal{D}_0 \quad \forall \delta \in [0, r]
$$
by virtue of the relations~(\ref{Eq_72}), (\ref{Eq_79}) and (\ref{Eq_78}). Besides, Item~4 can be established similarly to
Item~1 of Proposition~\ref{Pro_17} and the last sentence in Theorem~\ref{Thm_18}.

It hence remains to verify Item~2. If $ \, x_0 \, \in \, G \setminus \mathcal{D}_0, \, $ the statement follows directly from
Item~1 and the representation~(\ref{Eq_79}). Now consider arbitrary $ x_0 \in \mathcal{D}_0 $ and $ \delta \in (0, r] $.
Then the set of control strategies~$ u(\cdot) \in \mathcal{U} $ satisfying $ \, T_{\delta}(x_0, u(\cdot)) < +\infty \, $ is
nonempty. Due to Proposition~\ref{Pro_38}, one has
\begin{equation}
\begin{aligned}
& v_0(x_0) \:\: = \:\: \inf\limits_{\substack{u(\cdot) \, \in \, \mathcal{U} \: \colon \\
T_{\delta}(x_0, \, u(\cdot)) \: < \: +\infty}} \:
\left\{ 1 \:\, - \:\, \mu \left( x_0, \, u(\cdot), \, T_{\delta}(x_0, u(\cdot)) \right) \right. \\
& \qquad\qquad\quad
\left. + \:\, \mu \left( x_0, \, u(\cdot), \, T_{\delta}(x_0, u(\cdot)) \right) \: \cdot \:
v_0 \left(x \left( T_{\delta}(x_0, u(\cdot)); \: x_0, u(\cdot) \right) \right) \right\}.
\end{aligned}  \label{Eq_82}
\end{equation}
By using the nonnegativity of the running cost, as well as the formulas~(\ref{Eq_81}), (\ref{Eq_77}) and (\ref{Eq_78}), one
arrives at
\begin{equation}
\arraycolsep=1.5pt
\def\arraystretch{2}
\begin{array}{c}
0 \: < \: \mu(x_0, u(\cdot), T) \: \leqslant \: 1 \quad
\forall u(\cdot) \in \mathcal{U} \quad \forall T \geqslant 0, \\
v_{\delta}(x_0) \:\: = \:\: \inf\limits_{\substack{u(\cdot) \, \in \, \mathcal{U} \: \colon \\
T_{\delta}(x_0, \, u(\cdot)) \: < \: +\infty}} \:
\left\{ 1 \:\, - \:\, \mu \left( x_0, \, u(\cdot), \, T_{\delta}(x_0, u(\cdot)) \right) \right\}.
\end{array}  \label{Eq_83}
\end{equation}
Next, the obtained relations~(\ref{Eq_82}), (\ref{Eq_83}) lead to
$$
\begin{aligned}
0 \:\: & \leqslant \:\: v_0(x_0) \: - \: v_{\delta}(x_0) \\
& \leqslant \:\: \sup\limits_{\substack{u(\cdot) \, \in \, \mathcal{U} \: \colon \\
T_{\delta}(x_0, \, u(\cdot)) \: < \: +\infty}} \: \left\{
\mu \left( x_0, \, u(\cdot), \, T_{\delta}(x_0, u(\cdot)) \right) \: \cdot \:
v_0 \left(x \left( T_{\delta}(x_0, u(\cdot)); \: x_0, u(\cdot) \right) \right) \right\} \\
& \leqslant \:\: \sup\limits_{\substack{u(\cdot) \, \in \, \mathcal{U} \: \colon \\
T_{\delta}(x_0, \, u(\cdot)) \: < \: +\infty}} \:
v_0 \left(x \left( T_{\delta}(x_0, u(\cdot)); \: x_0, u(\cdot) \right) \right) \\
& \leqslant \:\: \max\limits_{y \, \in \, \bar{\mathrm{B}}_{\delta}(0_n)} \, v_0(y).
\end{aligned}
$$
In order to complete the proof, it now suffices to use the property
$$
\lim\limits_{\delta \, \to \, +0} \: \max\limits_{y \, \in \, \bar{\mathrm{B}}_{\delta}(0_n)} \, v_0(y) \:\, = \:\, 0,
$$
which follows from the equality~$ v_0(0_n) = 0 $ and the continuity of $ v_0(\cdot) $ on $ G $ mentioned in
Theorem~\ref{Thm_37}.
\end{proof}

\begin{remark}  \label{Rem_41}  \rm
In contrast with the global CLF characterization in Theorem~\ref{Thm_18} involving a local CLF, the approximating value
function~$ V_{\delta}(\cdot) $ for a fixed sufficiently small~$ \delta \in (0, r] $ leads to the so-called practical
stabilization in $ \mathcal{D}_0 $ (see, e.\,g., \cite[Subsection~2.11]{GieslHafstein2015}), but not to the asymptotic one.
Indeed, the exit-time optimal control problem~(\ref{Eq_75}) is stated without using a local CLF and therefore does not
allow to obtain stabilizing control actions in the target ball~$ \bar{\mathrm{B}}_{\delta}(0_n) $. \qed
\end{remark}

In order to ensure the existence of optimal control strategies and to use Pontryagin's principle for the exit-time
problem~(\ref{Eq_75}) with $ \, x_0 \, \in \, \mathcal{D}_0 \setminus \bar{\mathrm{B}}_{\delta}(0_n) \, $ and
$ \delta \in (0, r] $, we also need Assumptions~\ref{Ass_19} and \ref{Ass_21}. Note that the case when
$ \, x_0 \in \bar{\mathrm{B}}_{\delta}(0_n) \, $ with $ \delta \in (0, r] $ is trivial and yields
$ \: T_{\delta}(x_0, u(\cdot)) \, = \, 0 \: $ for all $ u(\cdot) \in \mathcal{U} $.

The existence result can be verified similarly to Theorem~\ref{Thm_22}.

\begin{theorem}  \label{Thm_42}
Let Assumptions~{\rm \ref{Ass_1}, \ref{Ass_32}, \ref{Ass_36}} and {\rm \ref{Ass_19}} hold. For any fixed initial
state~$ \, x_0 \, \in \, \mathcal{D}_0 \setminus \bar{\mathrm{B}}_{\delta}(0_n) \, $ and parameter~$ \delta \in (0, r] ${\rm ,}
there exists an optimal control strategy for the exit-time problem~{\rm (\ref{Eq_75})} or{\rm ,} equivalently{\rm ,} for
{\rm (\ref{Eq_77})}.
\end{theorem}

\begin{remark}  \label{Rem_43}  \rm
Under Assumptions~\ref{Ass_1}, \ref{Ass_32}, \ref{Ass_36}, \ref{Ass_19} and \ref{Ass_21}, Pontryagin's principle for
the exit-time problem~(\ref{Eq_75}) with a fixed initial
state~$ \, x_0 \, \in \, \mathcal{D}_0 \setminus \bar{\mathrm{B}}_{\delta}(0_n) \, $ and a fixed parameter~$ \delta \in (0, r] $
can be formulated similarly to Theorem~\ref{Thm_23}, but with the difference that now the terminal set appears as
the ball~$ \bar{\mathrm{B}}_{\delta}(0_n) $ and can be reduced to the sphere~$ \partial \mathrm{B}_{\delta}(0_n) $, while
the terminal cost vanishes. One should consequently have
$$
\arraycolsep=1.5pt
\def\arraystretch{1.5}
\begin{array}{c}
T^* \, = \, T_{\delta}(x_0, u^*(\cdot)), \quad \| x^*(T^*) \| \, = \, \delta, \\
p^*(T^*) \:\, \in \:\, \mathrm{N}(x^*(T^*); \, \bar{\mathrm{B}}_{\delta}(0_n)) \:\, = \:\,
\{ \varkappa \, x^*(T^*) \: \colon \: \varkappa \geqslant 0 \}
\end{array}
$$
in the modified characteristic boundary value problem. \qed
\end{remark}

The following characteristics based representation is established similarly to Theorem~\ref{Thm_26}. For the Hamiltonian and
set-valued extremal control map, the notations~(\ref{Eq_31}) and (\ref{Eq_36}) are still used.

\begin{theorem}  \label{Thm_44}
Let Assumptions~{\rm \ref{Ass_1}, \ref{Ass_32}, \ref{Ass_36}, \ref{Ass_19}} and {\rm \ref{Ass_21}} hold. For any initial
state~$ \, x_0 \, \in \, \mathcal{D}_0 \setminus \bar{\mathrm{B}}_{\delta}(0_n) \, $ and parameter~$ \delta \in (0, r] ${\rm ,}
the Kruzhkov transformed value~$ v_{\delta}(x_0) $ defined by {\rm (\ref{Eq_75})--(\ref{Eq_78})} is the minimum of
$$
1 \:\: - \:\: \exp \left\{ -\int\limits_0^{T_{\delta}(x_0, \, u^*(\cdot))} g(x^*(t), \, u^*(t)) \: \mathrm{d} t \right\}
$$
over the solutions of the characteristic Cauchy problems
$$
\left\{ \begin{aligned}
& \dot{x^*}(t) \:\: = \:\: \mathrm{D}_p H(x^*(t), \, u^*(t), \, p^*(t), \, \tilde{p}^*) \:\: = \:\:
f(x^*(t), \, u^*(t)), \\
& \dot{p^*}(t) \:\: = \:\: -\mathrm{D}_x H(x^*(t), \, u^*(t), \, p^*(t), \, \tilde{p}^*) \\
& \qquad \:\:
= \:\: -(\mathrm{D}_x f(x^*(t), \, u^*(t)))^{\top} \: p^*(t) \:\, - \:\,
\tilde{p}^* \: \mathrm{D}_x g(x^*(t), \, u^*(t)), \\
& u^*(t) \: \in \: U^*(x^*(t), \, p^*(t), \, \tilde{p}^*), \\
& t \:\, \in \:\, \begin{cases}
[0, \, T_{\delta}(x_0, \, u^*(\cdot))], & T_{\delta}(x_0, \, u^*(\cdot)) \: < \: +\infty, \\
[0, +\infty), & T_{\delta}(x_0, \, u^*(\cdot)) \: = \: +\infty,
\end{cases} \\
& x^*(0) \, = \, x_0, \quad p^*(0) \, = \, p_0,
\end{aligned} \right.
$$
for all extended initial adjoint vectors
$$
(p_0, \tilde{p}^*) \:\, \in \:\, \{ (p, \tilde{p}) \: \colon \: p \in \mathbb{R}^n, \:\: \tilde{p} \in \{ 0, 1 \} \}.
$$
Moreover{\rm ,} the same value is obtained when minimizing over the bounded set
$$
(p_0, \tilde{p}^*) \:\, \in \:\, \{ (p, \tilde{p}) \: \in \: \mathbb{R}^n \times \mathbb{R} \: \colon \:
\| (p, \tilde{p}) \| \: = \: 1, \:\:\, \tilde{p} \geqslant 0 \},
$$
or even over its subset
$$
(p_0, \tilde{p}^*) \:\, \in \:\, \{ (p, \tilde{p}) \: \in \: \mathbb{R}^n \times \mathbb{R} \: \colon \:
\| (p, \tilde{p}) \| \: = \: 1, \:\:\, \tilde{p} \geqslant 0, \:\:\,
\mathcal{H}(x_0, p, \tilde{p}) \: = \: 0 \}.
$$
\end{theorem}

Theorems~\ref{Thm_37}, \ref{Thm_40}, \ref{Thm_42}, \ref{Thm_44} together with Remarks~\ref{Rem_41}, \ref{Rem_43}
constitute the theoretical foundation of a curse-of-dimensionality-free approach to approximating CLFs and feedback
stabilization in case when one does find an appropriate local CLF.

\section{A curse-of-dimensionality-free approach to CLF approximation and feedback stabilization}

In the introduction, several well-known grid based numerical methods for solving Hamilton--Jacobi equations and
constructing optimal feedback strategies were noted. They typically require dense state space discretizations and
may face the practical dilemma of selecting a suitable bounded region for computations (in order to reduce boundary
cutoff errors in a relevant subdomain).

Alternatively, one can use the results of Sections~2 and 3 in order to approximate CLFs and associated feedbacks
independently at different initial states. As was discussed in the introduction, this enables for attenuating
the curse of dimensionality and selecting arbitrary bounded regions and grids in the state space. Parallel
computations can also be arranged.

Furthermore, the stabilizing control action at any isolated state can be directly retrieved either as the initial value of
an approximate optimal open-loop control strategy computed via a direct method, or by the corresponding representation in
Pontryagin's principle (recall (\ref{Eq_36})) with the initial state and an approximate optimal initial costate. The latter can be
obtained via an indirect characteristics based method or as an appropriate costate estimate building on direct
collocation~\cite{Benson2006,Garg2011}. Possibly unstable approximations of the gradient of the CLF are therefore not needed.

As was also noted in the introduction, even if the curse of dimensionality is mitigated, the curse of complexity is still
a formidable issue when constructing global or semi-global solution approximations in high-dimensional regions. Sparse
grid frameworks (see, e.\,g., \cite{KangWilcox2017,Garcke2012} and \cite[\S 3.7]{PressNR2007}) may help to attenuate that
if the dimension is not too high (typically not greater than $ 6 $) and if the sought-after functions are smooth enough.
However, the range of applicability of sparse grids to solving feedback control problems has to be further investigated.

More details and recommendations on implementing the curse-of-dimensionality-free approach are given in Section~A.2 of
\hyperlink{Online_App}
{the appendix}. They focus on the setting of Section~2 with a local CLF involved. Similar
practical considerations excluding local CLF construction can be applied to the setting of Section~3. However, further
development of efficient numerical algorithms with software implementations is left for future research.

Subsection~A.2.1 of \hyperlink{Online_App}
{the appendix} describes a linearization based numerical technique for
building quadratic local CLFs under some additional conditions, with the considerations of \cite[Section~3]{ChenAllgower1998}
serving as an important motivation. Those considerations can also be employed for constructing quadratic local CLFs under
the same assumptions. Although the technique presented in \hyperlink{Online_App}
{the appendix} is less elegant and
may be more computationally expensive, it is more straightforward to use and does not restrict the right-hand sides in
the decrease conditions for the resulting local CLFs necessarily to quadratic functions (in contrast to the approach of
\cite[Section~3]{ChenAllgower1998}).

In \hyperlink{Online_App}
{the appendix}, Subsection~A.2.2 develops a numerical framework using the characteristics
based representation in Theorem~\ref{Thm_26}, Subsection~A.2.3 briefly discusses the use of direct approximation methods,
and, finally, Subsection~A.2.4 points out how our curse-of-dimensionality-free approach can be incorporated in model
predictive control schemes and how sparse grids may be involved.

\section{Numerical simulations}

In this section, we consider two examples for testing certain implementaions of the discussed curse-of-dimensionality-free
approach to CLF approximation and feedback stabilization. The first example involves a nonlinear control system with
two-dimensional state space and can be treated analytically to some extent, so that the exact and numerical solutions
can be compared with each other. In the second example, a model predictive control scheme (see Subsection~A.2.4 in
\hyperlink{Online_App}
{the appendix}) incorporating that approach is applied to an essentially more complicated nonlinear
control system with six-dimensional state space.

The numerical simulations were conducted on a relatively weak machine with 1.4~GHz Intel~2957U CPU, and no parallel programming
tools were used. The runtimes can be significantly shorter for more powerful machines, especially when parallelization is done.

\begin{example}  \label{Exa_45}  \rm
Consider the control system~(\ref{Eq_1}) with $ \: n = 2 $, $ \, m = 1 $, $ \, x  = (x_1, x_2)^{\top} \in \mathbb{R}^2 $,
$ \, x(0) = x^0 \in \mathbb{R}^2 $, $ \, G = \mathbb{R}^2 $, $ \, U \subseteq \mathbb{R}, \: $ and
\begin{equation}
f(x, u) \:\, = \:\, \begin{pmatrix}
x_1 \, + \, 2 x_2 \, + \, u \\
-x_2 \, - \, 2 x_1^3
\end{pmatrix} \quad
\forall x \in \mathbb{R}^2 \quad \forall u \in U  \label{Eq_84}
\end{equation}
(see \cite[Example~1.1]{Shahmansoorian2009}).

First, let $ U = \mathbb{R} $. The proper, positive definite, and infinitely differentiable function
\begin{equation}
\tilde{V}(x) \:\, \stackrel{\mathrm{def}}{=} \:\, \frac{1}{4} \, x_1^4 \: + \: \frac{1}{2} \, x_2^2 \quad
\forall x \in \mathbb{R}^2  \label{Eq_85}
\end{equation}
is a global CLF for this system in the whole state space. Indeed,
$$
\begin{aligned}
& \left< \mathrm{D} \tilde{V}(x), \, f(x, u) \right> \:\: = \:\: x_1^3 \, (x_1 \, + \, 2 x_2 \, + \, u) \: + \:
x_2 \, \left( -x_2 \, - \, 2 x_1^3 \right) \:\:
= \:\: x_1^3 \, (x_1 \, + \, u) \: - \: x_2^2 \\
& \forall x \in \mathbb{R}^2 \quad \forall u \in \mathbb{R},
\end{aligned}
$$
so that, for any constant~$ b > 0 $ and for any bounded continuous function~$ \, \chi \colon \mathbb{R} \to \mathbb{R} \, $
satisfying
\begin{equation}
\chi(x_1) \, \geqslant \, b \quad \forall x_1 \in \mathbb{R},  \label{Eq_86}
\end{equation}
the feedback control strategy
\begin{equation}
\tilde{u}(x) \:\, \stackrel{\mathrm{def}}{=} \:\, -(1 \, + \,  \chi(x_1)) \: x_1 \quad \forall x \in \mathbb{R}^2  \label{Eq_87}
\end{equation}
is globally stabilizing due to
$$
\left< \mathrm{D} \tilde{V}(x), \, f(x, \tilde{u}(x)) \right> \:\: = \:\: -\chi(x_1) \, x_1^4 \: - \: x_2^2 \:\: \leqslant \:\:
-b x_1^4 \: - \: x_2^2 \quad \forall x \in \mathbb{R}^2.
$$
Introduce also the running cost
\begin{equation}
g(x, u) \:\, = \:\, 2 x_1^4 \: + \: x_2^2 \: + \: \frac{1}{256} \, u^4 \quad
\forall x \in \mathbb{R}^2 \quad \forall u \in \mathbb{R}.  \label{Eq_88}
\end{equation}
With the help of the classical verification result \cite[Chapter~VII, Theorem~2.2]{Afanasiev1996}, one can show that (\ref{Eq_85}) is
the value function in the infinite-horizon optimal control problem
\begin{equation}
\tilde{V} \left( x^0 \right) \:\: = \:\: \min_{u(\cdot) \:\, \in \:\, L_{\mathrm{loc}}^{\infty}([0, +\infty), \: \mathbb{R})} \:
\left\{ \int\limits_0^{+\infty} g \left( x \left( t; \, x^0, u(\cdot) \right), \: u(t) \right) \: \mathrm{d} t \right\} \quad
\forall x^0 \in \mathbb{R}^2  \label{Eq_89}
\end{equation}
for the considered system and running cost, and that
\begin{equation}
\tilde{u}^*(x) \, = \, -4 x_1  \label{Eq_90}
\end{equation}
is the corresponding optimal feedback control strategy.

Next, take the set of pointwise control constraints as the bounded line segment
\begin{equation}
U \, = \, [-a, a]  \label{Eq_91}
\end{equation}
with a constant~$ a > 0 $. Furthermore, fix a constant~$ b > 0 $ and a bounded continuous
function~$ \, \chi \colon \mathbb{R} \to \mathbb{R} \, $ satisfying (\ref{Eq_86}), and choose a constant~$ c > 0 $ such that
the feedback strategy~(\ref{Eq_87}) fulfills
\begin{equation}
\tilde{u}(x) \, \in \, U \, = \, [-a, a] \quad \forall x \in \Omega_c  \label{Eq_92}
\end{equation}
in the sublevel set
\begin{equation}
\Omega_c \:\, \stackrel{\mathrm{def}}{=} \:\, \{ x \in \Omega_c \: \colon \: \tilde{V}(x) \, \leqslant \, c \}  \label{Eq_93}
\end{equation}
of $ \tilde{V}(\cdot) $. Hence, the restrictions of $ \tilde{V}(\cdot) $ and $ \tilde{u}(\cdot) $ to $ \Omega_c $ are a local CLF and
a locally stabilizing feedback, respectively. Besides, there exists a constant~$ c^* > 0 $ such that the restriction of
$ \tilde{u}^*(\cdot) $ to $ \Omega_{c^*} $ is also a locally stabilizing feedback. For convenience, denote
\begin{equation}
V_{\mathrm{loc}}(x) \, \stackrel{\mathrm{def}}{=} \, \tilde{V}(x), \quad
u_{\mathrm{loc}}(x) \, \stackrel{\mathrm{def}}{=} \, \tilde{u}(x), \quad
u^*_{\mathrm{loc}}(x) \, \stackrel{\mathrm{def}}{=} \, \tilde{u}^*(x) \quad
\forall x \in \mathbb{R}^2.  \label{Eq_94}
\end{equation}
Now it is not difficult to check that Theorems~\ref{Thm_18}, \ref{Thm_22}, \ref{Thm_23}, \ref{Thm_25} and \ref{Thm_26}
can be used with the specified local CLF~$ V_{\mathrm{loc}}(\cdot) $, sublevel set~$ \Omega_c $, and running cost~(\ref{Eq_88}).
In particular, a global CLF~$ V(\cdot) $ in $ \mathcal{D}_0 $ is determined by (\ref{Eq_21}),~(\ref{Eq_25}), and its Kruzhkov
transform is given by (\ref{Eq_35}). Introduce also the Kruzhkov transform of $ V_{\mathrm{loc}}(\cdot) $:
\begin{equation}
v_{\mathrm{loc}} \left( x^0 \right) \:\, \stackrel{\mathrm{def}}{=} \:\, 1 \, - \, e^{-V_{\mathrm{loc}} \left( x^0 \right)} \quad
\forall x^0 \in \mathbb{R}^2.  \label{Eq_95}
\end{equation}

In order to approximate the global CLF and stabilizing feedback, as well as to construct a reasonable inner estimate for
the domain of asymptotic null-controllability, the characteristics based framework of Subsection~A.2.2 in
\hyperlink{Online_App}
{the appendix} was used. It is not difficult to verify the fulfillment of
Assumptions~A.2.3--A.2.5 from that subsection.

We take
\begin{equation}
\arraycolsep=1.5pt
\def\arraystretch{1.5}
\begin{array}{c}
a \, = \, 1.2, \quad b \, = \, 1.4, \quad c \, = \, 0.015, \quad c_1 \, = \, 0.01, \quad T_{\max} \, = \, 10, \\
\varepsilon \, = \, 10^{-15}, \quad \varepsilon_1 \, = \, 0.005, \quad \delta_1 \, = \, 0.005, \quad \delta_2 \, = \, 0.005
\end{array}  \label{Eq_96}
\end{equation}
(the notations~$ T_{\max}, \varepsilon, \varepsilon_1 $ and $ c_1, \delta_1, \delta_2 $ were introduced respectively in
Subsections~A.2.2.1 and A.2.2.2 of \hyperlink{Online_App}
{the appendix}),
\begin{equation}
\chi(x_1) \:\, = \:\, \begin{cases}
3, & -4 x_1 \, \in \, [-a, a], \\
\max \left\{ \dfrac{a}{|x_1|} \, - \, 1, \:\, b \right\}, & -4 x_1 \, \notin \, [-a, a],
\end{cases}  \label{Eq_97}
\end{equation}
so that $ \chi(\cdot) $ is bounded and continuous, (\ref{Eq_86}) holds,
\begin{equation}
\frac{a}{|x_1|} \: > \: 1 \, + \, b \quad \mbox{at all points~$ \, x \in \Omega_c \, $ for which $ x_1 \neq 0 $,}  \label{Eq_98}
\end{equation}
and (\ref{Eq_87}), (\ref{Eq_94}), (\ref{Eq_97}), (\ref{Eq_98}) imply the following relations:
\begin{equation}
\begin{aligned}
& u_{\mathrm{loc}}(x) \:\, = \:\, \tilde{u}(x) \:\, = \:\, \begin{cases}
u^*_{\mathrm{loc}}(x) \, = \, \tilde{u}^*(x) \, = \, -4 x_1, & -4 x_1 \, \in \, [-a, a], \\
-x_1 \: \max \left\{ \dfrac{a}{|x_1|} \, , \:\, 1 + b \right\}, & -4 x_1 \, \notin \, [-a, a],
\end{cases} \\
& u_{\mathrm{loc}}(x) \: = \: -a \,\, \mathrm{sign} \: x_1 \quad \mbox{at all points~$ \, x \in \Omega_c \, $ for which
$ \, -4 x_1 \, \notin \, [-a, a] $.}
\end{aligned}  \label{Eq_99}
\end{equation}
For a stabilizing feedback $ \, u^* \colon \mathcal{D}_0 \to U \, $ related to the CLF~$ V(\cdot) $, we put
\begin{equation}
u^*(x) \, = \, u_{\mathrm{loc}}(x) \quad \forall x \in \Omega_c.  \label{Eq_100}
\end{equation}

The characteristic Cauchy problems~(\ref{Eq_38}) and (A.22) (see Theorem~\ref{Thm_26} as well as Subsection~A.2.2.2 in
\hyperlink{Online_App}
{the appendix}) were numerically solved via the Dor\-mand--Prince fifth-order
Runge--Kutta algorithm from \cite[Chapter~17]{PressNR2007}. When launching the related routine, the initial guess for
the stepsize was specified as $ 2 \cdot 10^{-4} $, and the absolute and relative tolerances were selected as $ 10^{-6} $.
The output data was obtained for the uniform time grid on $ [0, T_{\max}] $ with the stepsize~$ 2 \cdot 10^{-4} $. In particular,
the shooting costs were approximated from the state trajectory discretizations on this time grid.

The initial states were chosen from the grid on the rectangle~$ \, [-2, 2] \times [-2.5, 2.5] \, $ with the spatial steps
$ 0.0625 $ and $ 0.1 $ along $ x_1 $-axis and $ x_2 $-axis, respectively. For solving the main and auxiliary finite-dimensional
optimization problems formulated in Theorem~\ref{Thm_26} and Subsection~A.2.2.2 of \hyperlink{Online_App}
{the appendix}, we used
the Powell algorithm from \cite[\S 10.7]{PressNR2007} (which does not require evaluation of derivatives), and the corresponding
tolerances were set as $ 10^{-6} $ and $ 10^{-8} $, respectively. For each state~$ x^0 $ on the selected rectangular grid,
the Powell iterative process for the auxiliary shooting problem was run from $ \: N_{\mbox{\scriptsize opt. init. guess}} \, = \, 4 \: $
initial guesses that were randomly generated according to the uniform distribution with respect to the angles in the unit sphere
parametrization (see (A.32) in \hyperlink{Online_App}
{the appendix}). The unique point~(A.26) and the roots of
the function~(A.29) (see these formulas in \hyperlink{Online_App}
{the appendix}) were computed via the bisection algorithm from
\cite[\S 9.1]{PressNR2007}, and the tolerance was taken as $ 10^{-13} $. An upper bound for the sought-after root~$ \lambda $ in
(A.26) was chosen as $ 0.5 $ (since $ \, \Omega_{c_1} \subset \mathrm{B}_{0.5}(0_2) $). Possible multiple roots of
the function~(A.29) on $ (0, 1) $ were bracketed by using the \texttt{zbrak} routine from \cite[\S 9.1]{PressNR2007}.
The bracketing pairs were searched for after dividing the interval~$ [0, 1] $ into $ 100 $ equally spaced segments.

The tolerance for the practical verification of equalities and non-strict inequalities via strict inequalities was set as
$ 10^{-15} $.

In the case when we could not find a characteristic reaching the target set~$ \Omega_c $ and generating a cost less than
$ 1 - \varepsilon $ for some initial state~$ x^0 $ (this might be not only a node on the specified rectangular grid, but also
a shooting estimate~$ \hat{x}^0 $ as described in the end of Subsection~A.2.2.2 in \hyperlink{Online_App}
{the appendix}),
the computation was rerun with the increased parameter values $ \: T_{\max} \, = \, T_{\max, \: \mathrm{recomp.}} \, = \, 20 \: $ and
$ \: N_{\mbox{\scriptsize opt. init. guess}} \, = \, N_{\mbox{\scriptsize opt. init. guess, recomp.}} \, = \, 5 $. In such
situations, the resulting data was taken from the second attempt ($ v \left( x^0 \right) $ might again be estimated as
$ 1 - \varepsilon $, which would indeed be reasonable if $ x^0 \notin \mathcal{D}_0 $).

The related numerical simulation results are illustrated in Figs.~\ref{Fig_3}--\ref{Fig_5}.

Fig.~\ref{Fig_3} indicates the Kruzhkov transformed functions~$ v(\cdot), v_{\mathrm{loc}}(\cdot) $ and their difference.
Fig.~\ref{Fig_4} shows the corresponding feedback strategies~$ u^*(\cdot), u^*_{\mathrm{loc}}(\cdot) $ and their difference.
In Fig.~\ref{Fig_3}, some approximated level sets of $ v(\cdot) $ are depicted as well. In particular, the level
$ \, v(x) = 1 - e^{-c} \, $ (or, equivalently, $ V(x) = c $) describes the boundary $ l_c = \partial \Omega_c $, while the level
$ \, v(x) = 1 - \varepsilon_1 = 0.995 \, $ is selected to represent an inner estimate of the domain of asymptotic
null-controllability~$ \mathcal{D}_0 $. The illustrations agree with the reasonable expectation that $ v(\cdot) $ and
$ v_{\mathrm{loc}}(\cdot) $ should coincide in some region strictly containing $ \Omega_c $, and that
$ \, u^*(x) = u^*_{\mathrm{loc}}(x) \, $ for all $ x $ lying in this region and satisfying
$ \, u^*_{\mathrm{loc}}(x) \in U = [-a, a] \, $ (due to (\ref{Eq_99}) and (\ref{Eq_100}), one also has
$ \, u^*(x) = u_{\mathrm{loc}}(x) \, $ for all such $ x $ and everywhere in $ \Omega_c $).

\begin{figure}
\begin{center}
\includegraphics[width=0.32\linewidth]{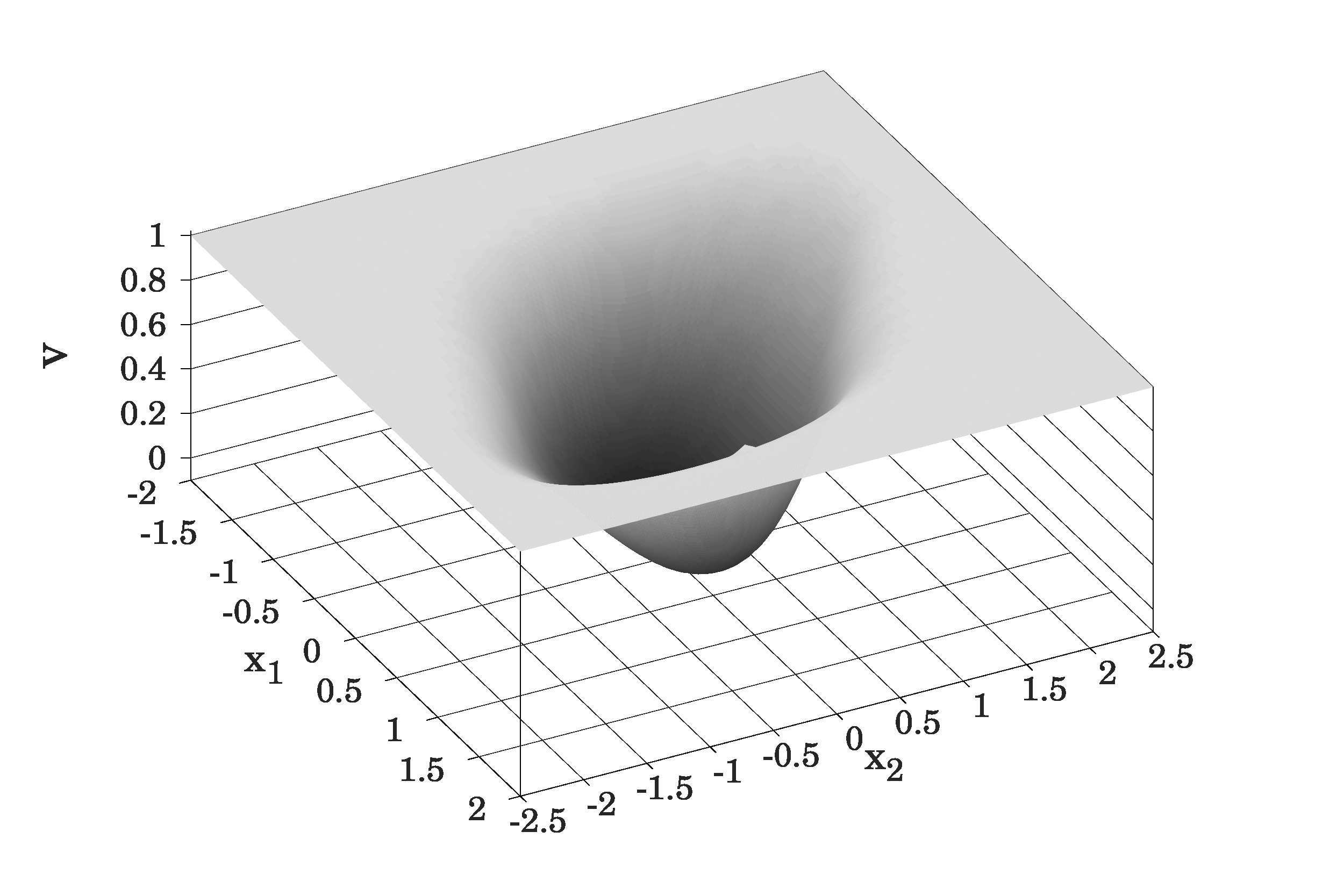}
\includegraphics[width=0.32\linewidth]{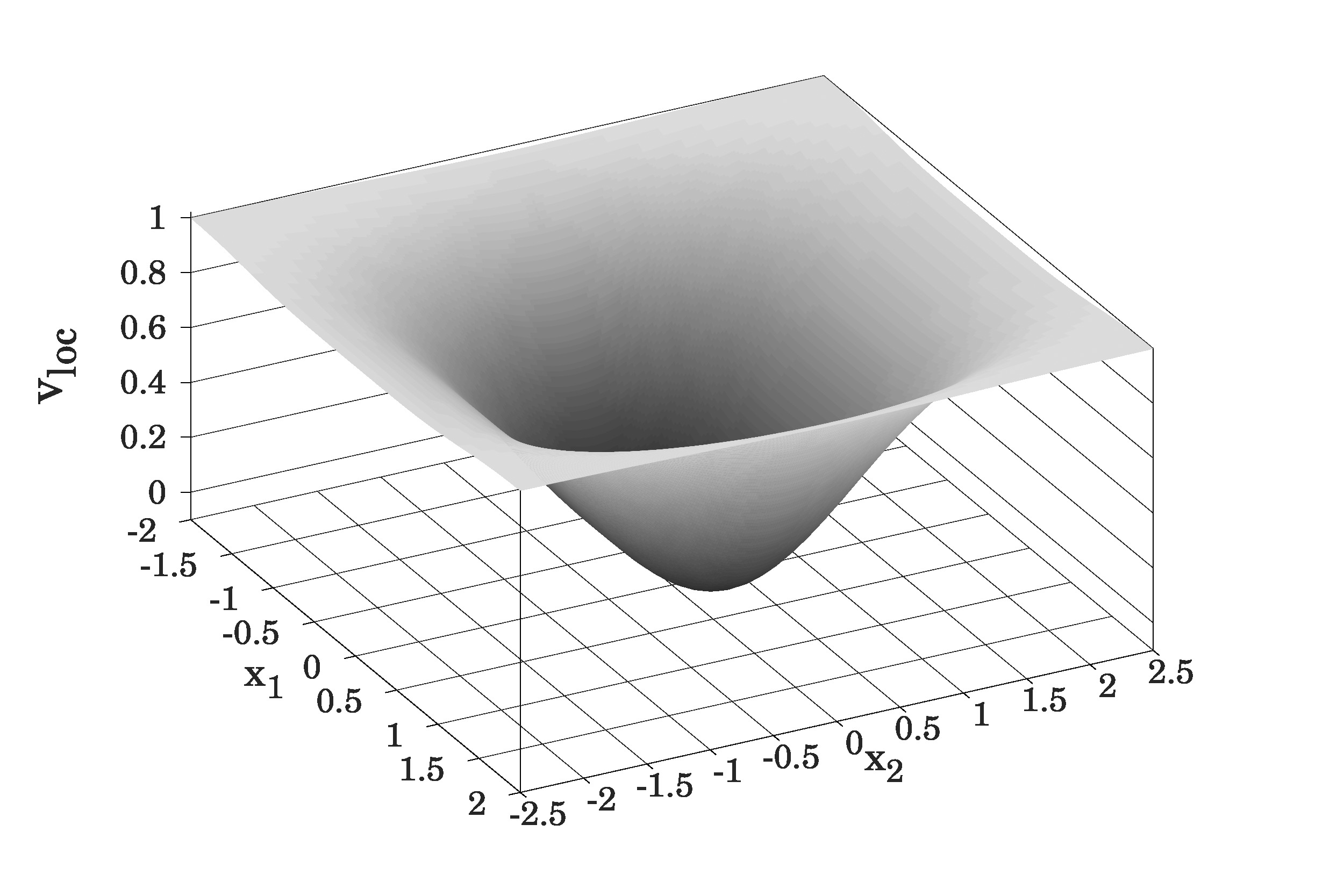}
\includegraphics[width=0.32\linewidth]{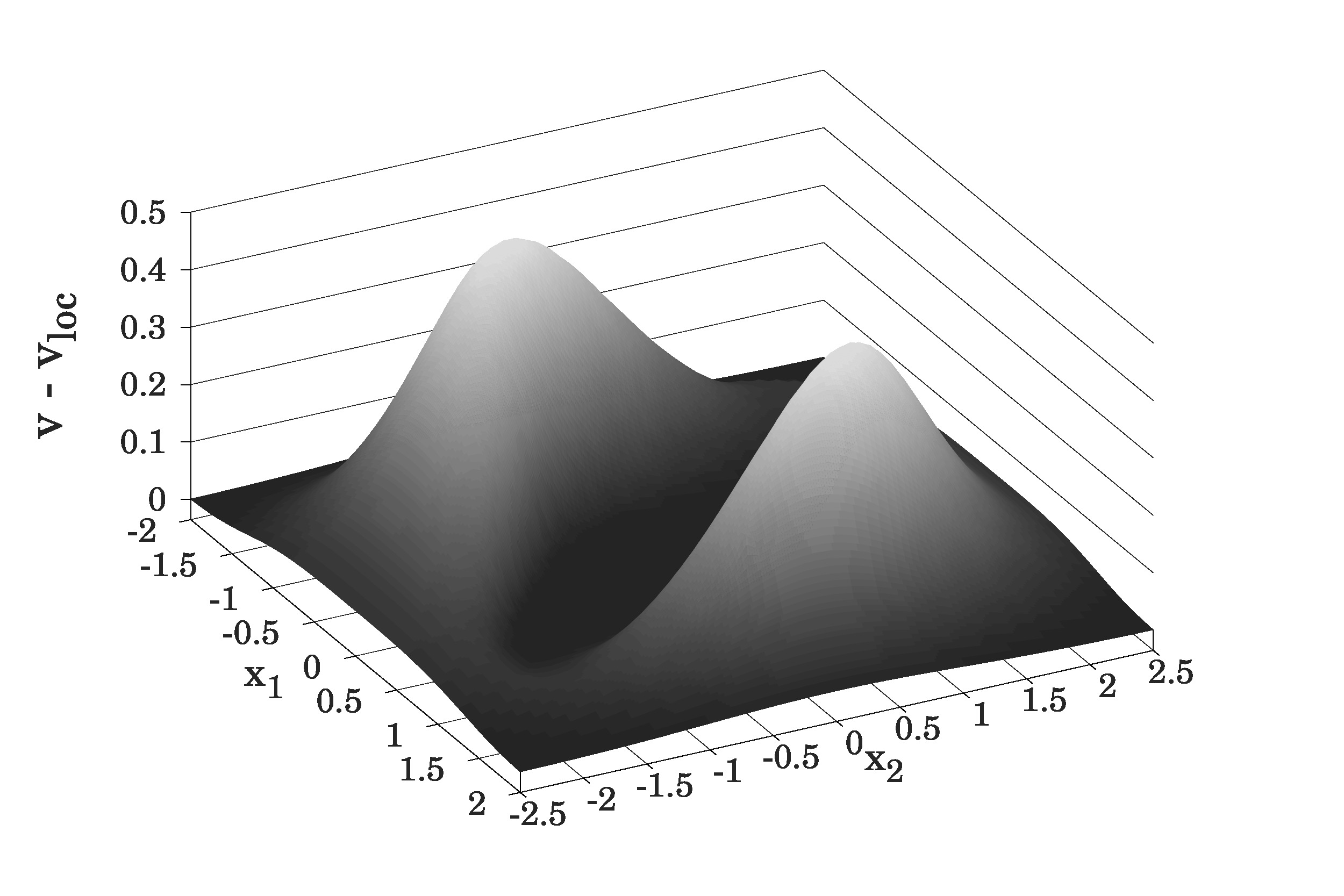} \\
\includegraphics[width=0.5\linewidth]{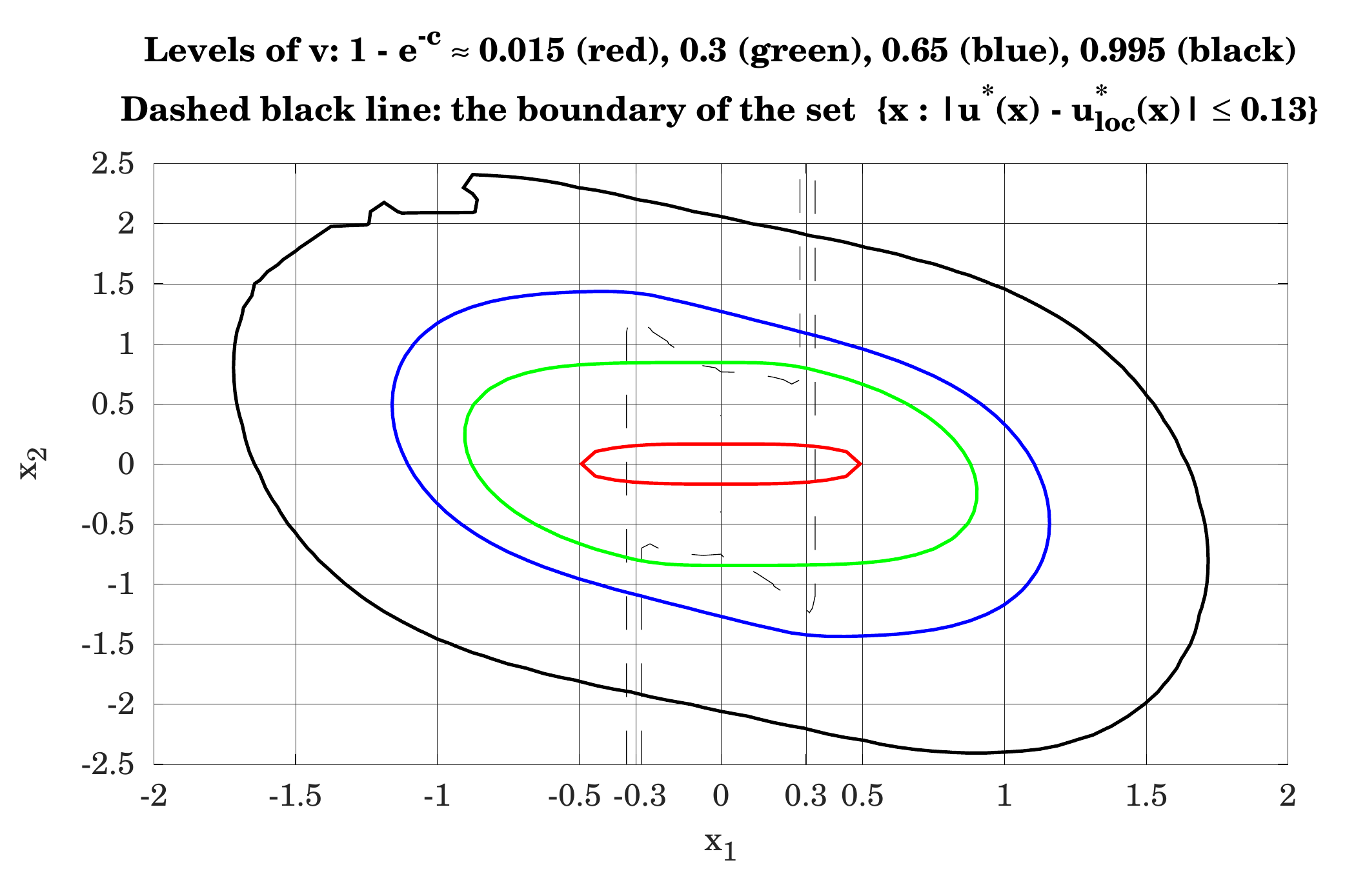}
\end{center}
\bf \caption{\rm The Kruzhkov transformed functions~$ v(\cdot), v_{\mathrm{loc}}(\cdot) $ and their difference for
$ U = [-1.2, 1.2] $ in Example~\ref{Exa_45}. Some approximated level sets of $ v(\cdot) $ are shown as well. In order to see
the graph of the difference between $ v(\cdot) $ and $ v_{\mathrm{loc}}(\cdot) $ clearer, the scale of the vertical axis
in the third subfigure is modified as compared to that in the first two subfigures.}
\label{Fig_3}
\end{figure}

\begin{figure}
\begin{center}
\includegraphics[width=0.32\linewidth]{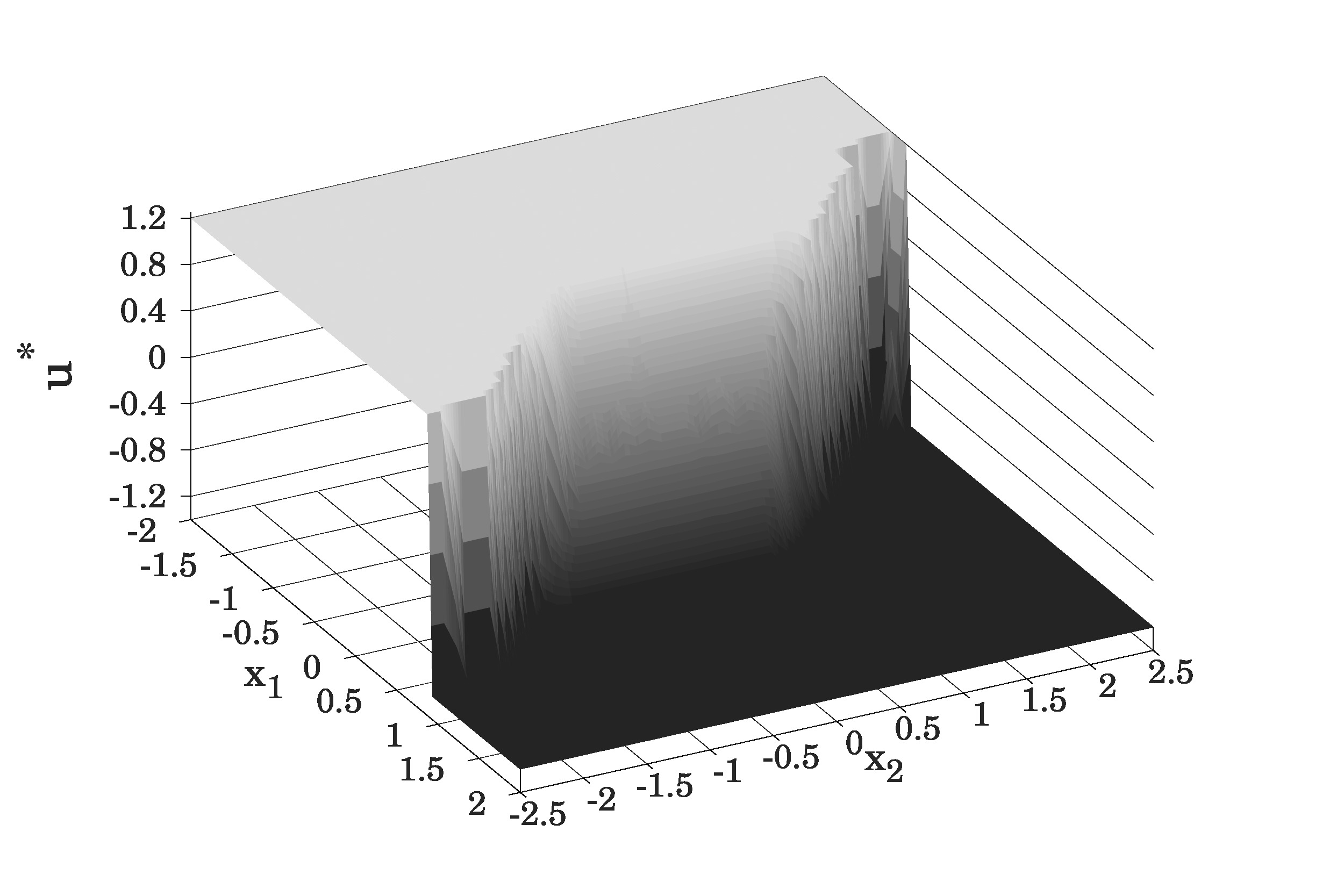}
\includegraphics[width=0.32\linewidth]{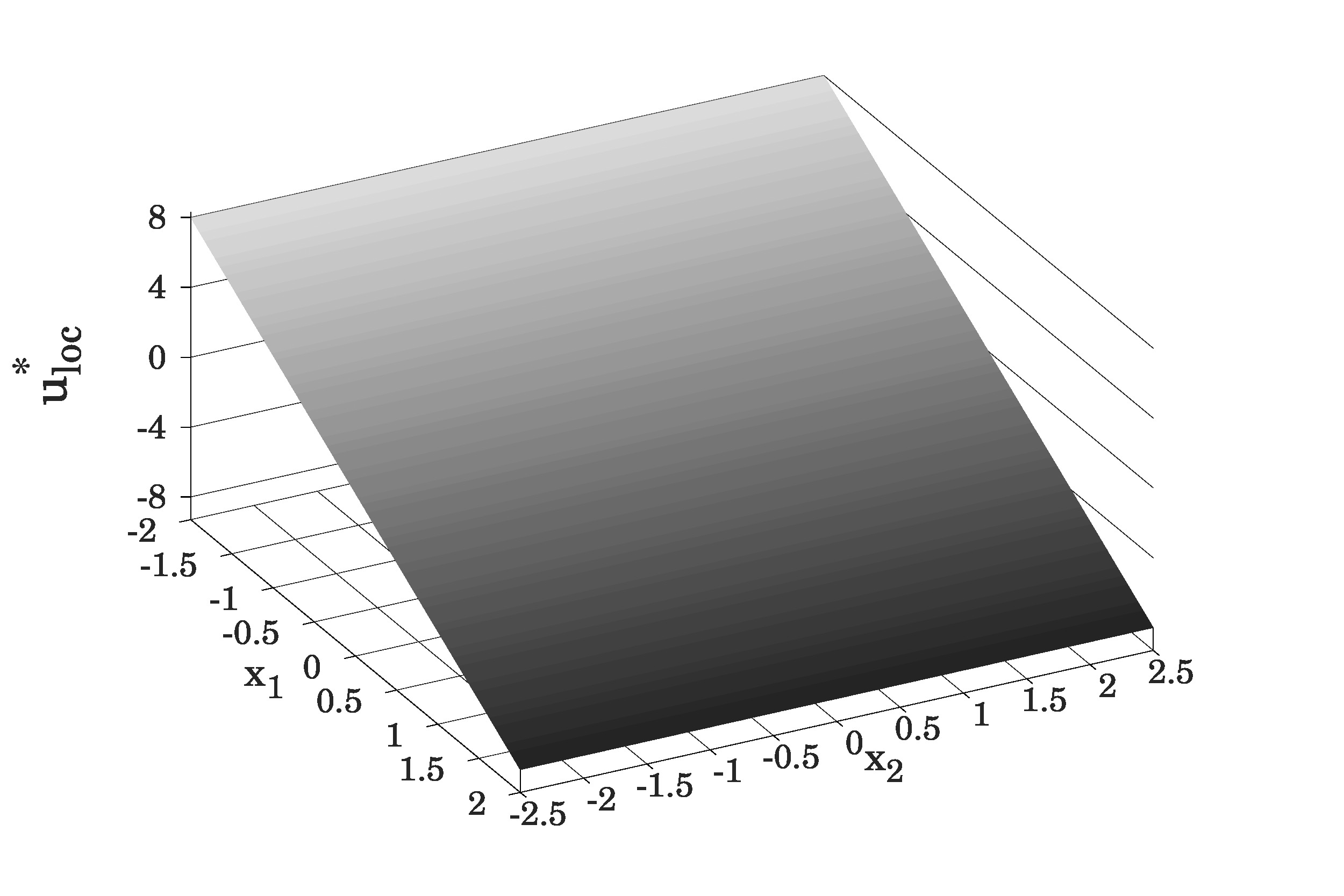}
\includegraphics[width=0.32\linewidth]{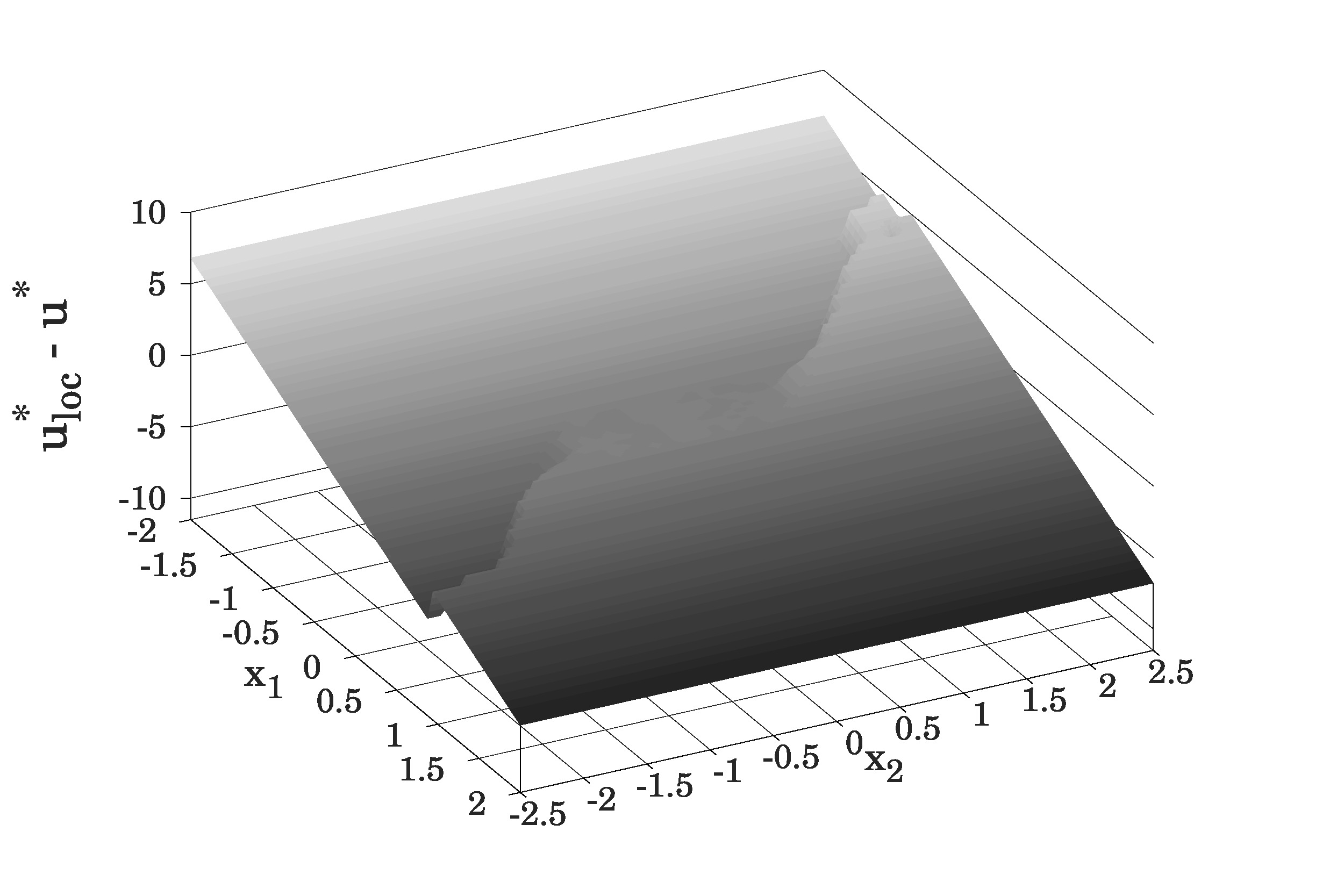}
\end{center}
\bf \caption{\rm The feedback control strategies~$ u^*(\cdot), u^*_{\mathrm{loc}}(\cdot) $ (corresponding to
$ v(\cdot), v_{\mathrm{loc}}(\cdot) $, respectively) and their difference for $ U = [-1.2, 1.2] $ in Example~\ref{Exa_45}.
In order to see the graphs clearer, we do not fix the same scale for the vertical axes in the subfigures.}
\label{Fig_4}
\end{figure}

Fig.~\ref{Fig_5} shows the graphs of the following functions:
\begin{itemize}
\setlength\itemsep{0em}
\item  the shooting state error defined as the square root of the numerical estimate of the minimum quadratic shooting cost
(see Subsection~A.2.2.2 in \hyperlink{Online_App}
{the appendix} and note that, even when the exact minimum value of
the lowest deviation~(A.31) is zero, its approximation does not vanish for $ x^0 \notin \Omega_c $);
\item  the shooting time defined as an approximate minimizer in (A.31) for an optimal shooting reverse-time
characteristic;
\item  the shooting value replacement indicator defined as zero if one arrives at a value less than $ 1 - \varepsilon $ after
one or two attempts to compute $ v \left( x^0 \right) $, and as the absolute difference
$ \, \left| v_1 \left( x^0 \right) - v \left( \hat{x}^0 \right) \right| \, $ in the other case when one uses the first-order
estimation technique proposed in the end of Subsection~A.2.2.2.
\end{itemize}
The shooting state error and the shooting value replacement indicator on the considered grid are small enough to conjecture that
the whole rectangle is contained in the domain of asymptotic null-controllability~$ \mathcal{D}_0 $. However, rigorous
verification of that for the selected bounded control constraint set $ U = [-1.2, 1.2] $ remains an open problem. Another
open question is whether $ \mathcal{D}_0 $ is bounded for a bounded $ U $ or not. Nevertheless, inner estimates of $ \mathcal{D}_0 $
obtained via our numerical approach may often suit practical needs.

\begin{figure}
\begin{center}
\includegraphics[width=0.32\linewidth]{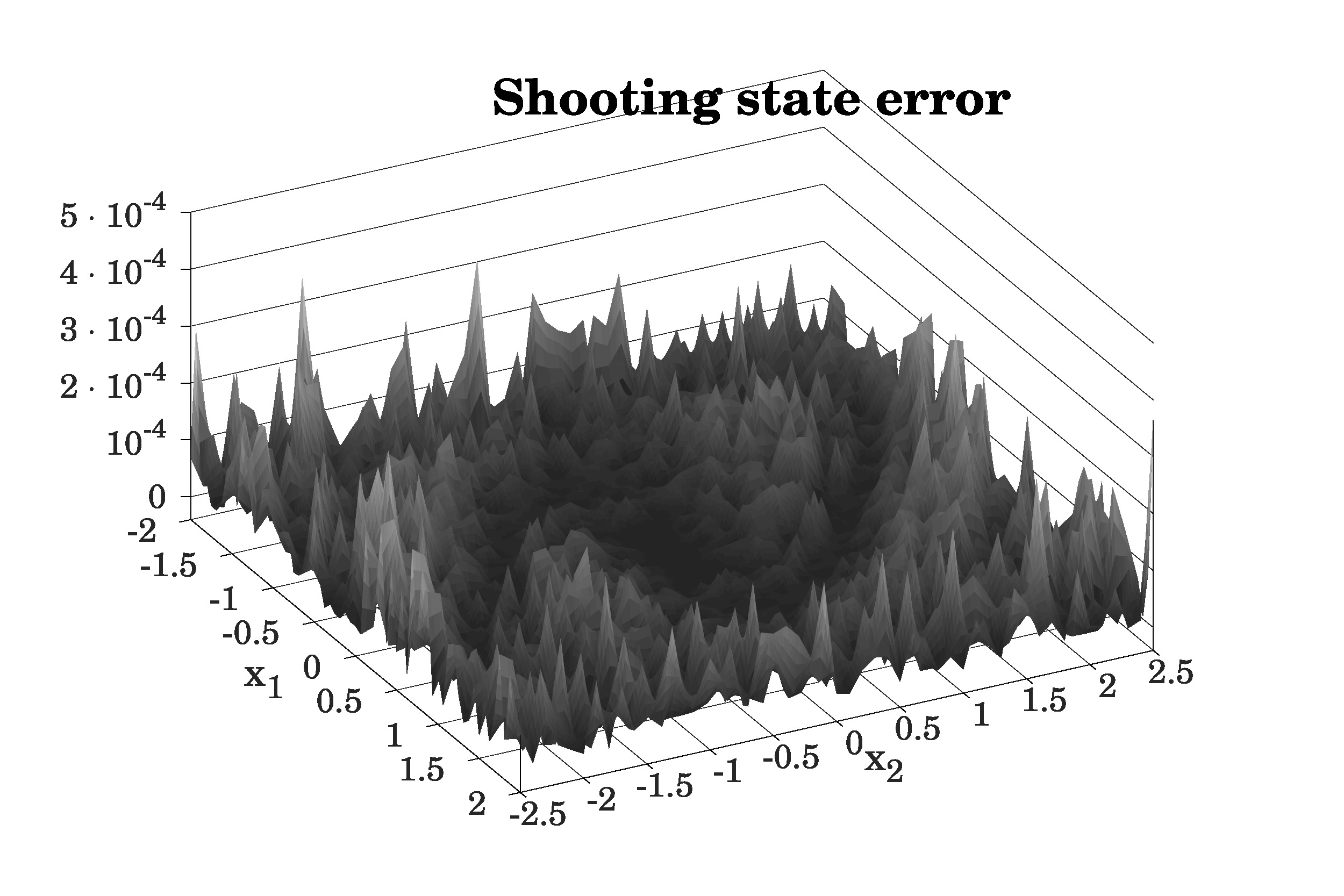}
\includegraphics[width=0.32\linewidth]{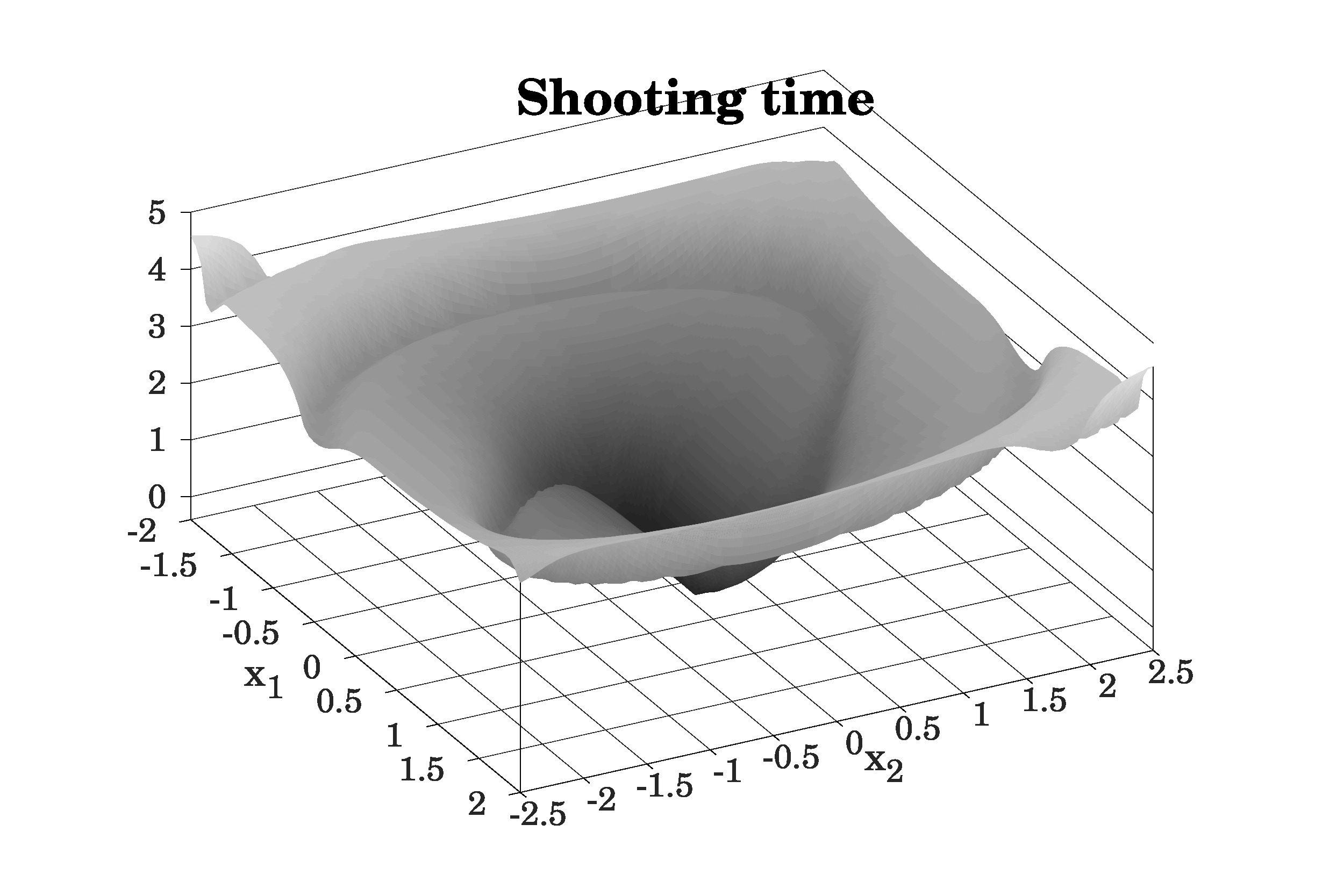}
\includegraphics[width=0.32\linewidth]{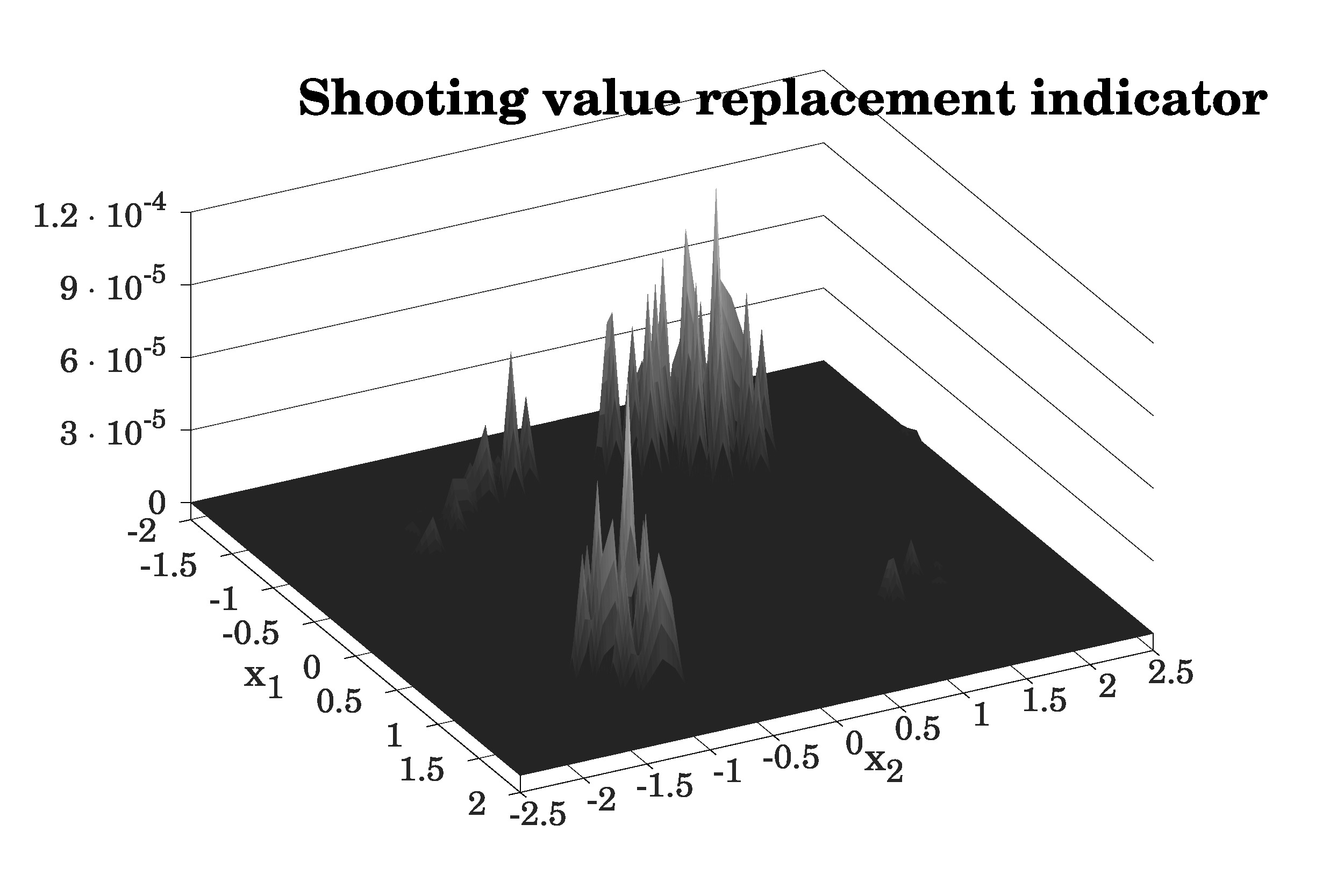}
\end{center}
\bf \caption{\rm The shooting state error, shooting time, and shooting value replacement indicator for $ U = [-1.2, 1.2] $
in Example~\ref{Exa_45}. Different scales are used for the vertical axes in the subfigures.}
\label{Fig_5}
\end{figure}

In order to practically check the obvious fact that, for a sufficiently large $ a $ and $ U = [-a, a] $, the functions
$ v(\cdot) $ and $ u^*(\cdot) $ should coincide in the considered bounded rectangle with $ v_{\mathrm{loc}}(\cdot) $ and
$ u^*_{\mathrm{loc}}(\cdot) $, respectively, we performed numerical simulations for the increased parameter value~$ a = 20 $.
We also reduced $ T_{\max} $ from $ 10 $ to $ 5 $. Moreover, the spatial steps along $ x_1 $-axis and $ x_2 $-axis for the grid
on the rectangle $ \, [-2, 2] \times [-2.5, 2.5] \, $ were increased to $ 0.1 $ and $ 0.15625 $, respectively. All other parameters
kept their values (in particular, $ \: T_{\max, \: \mathrm{recomp.}} \, = \, 20 \: $ remained the same). The results are illustrated
in Fig.~\ref{Fig_6}.

\begin{figure}
\begin{center}
\includegraphics[width=0.32\linewidth]{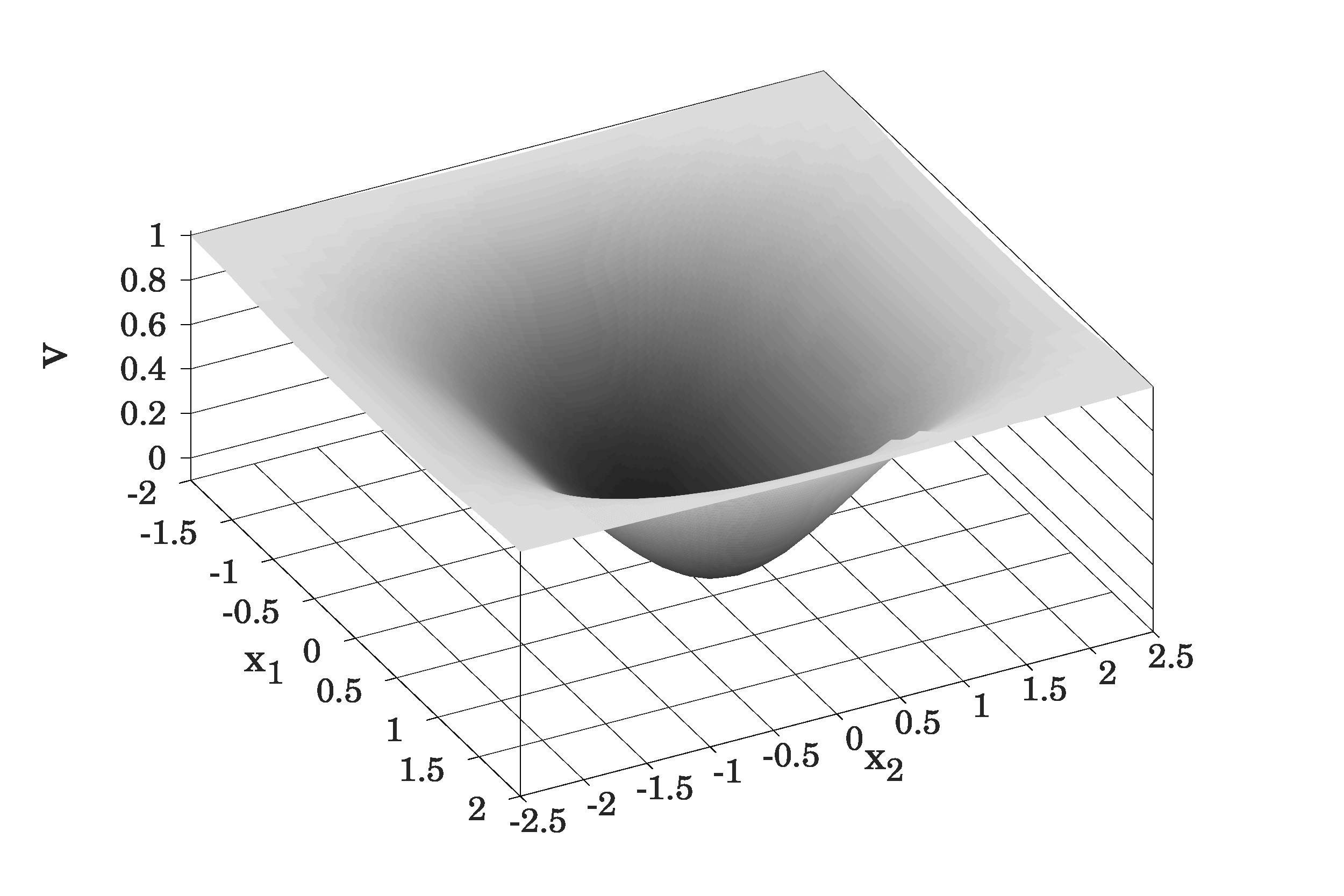}
\includegraphics[width=0.32\linewidth]{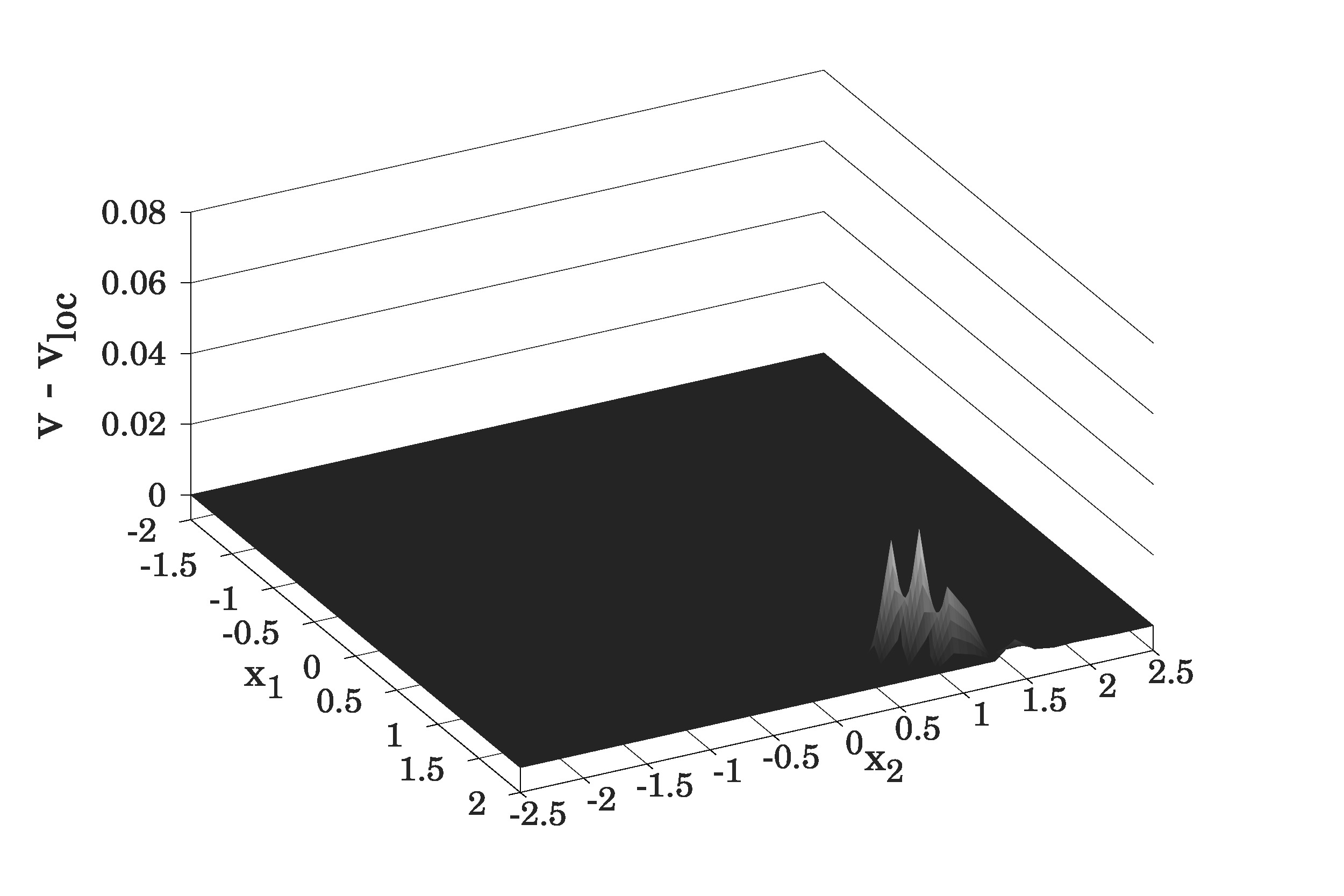} \\
\includegraphics[width=0.32\linewidth]{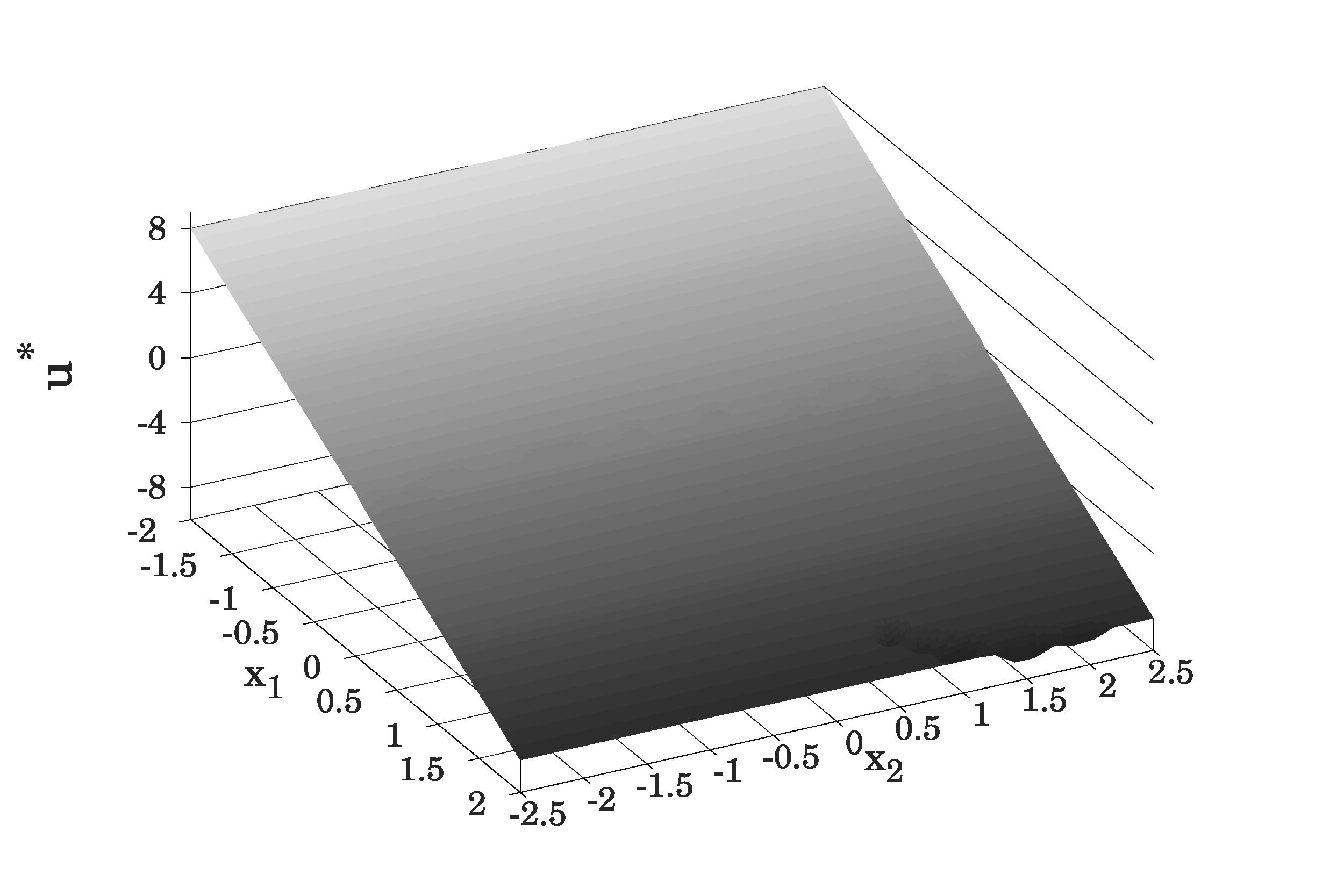}
\includegraphics[width=0.32\linewidth]{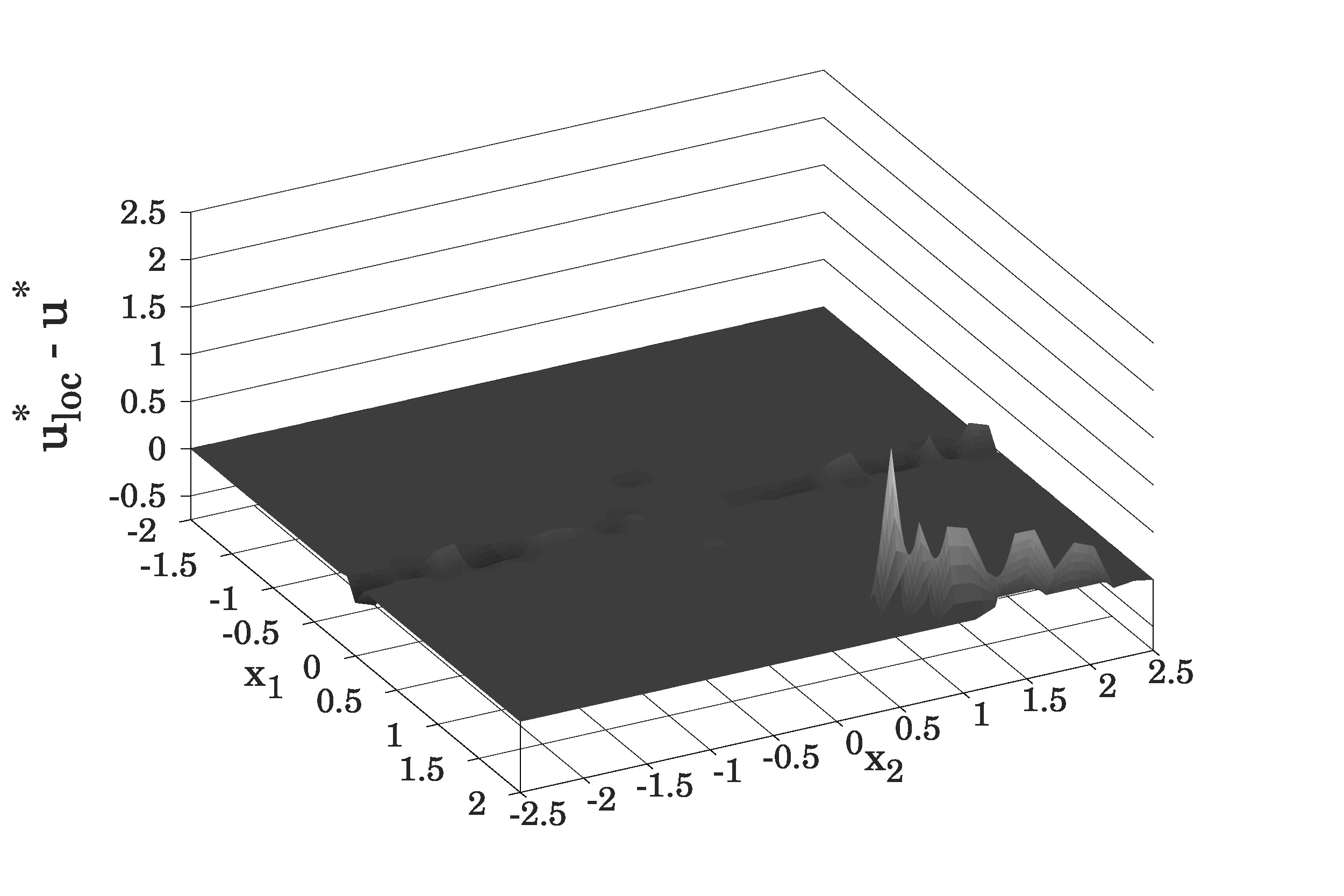}
\end{center}
\bf \caption{\rm The Kruzhkov transformed function~$ v(\cdot) $, the difference between $ v(\cdot) $ and
$ v_{\mathrm{loc}}(\cdot) $, the feedback control strategy~$ u^*(\cdot) $, and the difference between
$ u^*_{\mathrm{loc}}(\cdot) $ and $ u^*(\cdot) $ for $ U = [-20, 20] $ in Example~\ref{Exa_45}. In order to see the graphs
clearer, we do not fix the same scale for the vertical axes in the subfigures.}
\label{Fig_6}
\end{figure}

The average runtime per one initial state was around 18~seconds when obtaining the data for Figs.~\ref{Fig_3}--\ref{Fig_5}
and around 11~seconds when obtaining the data for Fig.~\ref{Fig_6}. Such relatively long runtimes can be explained as follows.
First, we used a rather weak machine, as was already noted in the beginning of this section. Second, our rectangle in the state
space was large enough for constructing a reasonable inner estimate of $ \mathcal{D}_0 $, and the runtimes for the grid nodes
not very far from $ \Omega_c $ were much shorter than the average runtime (for $ x^0 \in \Omega_c $, we put
$ \: V \left( x^0 \right) = V_{\mathrm{loc}} \left( x^0 \right) $,
$ \, u^* \left( x^0 \right) = u_{\mathrm{loc}} \left( x^0 \right) \: $ and did not even need to optimize).  \qed
\end{example}

\begin{example}  \label{Exa_46}  \rm
The dynamics of a planar vertical takeoff and landing (PVTOL) aircraft can be described by the control system
\cite{Fantoni2002,FantoniIFAC2002,Fantoni2006,Hably2006}
\begin{equation}
\left\{ \begin{aligned}
& \dot{x}_1(t) \:\: = \:\: x_2(t), \\
& \dot{x}_2(t) \:\: = \:\: -(1 + u_1(t)) \, \sin \, x_5(t) \,\, + \,\, \alpha \, u_2(t) \, \cos \, x_5(t), \\
& \dot{x}_3(t) \:\: = \:\: x_4(t), \\
& \dot{x}_4(t) \:\: = \:\: (1 + u_1(t)) \, \cos \, x_5(t) \,\, + \,\, \alpha \, u_2(t) \, \sin \, x_5(t) \,\, - \,\, 1, \\
& \dot{x}_5(t) \:\: = \:\: x_6(t), \\
& \dot{x}_6(t) \:\: = \:\: u_2(t), \\
& t \geqslant 0, \\
& x \: = \: (x_1, x_2, x_3, x_4, x_5, x_6)^{\top}, \quad u \: = \: (u_1, u_2)^{\top}, \quad n = 6, \quad m = 2, \\
& x(0) \, = \, x^0 \, \in \, G \, = \, \mathbb{R}^6, \\
& u(\cdot) \: \in \: \mathcal{U} \: \stackrel{\mathrm{def}}{=} \: L^{\infty}_{\mathrm{loc}}([0, +\infty), \, U), \quad
U \subseteq \mathbb{R}^2,
\end{aligned} \right.  \label{Eq_101}
\end{equation}
where the following notation is used (see Fig.~\ref{Fig_7}):
\begin{itemize}
\setlength\itemsep{0em}
\item  $ t $ is a time variable;
\item  $ x_1 $ and $ x_3 $ are normalized quantities that correspond to the horizontal and vertical coordinates of
the center of mass of the aircraft in a fixed inertial frame;
\item  $ x_5 $ is the roll angle that the aircraft makes with the positive horizontal axis;
\item  $ x_2 $, $ x_4 $, and $ x_6 $ are the rates of change of $ x_1 $, $ x_3 $, and $ x_5 $, respectively;
\item  $ u_1 $ and $ u_2 $ are normalized control inputs such that $ 1 + u_1 $ corresponds to the thrust (directed out
the bottom of the aircraft), $ u_2 $ is related to the angular acceleration (rolling moment), and the origin~$ x = 0_6 $
is a steady state for $ u = 0_2 $;
\item  the term $ -1 $ in the fourth dynamical equation represents the normalized gravitational acceleration;
\item  $ \alpha > 0 $ is a constant coefficient that characterizes the coupling between the rolling moment and
the lateral acceleration of the aircraft.
\end{itemize}

\begin{figure}
\begin{center}
\includegraphics[width=0.25\linewidth]{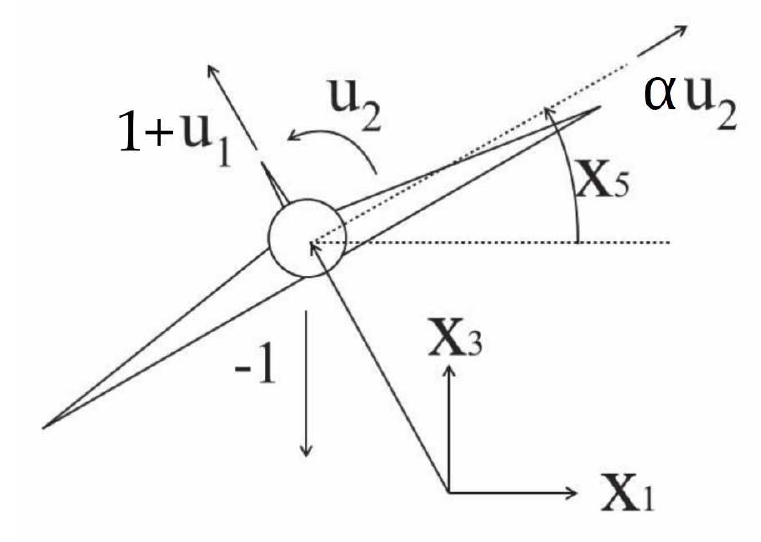}
\end{center}
\bf \caption{\rm The PVTOL aircraft system in Example~\ref{Exa_46}.}
\label{Fig_7}
\end{figure}

Let $ a_1, a_2 $ be positive constants and consider the compact convex control constraint set
\begin{equation}
U \: = \: [-a_1, a_1] \, \times \, [-a_2, a_2].  \label{Eq_102}
\end{equation}
Introduce also the quadratic running cost
\begin{equation}
g(x, u) \:\, = \:\, \frac{\lambda_1}{2} \, \| x \|^2 \: + \: \frac{\lambda_2}{2} \, \| u \|^2 \quad
\forall x \in \mathbb{R}^6 \quad \forall u \in \mathbb{R}^2  \label{Eq_103}
\end{equation}
with positive constants~$ \lambda_1, \lambda_2 $. It is not difficult to verify that a quadratic local
CLF~$ V_{\mathrm{loc}}(\cdot) $ can be constructed via linearization as described in Remark~A.2.2 of Subsection~A.2.1
in \hyperlink{Online_App}
{the appendix} and that Theorems~\ref{Thm_18}, \ref{Thm_22}, \ref{Thm_23}, \ref{Thm_25}
and \ref{Thm_26} can be used with the running cost~(\ref{Eq_103}) and with the mentioned local CLF.

We take
\begin{equation}
\alpha \, = \, 0.1, \quad a_1 \, = \, a_2 \, = \, 5, \quad \lambda_1 \, = \, 0.2, \quad \lambda_2 \, = \, 0.04.
\label{Eq_104}
\end{equation}

The algebraic Riccati equation (see (A.14) in \hyperlink{Online_App}
{the appendix}) was numerically solved via the \texttt{care}
routine in the GNU Octave environment (the \texttt{care} routine in the MATLAB environment can be used as well). The range of appropriate
levels~$ c $ was approximated with the help of the related recommendations in Subsection~A.2.1 of \hyperlink{Online_App}
{the appendix}
(in order to handle possible multi-extremality, $ 20 $ initial guesses were randomly generated for each of the corresponding
finite-dimensional optimization problems). We finally selected
\begin{equation}
c \, = \, 0.017.  \label{Eq_105}
\end{equation}

However, when trying to compute the CLF for the considered six-dimensional (three-degree-of-freedom) system at some
states even not far from $ \Omega_c $ via a characteristics based implementation similar to that used in Example~\ref{Exa_45},
we faced huge difficulties in achieving a suitable shooting accuracy in the auxiliary problem for the reverse-time
characteristics (even with the implicit Rosenbrock scheme~\cite[\S 17.5.1]{PressNR2007} used instead of the explicit
Runge--Kutta scheme for numerical integration of ODEs). We hence used the ACADO Toolkit~\cite{Houska2011,ACADO_Manual}
implementing a direct approximation method for optimal open-loop control problems. The other software packages mentioned in
Subsection~A.2.3 of \hyperlink{Online_App}
{the appendix} involve more advanced and efficient direct collocation
techniques and could also be applied. The ACADO Toolkit was chosen due to its relative simplicity, and also because its
capabilities were enough for the purposes of this example. Regarding the characteristics based framework of Subsection~A.2.2
in \hyperlink{Online_App}
{the appendix}, it may help to numerically treat the current example if its implementation is
modified in order to involve also multiple shooting or indirect collocation as applied to the characteristic system (see
the general discussion of these techniques, e.\,g., in \cite{Rao2010,Trelat2012}), but we leave that for future investigation.

We launched the ACADO Toolkit with the multiple shooting option, the maximum time horizon~$ 20 $, the tolerance~$ 10^{-6} $
for the default Runge--Kutta integrator, the Karush--Kuhn--Tucker tolerance~$ 10^{-4} $ (involved in the practical convergence
criterion for the sequential quadratic programming algorithm), and with $ 40 $ control intervals (these were time subintervals of
equal length, and the constrained numerical optimization was performed over piecewise constant control strategies which might
switch only at the endpoints of the subintervals).

For testing the performance and robustness of the MPC algorithm formulated in the beginning of Subsection~A.2.4 in
\hyperlink{Online_App}
{the appendix}, we also consider a stochastic perturbation of the system~(\ref{Eq_101}).
The noise is included in the second, fourth, and sixth dynamical equations (describing the accelerations for the three
degrees of freedom). Let us write the resulting system:
\begin{equation}
\left\{ \begin{aligned}
& \dot{x}_1(t) \:\: = \:\: x_2(t), \\
& \mathrm{d} x_2(t) \:\: = \:\: (-(1 + u_1(t)) \, \sin \, x_5(t) \,\, + \,\, \alpha \, u_2(t) \, \cos \, x_5(t)) \:
\mathrm{d} t \:\, + \:\, \sigma_2 \, \mathrm{d} w_2(t), \\
& \dot{x}_3(t) \:\: = \:\: x_4(t), \\
& \mathrm{d} x_4(t) \:\: = \:\: ((1 + u_1(t)) \, \cos \, x_5(t) \,\, + \,\, \alpha \, u_2(t) \, \sin \, x_5(t) \,\, - \,\, 1) \:
\mathrm{d} t \:\, + \:\, \sigma_4 \, \mathrm{d} w_4(t), \\
& \dot{x}_5(t) \:\: = \:\: x_6(t), \\
& \mathrm{d} x_6(t) \:\: = \:\: u_2(t) \, \mathrm{d} t \:\, + \:\, \sigma_6 \, \mathrm{d} w_6(t), \\
& t \geqslant 0, \\
& x(0) \, = \, x^0 \, \in \, \mathbb{R}^6.
\end{aligned} \right.  \label{Eq_106}
\end{equation}
Here $ x^0 $ is a deterministic initial state, $ \sigma_2 $, $ \sigma_4 $, and $ \sigma_6 $ are nonnegative constants
(noise intensity parameters), $ \, (w_2(\cdot), w_4(\cdot), w_6(\cdot)) \, $ is a three-dimensional standard Brownian
motion (Wiener process) on the time interval~$ [0, +\infty) $, and the stochastic ordinary differential equations are
understood in the It\^o sense. An open-loop control strategy can also represent a stochastic process if it is obtained
from a closed-loop map. Let us assess the control performance (quality) on a finite time interval~$ [0, T] $ through
the mean value
\begin{equation}
\mathbb{E} \left[ \int\limits_0^T \| x(t) \| \: \mathrm{d} t \right] \:\, = \:\,
\int\limits_0^T \mathbb{E} \, \| x(t) \| \: \mathrm{d} t.  \label{Eq_107}
\end{equation}
The lower this value, the higher the control quality. The control goal is therefore interpreted as mitigating the random
vibrations whose strength on $ [0, T] $ is given by (\ref{Eq_107}).

We select
\begin{equation}
x^0 \: = \: \left( 2, \, 3, \, 4, \, 1, \, \frac{\pi}{3} \, , \, 1 \right)^{\top}  \label{Eq_108}
\end{equation}
(as in \cite[Section~5]{FantoniIFAC2002}),
\begin{equation}
T \, = \, 15, \quad \sigma_2 \, = \, \sigma_4 \, = \, \sigma_6 \, = \, \sigma,  \label{Eq_109}
\end{equation}
and consider the two cases
\begin{equation}
\sigma \, = \, 0 \:\: \mbox{(the deterministic case)} \:\:\: \mbox{and} \:\:\: \sigma \, = \, 0.08.  \label{Eq_110}
\end{equation}

According to the MPC algorithm, we implemented the piecewise constant control policy that was recomputed every
$ \, \Delta t_{\mathrm{recomp.}} \, = \, 0.1 \, $ time units as the stabilizing control action at the current state.
When the state lied outside $ \Omega_c $, the control action was approximated by applying the ACADO Toolkit to
the original deterministic system. Otherwise, the value of the locally stabilizing linear feedback
(as mentioned in Subsection~A.2.1 of \hyperlink{Online_App}
{the appendix}) at the current state in $ \Omega_c $ was used.
The It\^o stochastic differential equations were solved via the Euler--Maruyama scheme that coincides with the Milstein
scheme if the noise intensity matrix is constant and diagonal~\cite{KloedenPlaten1995,Carletti2006}. The latter condition
obviously holds for the system~(\ref{Eq_106}). The corresponding time step was set as $ \, \Delta t_{\mathrm{SDE}} = 10^{-5} $.
Under certain smoothness and Lipschitz continuity conditions on the drift vector function and noise intensity matrix function,
the Milstein scheme has the first strong convergence order (while the order of the Euler--Maruyama scheme in general
equals~$ 0.5 $ if the noise intensity matrix is not constant). The first order of accuracy should be preserved in our MPC
implementation, because the control policy is piecewise constant and the ratio
$ \: \Delta t_{\mathrm{recomp.}} \, / \, \Delta t_{\mathrm{SDE}} \: = \: 10^4 \: $ is integer.

The black solid curves in Fig.~\ref{Fig_8} indicate estimates of the mean values $ \, \mathbb{E} \, \| x(t) \| \, $ and
standard deviations $ \, \sqrt{\mathrm{Var} \, \| x(t) \|} \, $ on the time interval~$ [0, T] $ for our MPC implementation.
The deterministic and stochastic cases~(\ref{Eq_110}) are illustrated. For the stochastic case, $ N = 200 $ Monte Carlo
iterations were performed, and $ x^{[i]}(\cdot) $ denotes the state trajectory at the $ i $-th iteration, $ i = \overline{1, N} $.

\begin{figure}
\begin{center}
\includegraphics[width=0.49\linewidth]{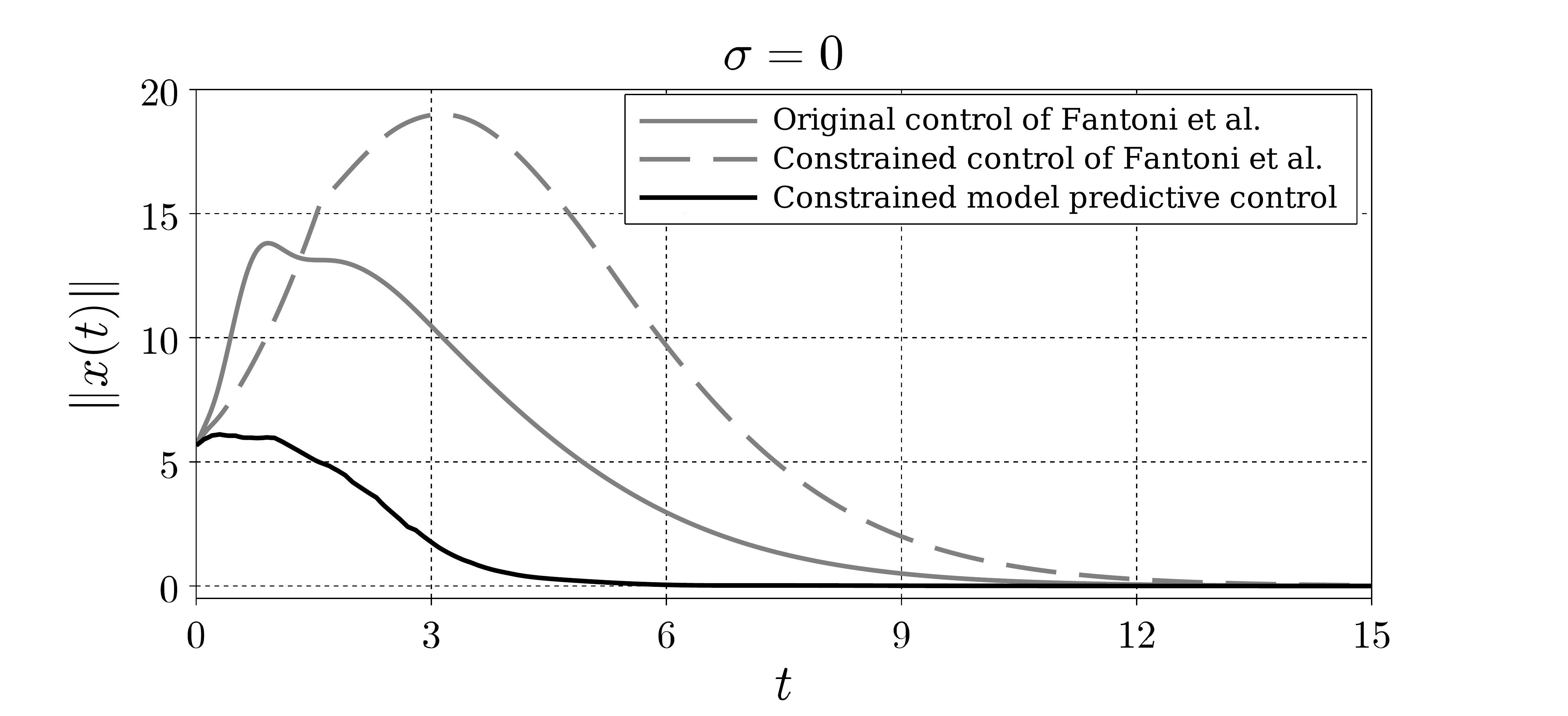}
\includegraphics[width=0.49\linewidth]{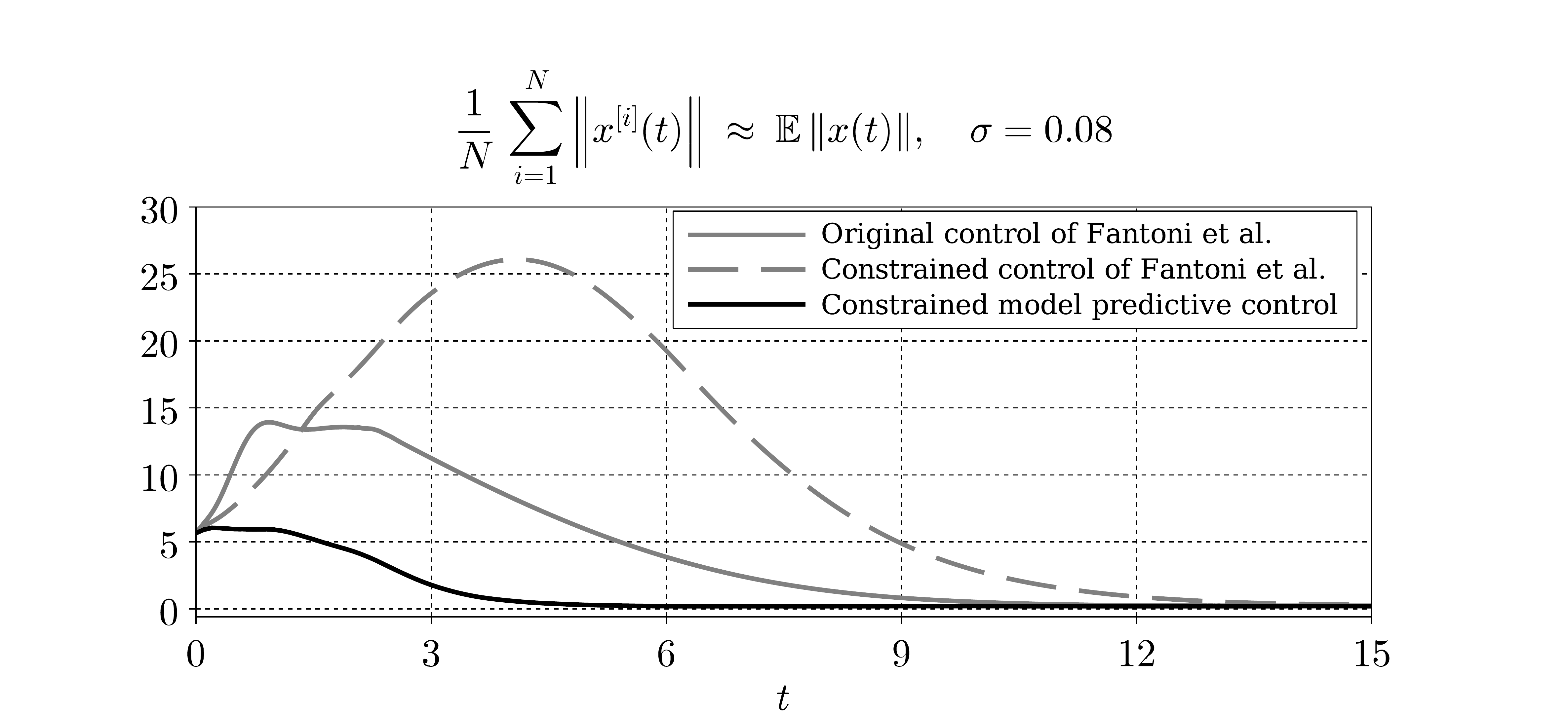} \\
\includegraphics[width=0.65\linewidth]{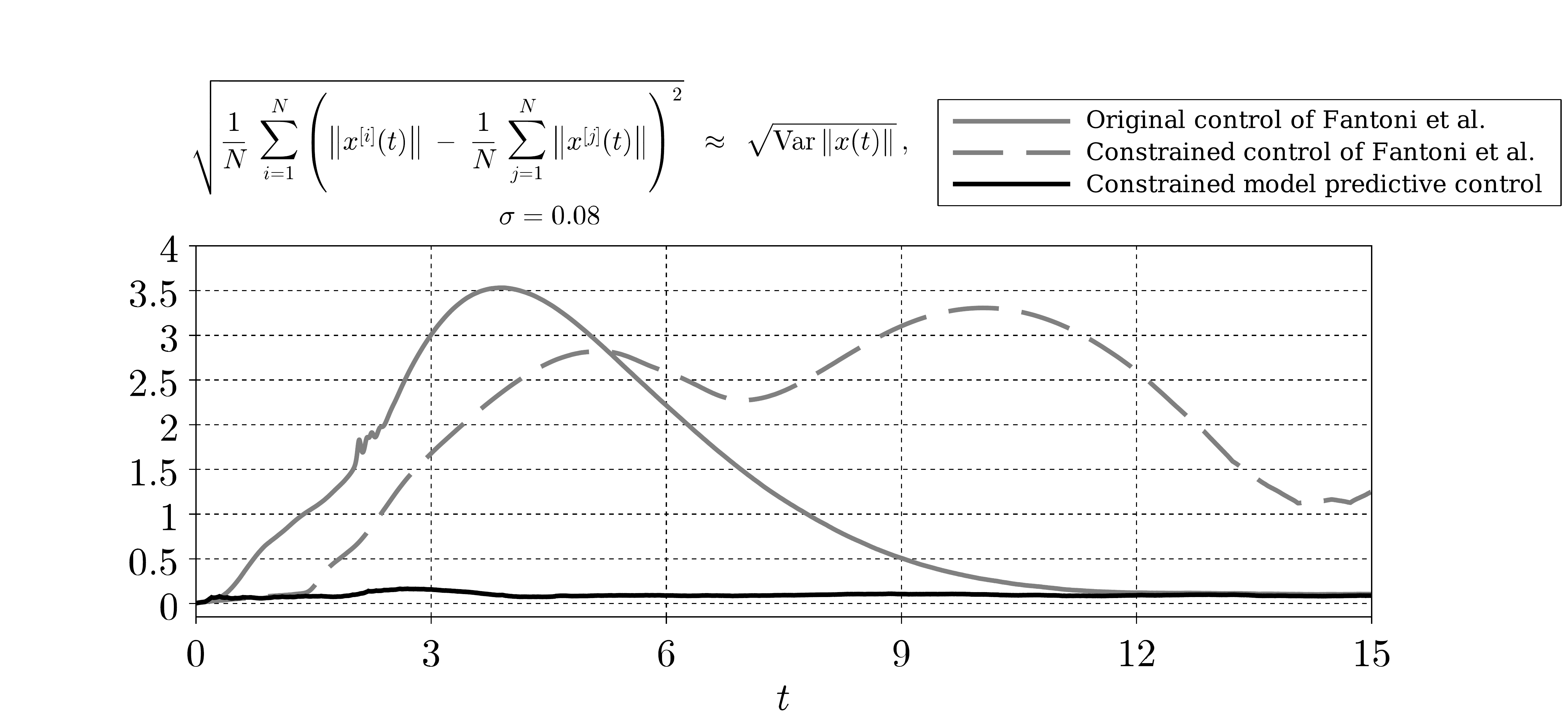}
\end{center}
\bf \caption{\rm Estimates of the mean values $ \, \mathbb{E} \, \| x(t) \| \, $ and standard deviations
$ \, \sqrt{\mathrm{Var} \, \| x(t) \|} \, $ for the original and constrained versions of the stabilizing feedback control
law of Fantoni~et~al.~\cite{FantoniIFAC2002,Fantoni2006} and for the MPC implementation in Example~\ref{Exa_46}. For
the noise intensity parameter $ \, \sigma_2 = \sigma_4 = \sigma_6 = \sigma, \, $ the two cases~(\ref{Eq_110}) are
considered. In order to see the graphs clearer, we do not fix the same scale for the vertical axes in the subfigures.}
\label{Fig_8}
\end{figure}

For comparison, we also integrated the systems~(\ref{Eq_101}) and (\ref{Eq_106}) with the substituted continuous
feedback control strategy that was developed and tested on real experiments by Fantoni~et~al.~\cite{FantoniIFAC2002,Fantoni2006}.
The corresponding analytical representation was obtained after a change of the state and control variables that
transformed the system~(\ref{Eq_101}) to a certain form without the coupling coefficient~$ \alpha $. This strategy was
established to be locally stabilizing for the deterministic PVTOL system and to have a rather wide region of asymptotic
null-controllability (see \cite[Theorem~3.1]{FantoniIFAC2002} or \cite[Theorem~1]{Fantoni2006}). It has to be noted
that the expressions for the first and second time derivatives of the auxiliary variable~$ r_1 $ in
\cite[(16) and (17) on page~414]{Fantoni2006} are incorrectly written ($ 2 \, / \cos \, \theta \, $ in the formula for
$ \dot{r}_1 $ should be replaced with $ \, 2 \, \tan \, \theta, \, $ and the formula for $ \ddot{r}_1 $ should be
accordingly modified). However, there is evidence that the correct relations were used in the further theoretical and
practical investigation of~\cite{Fantoni2006}. In Fig.~\ref{Fig_8}, the solid gray curves correspond to the original
strategy of Fantoni~et~al., while the dashed gray curves indicate its constrained (saturating) version defined as
the orthogonal projection to the compact convex control constraint set~(\ref{Eq_102}). The latter control law was not
considered by Fantoni~et~al., and we tested it to see how the saturation would reduce the control performance.
The Euler--Maruyama scheme was used with the same stepsize $ \, \Delta t_{\mathrm{SDE}} = 10^{-5}, \, $ and the number of
Monte Carlo iterations was again~$ N = 200 $. The explicit control representations could be handled very fast, so no MPC
had to be arranged.

Fig.~\ref{Fig_8} shows that the constrained MPC has an essentially better performance than the original unconstrained
strategy of Fantoni~et~al. and the saturating version of the latter. The higher robustness of the MPC with respect to
stochastic uncertainties can be seen as well. In general, random vibrations can be effectively attenuated only for
moderate noise intensities.

The average runtime of computing the CLF and the related control action at a single state outside $ \Omega_c $ via
the ACADO Toolkit was around 5~seconds on our relatively weak PC. This is much faster and more efficient compared to using
our characteristics based implementation, although the latter is more justified from the theoretical point of view. Besides, more
advanced and accurate optimal control solvers building on direct collocation methods (see Subsection~A.2.3 in
\hyperlink{Online_App}
{the appendix}) often work noticeably faster than ACADO. However, the computational cost of such
MPC implementations may still be rather high for real-time engineering applications. In general, there is a crucial trade-off
between increasing the overall control quality and speeding up the online control evaluation.

In Subsection~A.2.4 of \hyperlink{Online_App}
{the appendix}, we discuss how sparse grids can be incorporated in
the MPC algorithm with the aim to make it faster though less accurate. For investigating the applicability of the modified MPC
algorithm, we estimated the accuracy of a typical high-dimensional sparse grid interpolation technique in the current example.
We used the source code of the C++ library SPARSE\_INTERP\_ND~\cite{BurkardtCode} involving Clenshaw--Curtis nodes,
hierarchical Smolyak's constructions, and weighted sums of polynomial interpolants. First, we built the Clenshaw--Curtis
sparse grid of level~$ 7 $ consisting of $ 44689 $ distinct nodes on the six-dimensional cube~$ [-1, 1]^6 $. At each of
the nodes and also at $ 1000 $ points randomly generated from the uniform distribution on $ [-1, 1]^6 $, the CLF~$ V(\cdot) $
as well as the related control action and costate were approximated by means of the GPOPS--II
software~\cite{PattersonRao2014,GPOPS_II_User_guide} (the latter involves direct collocation and is in general more effective
than ACADO). In order to reduce the interpolation errors caused by the nonsmoothness of the CLF and the discontinuity of
the costate on the surface~$ l_c = \partial \Omega_c $, we decreased the parameter~$ c $ to $ 4 \cdot 10^{-4} $ and hence
reduced the area of $ l_c $ (however, it becomes more difficult to solve exit-time optimal control problems after reducing
terminal sets). For the numerical optimization via GPOPS--II, the IPOPT nonlinear programming solver and the default collocation
method were selected, and the corresponding tolerance was set as $ 10^{-7} $. Regarding the mesh refinement algorithm, we used
the Patterson--Rao and Liu--Rao--Legendre methods with default parameters (if the numerical optimization process for a particular
initial state did not converge with the Patterson--Rao mesh refinement option, it was rerun with the Liu--Rao--Legendre option).
With the help of the obtained data, we then evaluated the errors of the sparse grid interpolation at the $ 1000 $ randomly generated
states (the values computed directly via GPOPS--II were compared with the interpolation estimates). The average relative errors
turned out to be very large, more than $ 50\% $ for the CLF as well as for the related feedback control and costate. The infinitesimal
decrease condition for the CLF~$ V(\cdot) $ was violated at more than half of the selected states after substituting the costate
interpolation estimates instead of the gradient of $ V(\cdot) $. One might think of considering a higher-level sparse grid, but
the number of nodes and the complexity of interpolation would then dramatically increase, while the accuracy still might not become
acceptable. Note also that the cube~$ [-1, 1]^6 $ may not be large enough for practical purposes, while the interpolation for
the sparse grid of the same type and level on a larger parallelepiped is even less accurate.

The sparse grid interpolation was therefore highly inaccurate in this example, even though similar tests for other optimal
control problems in \cite{KangWilcox2017} were successful. Thus, the range of applicability of sparse grid frameworks to
solving feedback control problems and to reducing the complexity of online computations for MPC is a relevant
subject of future research. Other techniques of scattered data interpolation, such as the Kriging method originally arising
from geostatistics (see, e.\,g., \cite[\S 3.7]{PressNR2007}), may be tested as well.  \qed
\end{example}

\section{Conclusion}

In this work, we used exit-time optimal control settings in order to obtain global CLF characterizations,
which could lead to curse-of-dimensionality-free approaches to feedback stabilization for certain classes of
deterministic nonlinear control systems described by ODEs. Both theoretical and practical aspects were
investigated. The computation of the values of the CLFs and stabilizing feedbacks at any selected states could be
reduced to finite-dimensional nonlinear programming problems via characteristics based or direct approximation
techniques as applied to specific exit-time optimal control problems. Recall that the direct numerical frameworks are
less justified from a theoretical perspective but may be more robust in computations, compared to using the method of
characteristics. We also indicated that our framework could be incorporated in MPC schemes for online stabilization.

Unlike Example~\ref{Exa_45} with two-dimensional state space, Example~\ref{Exa_46} with six-dimensional state space
could not be successfully treated via our characteristics based numerical implementation. We hence employed
a direct approximation technique for Example~\ref{Exa_46}. The range of practical applicability of characteristics
based frameworks to exit-time optimal control problems arising in stabilization problems with relatively high
state space dimensions is worth studying further. In particular, an efficient implementation may combine
the framework of Subsection~A.2.2 in \hyperlink{Online_App}
{the appendix} with multiple shooting or
indirect collocation~\cite{Rao2010,Trelat2012}.

Another remaining dilemma is how to reasonably reduce the complexity of online computations in the related MPC schemes,
while preserving a suitable level of accuracy and the stabilization property. In Example~\ref{Exa_46}, a typical
sparse grid framework could not achieve that, although similar tests for other optimal control problems in
\cite{KangWilcox2017} showed acceptable results. Thus, the range of applicability of sparse grids to solving
feedback control problems may also be an interesting subject of future research. Other techniques of scattered data
interpolation (such as the Kriging method \cite[\S 3.7]{PressNR2007}) may be additionally tested. Moreover, since
the characteristics based techniques and advanced direct collocation methods enable costate estimation for optimal control
problems (recall the relation between the costates and the gradient of a value function, as well as the local
Lipschitz continuity properties in Theorems~\ref{Thm_18},~\ref{Thm_40}), it is relevant to design methods for
constructing piecewise affine global CLFs in relatively high dimensions. The framework of \cite{HafsteinKellet2015}
may help in that effort.

\section*{Acknowledgements}

This work was supported in part by AFOSR/AOARD grant FA2386-16-1-4066. We thank Professor Wei~Kang, Naval Postgraduate School,
for enlightening discussions and comments.

\hypertarget{Online_App}{}

\appendix

\setcounter{equation}{0} \renewcommand{\theequation}{A.\arabic{equation}}
\setcounter{section}{0} \renewcommand{\thesection}{A.\arabic{section}}

\section{Proofs of some auxiliary results}

\subsection{Proof of Lemma~2.12}

\begin{lemmaaux}  \label{Thm_A_1_A}
If $ E \subseteq \mathbb{R}^n $ is an open set and a function $ \: \varphi \, \colon \, E \to \mathbb{R} \: $
is Lipschitz continuous with constant~$ C > 0 ${\rm ,} then
\begin{equation}
\| \zeta \| \, \leqslant \, C \sqrt{n} \quad \forall \, \zeta \, \in \, \mathrm{D}^-_{\mathrm{P}} \varphi(x) \quad
\forall x \in E.  \label{Eq_A_1}
\end{equation}
\end{lemmaaux}

\begin{proof}
Let $ x \in E $ and $ \, \zeta \in \mathrm{D}^-_{\mathrm{P}} \varphi(x) $. According to the Lipschitz continuity of
$ \varphi(\cdot) $ and the definition of a proximal subgradient (see \cite[p.~5]{ClarkeLedyaev1998}), there exist
positive numbers~$ \varepsilon_0, \sigma $ (depending on $ x $) such that $ \, \mathrm{B}_{\varepsilon_0}(x) \subseteq E \, $
and
\begin{equation}
C \, \| x' - x \| \:\: \geqslant \:\: \varphi(x') \: - \: \varphi(x) \:\: \geqslant \:\:
\left< \zeta, \, x' - x \right> \: - \: \sigma \, \| x' - x \|^2  \label{Eq_A_2}
\end{equation}
for all $ \, x' \in \mathrm{B}_{\varepsilon_0}(x) $. For every $ \: i \, \in \, \{ 1, 2, \ldots, n \}, \: $ let
$ e_i \in \mathbb{R}^n $ be such that its $ i $-th coordinate equals $ 1 $ and all the other coordinates vanish.
Take arbitrary $ \: i \, \in \, \{ 1, 2, \ldots, n \} \: $ and $ \varepsilon \in (0, \varepsilon_0) $. Then
(\ref{Eq_A_2}) reduces to
$ \: C \: \geqslant \: \left< \zeta, e_i \right> \, - \, \sigma \varepsilon \: $
for $ \: x' \, = \, x + \varepsilon e_i \: $ and to
$ \: C \: \geqslant \: -\left< \zeta, e_i \right> \, - \, \sigma \varepsilon \: $
for $ \: x' \, = \, x - \varepsilon e_i $. As $ \varepsilon \to +0 $, one obtains
$ \, \left| \left< \zeta, e_i \right> \right| \leqslant C $. This leads directly to (\ref{Eq_A_1}).
\end{proof}

\subsection{Proof of Proposition~2.29}

First, recall the notation
\begin{equation}
\begin{aligned}
& H(x, u, p, \tilde{p}) \:\, \stackrel{\mathrm{def}}{=} \:\, \left< p, f(x, u) \right> \: + \: \tilde{p} \, g(x, u), \\
& \mathcal{H}(x, p, \tilde{p}) \:\, \stackrel{\mathrm{def}}{=} \:\, \min_{u' \, \in \, U} \, H(x, u', p, \tilde{p}), \\
& U^*(x, p, \tilde{p}) \:\, \stackrel{\mathrm{def}}{=} \:\, \mathrm{Arg} \min_{u' \, \in \, U} \, H(x, u', p, \tilde{p}) \\
& \forall \: (x, u, p, \tilde{p}) \: \in \: G \times U \times \mathbb{R}^n \times \mathbb{R}
\end{aligned}  \label{Eq_31_A}
\end{equation}
(see (32), (37)), as well as the characteristic Cauchy problems
\begin{equation}
\left\{ \begin{aligned}
& \dot{x^*}(t) \:\: = \:\: \mathrm{D}_p H(x^*(t), \, u^*(t), \, p^*(t), \, \tilde{p}^*) \:\: = \:\:
f(x^*(t), \, u^*(t)), \\
& \dot{p^*}(t) \:\: = \:\: -\mathrm{D}_x H(x^*(t), \, u^*(t), \, p^*(t), \, \tilde{p}^*) \\
& \qquad \:\:
= \:\: -(\mathrm{D}_x f(x^*(t), \, u^*(t)))^{\top} \: p^*(t) \:\, - \:\,
\tilde{p}^* \: \mathrm{D}_x g(x^*(t), \, u^*(t)), \\
& u^*(t) \: \in \: U^*(x^*(t), \, p^*(t), \, \tilde{p}^*), \\
& t \:\, \in \:\, I(x_0, \, u^*(\cdot)) \:\, \stackrel{\mathrm{def}}{=} \:\, \begin{cases}
[0, \, T_{\Omega_c}(x_0, \, u^*(\cdot))], & T_{\Omega_c}(x_0, \, u^*(\cdot)) \: < \: +\infty, \\
[0, +\infty), & T_{\Omega_c}(x_0, \, u^*(\cdot)) \: = \: +\infty,
\end{cases} \\
& x^*(0) \, = \, x_0, \quad p^*(0) \, = \, p_0
\end{aligned} \right.  \label{Eq_38_A}
\end{equation}
(see (39)), where $ \: (x_0, p_0, \tilde{p}^*) \: \in \: G \, \times \, \mathbb{R}^n \, \times \, [0, +\infty) $.

\begin{propaux}  \label{Pro_27_A}
Under the conditions of Theorem~{\rm 2.28,} the Hamiltonian is conserved along any solution of
the characteristic Cauchy problem~{\rm (\ref{Eq_38_A})} with
$ \: (x_0, p_0, \tilde{p}^*) \: \in \: G \, \times \, \mathbb{R}^n \, \times \, [0, +\infty) $.
\end{propaux}

\begin{proof}
By using the representation of directional derivatives of minimum functions (see, e.\,g., \cite[Theorem~I.3.4]{Demyanov1995},
which considers maximum functions, but can be similarly reformulated for minimum functions), one can verify that
\begin{equation}
\frac{\mathrm{d}}{\mathrm{d} t} \: \mathcal{H}(x^*(t), \, p^*(t), \, \tilde{p}^*) \:\, = \:\, 0 \quad
\mbox{for almost all} \:\:\, t \, \in \, I(x_0, \, u^*(\cdot)).  \label{Eq_42}
\end{equation}
Since $ x^*(\cdot) $ and $ p^*(\cdot) $ are absolutely continuous on every compact subset of $ \, I(x_0, \, u^*(\cdot)) \, $
and $ \mathcal{H}(\cdot, \cdot, \cdot) $ is Lipshitz continuous on every compact subset of
$ \, G \times \mathbb{R}^n \times \mathbb{R} \, $ (due to, e.\,g., \cite[Remark~I.3.2]{Demyanov1995}),
the function $ \: I(x_0, \, u^*(\cdot)) \, \ni \, t \: \longmapsto \: \mathcal{H}(x^*(t), \, p^*(t), \, \tilde{p}^*) \: $
is also absolutely continuous on any compact subset of $ \, I(x_0, \, u^*(\cdot)) $. Hence, (\ref{Eq_42}) implies that
the latter function is constant on $ \, I(x_0, \, u^*(\cdot)) $.
\end{proof}

\section{Further details on implementing the curse-of-dimensionality-free approach to CLF approximation and feedback stabilization}

This section accompanies Section~4 and discusses how to practically evaluate the global CLF~$ V(\cdot) $ (or, equivalently,
the Kruzhkov transformed CLF~$ v(\cdot) $) together with the corresponding feedback strategy at any selected state in $ G $,
based on the theoretical results of Section~2 (similar considerations excluding local CLF construction can be applied to
the setting of Section~3). It is also pointed out that our framework can be incorporated in model predictive control schemes for
online stabilization. For convenience, the description is divided into a number of subsections.

\subsection{Construction of a local CLF}

The results of Section~2 were established under the a~priori assumption that a local CLF with desired properties
could be obtained. Analytical construction of local CLFs may in general be a difficult task, if one first considers
the ideal case of unconstrained control inputs and tries to exactly find the corresponding global CLF (by using, e.\,g.,
the results of \cite[\S 9.4]{Isidori1995} or \cite[Chapter~5]{Sepulchre1997}), which can then work locally in case of pointwise
control constraints. Besides, for a number of well-known continuous-time mechanical models, the local or global asymptotic
stabilization properties of certain feedbacks are derived by means of nonstrict Lyapunov functions, such that the right-hand
sides in the related infinitesimal decrease conditions vanish not only at the origin
\cite{Fantoni2002,Malisoff2009,ChoukchouBraham2014,AguilarIbanez2005,AguilarIbanez2011,AguilarIbanez2012}. However,
Definition~2.5 of CLFs and the sufficient conditions of local asymptotic null-controllability used in Remark~2.9
include the strictness.

In this subsection, we propose a linearization based numerical technique for building quadratic local CLFs under some
additional conditions, with the considerations of \cite[Section~3]{ChenAllgower1998} serving as an important motivation.
Those considerations can also be employed for constructing quadratic local CLFs under the same assumptions. Although
the technique presented in the current subsection is less elegant and may be more computationally expensive, it is
more straightforward to use and does not restrict the right-hand sides in the decrease conditions for the resulting
local CLFs necessarily to quadratic functions (in contrast to the approach of \cite[Section~3]{ChenAllgower1998}).

\begin{assumption}  \label{Ass_31}
In addition to Assumption~{\rm 2.1,} suppose that $ \, 0_m \in \mathrm{int} \, U ${\rm ,} $ \, f(0_n, 0_m) = 0_n, \, $
the function~$ f(\cdot, \cdot) $ is continuously differentiable{\rm ,} and the linearization
\begin{equation}
\left\{ \begin{aligned}
& \dot{x}(t) \: = \: A \, x(t) \, + \, B \, u(t), \quad t \geqslant 0, \\
& x(0) \, = \, x_0 \, \in \, G, \\
& u(\cdot) \: \in \: \mathcal{U} \: \stackrel{\mathrm{def}}{=} \: L^{\infty}_{\mathrm{loc}}([0, +\infty), \, U), \\
& A \: \stackrel{\mathrm{def}}{=} \: \mathrm{D}_x f(0_n, 0_m) \: \in \: \mathbb{R}^{n \times n}, \quad
B \: \stackrel{\mathrm{def}}{=} \: \mathrm{D}_u f(0_n, 0_m) \: \in \: \mathbb{R}^{n \times m},
\end{aligned} \right.  \label{Eq_60_A}
\end{equation}
of the system~{\rm (1)} is asymptotically null-controllable.
\end{assumption}

Due to \cite[\S 5.8, Theorem~19]{SontagBook1998}, Assumption~\ref{Ass_31} ensures the existence of a positive definite
matrix~$ P \in \mathbb{R}^{n \times n} $ and a matrix~$ S \in \mathbb{R}^{m \times n} $ such that the functions
\begin{equation}
\breve{V}(x_0) \: = \: \left< P x_0, \, x_0 \right>, \:\:\: \breve{u}(x_0) \: = \: S x_0 \quad
\forall x_0 \in \mathbb{R}^n  \label{Eq_62}
\end{equation}
are respectively a local quadratic CLF and a locally stabilizing linear feedback for (1) in some neighborhood of
the origin~$ 0_n $. One can search for such a neighborhood in the form of a sublevel set of $ \breve{V}(\cdot) $.

In line with \cite[\S 5.8, Proof of Theorem~19]{SontagBook1998}, the control matrix~$ S $ is selected so that
$ \: A + BS \, \in \, \mathbb{R}^{n \times n} \: $ becomes Hurwitz, and $ P $ is a unique positive definite solution of
the matrix equation
\begin{equation}
(A + BS)^{\top} \, P \: + \: P \, (A + BS) \:\, = \:\, -\alpha I_{n \times n}  \label{Eq_63_2}
\end{equation}
with a constant~$ \alpha > 0 $ (one has $ \alpha = 1 $ in that proof, though any $ \alpha > 0 $ would in fact work).

The gradient of $ \breve{V}(\cdot) $ is given by
$$
\mathrm{D} \breve{V}(x_0) \: = \: 2 \, P x_0 \quad \forall x_0 \in \mathbb{R}^n,
$$
and the sublevel sets
\begin{equation}
\breve{\Omega}_r \:\, \stackrel{\mathrm{def}}{=} \:\, \{ x \in \mathbb{R}^n \: \colon \: \breve{V}(x) \, \leqslant \, r \} \quad
\forall r > 0  \label{Eq_64}
\end{equation}
are closed ellipsoidal domains in $ \mathbb{R}^n $. If a level~$ c' > 0 $ satisfies
\begin{equation}
\arraycolsep=1.5pt
\def\arraystretch{1.5}
\begin{array}{c}
\breve{\Omega}_{c'} \, \subset \, G, \\
\left< \mathrm{D} \breve{V}(x), \: f( x, \breve{u}(x)) \right> \:\, < \:\, 0 \quad
\forall \, x \, \in \, \breve{\Omega}_{c'} \setminus \{ 0_n \}, \\
\breve{u}(x) \in U \quad \forall x \in \breve{\Omega}_{c'},
\end{array}  \label{Eq_65}
\end{equation}
then one can take
\begin{equation}
\Omega \: = \: \mathrm{int} \, \breve{\Omega}_{c'}, \quad c \in (0, c'), \quad \Omega_c = \breve{\Omega}_{c}
\label{Eq_66}
\end{equation}
and select a local CLF and a locally stabilizing feedback as
\begin{equation}
V_{\mathrm{loc}}(x) \, \stackrel{\mathrm{def}}{=} \, \breve{V}(x), \:\:\:
u_{\mathrm{loc}}(x) \, \stackrel{\mathrm{def}}{=} \, \breve{u}(x) \quad \forall x \in \Omega.  \label{Eq_67}
\end{equation}
The greater such a level~$ c $, the wider the target set~$ \Omega_c $, and, hence, the easier to numerically solve the exit-time
optimal control problem~(21). This leads to the problem of finding the supremum~$ c_{\sup} $ of all suitable levels,
which can be practically treated by testing the nodes of a grid on the interval $ [0, \tilde{c}] $ with a sufficiently large
right endpoint~$ \tilde{c} > 0 $. For each node, an appropriate finite-dimensional optimization problem should be numerically
solved. Finally, it is reasonable to select $ c $ somewhat lower than $ c_{\sup} $, so that
$ \: \left< \mathrm{D} V_{\mathrm{loc}}(x), \: f(x, u_{\mathrm{loc}}(x)) \right> \: $ is not very close to zero at
states~$ x $ near the boundary~$ \, l_c = \partial \Omega_c $.

\begin{remark}  \label{Rem_31_2}  \rm
If the asymptotic null-controllability condition in Assumption~\ref{Ass_31} is replaced with the stronger exact null-controllability
condition
\begin{equation}
\mathrm{rank} \: \left[ B, \, A B, \, A^2 B, \, \ldots, \, A^{n - 1} B \right] \:\, = \:\, n,  \label{Eq_60_A_A}
\end{equation}
then the matrix~$ P \in \mathbb{R}^{n \times n} $ in (\ref{Eq_62}) can be chosen as a unique positive definite solution of the algebraic
Riccati equation
\begin{equation}
A^{\top} P \: + \: P A \: - \: P B R^{-1} B^{\top} P \: + \: Q \:\, = \:\, 0_{n \times n}  \label{Eq_63}
\end{equation}
with arbitrary positive definite matrices $ \, Q \in \mathbb{R}^{n \times n} ${\rm ,} $ R \in \mathbb{R}^{m \times m}, \, $ and
the control matrix in (\ref{Eq_62}) can be taken as $ \, S = -R^{-1} B^{\top} P \, $ (see, e.\,g.,
\cite[Chapter~VII, \S 3.3]{Afanasiev1996}). In this case, (\ref{Eq_62}) gives the value function and optimal feedback strategy for
the infinite-horizon linear-quadratic optimal control problem
\begin{equation}
\left\{ \begin{aligned}
& \dot{x}(t) \:\, = \:\, A \, x(t) \: + \: B \, u(t), \quad t \geqslant 0, \\
& x(0) \, = \, x_0 \, \in \, \mathbb{R}^n, \\
& u(\cdot) \: \in \: L_{\mathrm{loc}}^{\infty}([0, +\infty), \, \mathbb{R}^m), \\
& \int\limits_0^{+\infty} (\left< Q \, x(t), \, x(t) \right> \: + \:
\left< R \, u(t), \, u(t) \right>) \: \mathrm{d} t \:\, \longrightarrow \:\, \min,
\end{aligned} \right.  \label{Eq_61}
\end{equation}
with unconstrained control inputs.  \qed
\end{remark}

\subsection{Using the characteristics based representation of the value function}
\label{subsect_char_based_method}

This subsection describes the use of the characteristics based representation of $ v(\cdot) $ in $ \mathcal{D}_0 \setminus \Omega_c $
given by Theorem~2.28, whose formulation is repeated here for convenience.

\begin{theoremaux}
Let Assumptions~{\rm 2.1, 2.7, 2.15, 2.21} and {\rm 2.23} hold. For any initial
state~$ \, x_0 \, \in \, \mathcal{D}_0 \setminus \Omega_c, \, $ the Kruzhkov transformed value~$ v(x_0) $
defined by {\rm (20), (21), (36)} is the minimum of
\begin{equation}
1 \:\: - \:\: \exp \left\{ -\int\limits_0^{T_{\Omega_c}(x_0, \, u^*(\cdot))} g(x^*(t), \, u^*(t)) \: \mathrm{d} t \:\, - \:\, c \right\}
\label{Eq_37_A}
\end{equation}
over the solutions of the characteristic Cauchy problems~{\rm (\ref{Eq_38_A})} for all extended initial adjoint vectors
\begin{equation}
(p_0, \tilde{p}^*) \:\, \in \:\, \{ (p, \tilde{p}) \: \colon \: p \in \mathbb{R}^n, \:\: \tilde{p} \in \{ 0, 1 \} \}.
\label{Eq_39_A}
\end{equation}
Moreover{\rm ,} the same value is obtained when minimizing over the bounded set
\begin{equation}
(p_0, \tilde{p}^*) \:\, \in \:\, \{ (p, \tilde{p}) \: \in \: \mathbb{R}^n \times \mathbb{R} \: \colon \:
\| (p, \tilde{p}) \| \: = \: 1, \:\:\, \tilde{p} \geqslant 0 \},  \label{Eq_40_A}
\end{equation}
or even over its subset
\begin{equation}
(p_0, \tilde{p}^*) \:\, \in \:\, \{ (p, \tilde{p}) \: \in \: \mathbb{R}^n \times \mathbb{R} \: \colon \:
\| (p, \tilde{p}) \| \: = \: 1, \:\:\, \tilde{p} \geqslant 0, \:\:\,
\mathcal{H}(x_0, p, \tilde{p}) \: = \: 0 \}.  \label{Eq_41_A}
\end{equation}
\end{theoremaux}

In order to ensure the uniqueness of the solutions of the characteristic Cauchy problems~(\ref{Eq_38_A}) in the normal
case~$ \tilde{p}^* > 0 $, the following conditions are imposed.

\begin{assumption}  \label{Ass_28}
The extremal control map~$ U^*(x, p, \tilde{p}) $ {\rm (}see {\rm (\ref{Eq_31_A}))} is singleton for $ \tilde{p} > 0 ${\rm ,}
and the corresponding function defined on $ \: G \, \times \, \mathbb{R}^n \, \times \, (0, +\infty) \: $ and taking values in $ U $
is locally Lipschitz continuous.
\end{assumption}

This often holds if, for instance, the running cost is regularized by adding a suitable control-dependent term.
In the abnormal case~$ \tilde{p}^* = 0 $, there typically exist states and adjoint vectors for which the extremal
control map is nonsingleton, regardless of the running cost. However, one can expect that abnormal characteristics
would rarely be optimal, since they do not take the running cost into account. If the extremal control map on
a characteristic trajectory becomes nonsingleton at some time during computations, one can select any extremal
control action at this time. Note also that, for some particular classes of optimal control problems, an additional
analysis via Pontryagin's principle may allow characteristics based methods to be modified so that singular regimes are
explicitly handled and the nonuniqueness in the choice of extremal control actions is avoided (see
\cite[Examples~3.14 and 3.15]{YegorovDower2017} related to the case of a fixed finite horizon).

\subsubsection{Practical specification of exit times and the domain of asymptotic null-controlla\-bility}

Next, recall the representation of the domain of asymptotic null-controllability~$ \mathcal{D}_0 $ in Theorem~2.27:
\begin{equation}
\mathcal{D}_0 \:\: = \:\: \{ x_0 \in G \: \colon \: V(x_0) \, < \, +\infty \} \:\: = \:\:
\{ x_0 \in G \: \colon \: v(x_0) \, < \, 1 \}.  \label{Eq_43_2}
\end{equation}
Before evaluating the CLF at a particular state~$ x_0 \in G $, one usually does not know if $ x_0 \in \mathcal{D}_0 $ or not.
Even if the initial state lies in $ \mathcal{D}_0 $, there may still exist extended initial adjoint vectors~$ (p_0, \tilde{p}^*) $
that generate characteristic trajectories with infinite exit time and with the cost~(\ref{Eq_37_A}) equal to $ 1 $.

Let us provide a practical rule to determine costs and exit (terminal) times during numerical integration of
the characteristic Cauchy problems~(\ref{Eq_38_A}). A~priori, it is reasonable to fix a sufficiently large finite
upper bound~$ T_{\max} > 0 $ for exit times, even though this is in general a heuristic choice. Take also
a sufficiently small parameter~$ \varepsilon \in (0, 1) $. It is proposed to stop integrating the characteristic system
when at least one of the following conditions starts to hold:
\begin{list}{\rm \arabic{count})}%
{\usecounter{count}}
\setlength\itemsep{0em}
\item  the a~priori selected upper bound for terminal times is reached, i.\,e., $ t = T_{\max} $;
\item  the target set is entered, i.\,e., $ x^*(t) \in \Omega_c $;
\item  the accumulated cost becomes very close to the maximum value~$ 1 $, i.\,e.,
$$
1 \:\: - \:\: \exp \left\{ -\int\limits_0^t g(x^*(s), \, u^*(s)) \: ds \:\, - \:\, c \right\} \:\: \geqslant \:\:
1 - \varepsilon.
$$
\end{list}
These three cases accordingly define practical exit times. In Case~2, the cost is specified as (\ref{Eq_37_A}),
while, in Cases~1 and 3, it is set as $ 1 - \varepsilon $.

Let $ \: \hat{v} \colon \: G \setminus \Omega_c \, \to \, [0, \, 1 - \varepsilon] \: $ be the approximation of
$ v(\cdot) $ in $ G \setminus \Omega_c $ obtained by incorporating the aforementioned arguments in a numerical
method building on Theorem~2.28. Select one more parameter~$ \varepsilon_1 \in (0, 1) $, which is
sufficiently small but not less than $ \varepsilon $. Then the domain~$ \mathcal{D}_0 $ can be approximated by
its inner estimate
\begin{equation}
\hat{\mathcal{D}}_0 \:\: \stackrel{\mathrm{def}}{=} \:\: \Omega_c \:\, \cup \:\,
\{ x_0 \, \in \, G \setminus \Omega_c \: \colon \: \hat{v}(x_0) \, < \, 1 - \varepsilon_1 \}.
\label{Eq_43}
\end{equation}

\subsubsection{An auxiliary problem for finding an appropriate initial guess for the main optimization problem}

Let us refer to the finite-dimensional optimization problem formulated in Theorem~2.28 as the main problem.
The cost function in this problem may be essentially multi-extremal for some initial states~$ x_0 \in G $.
There may in particular exist a relatively large subset of (\ref{Eq_40_A}) consisting of the extended initial
adjoint vectors~$ (p_0, \tilde{p}^*) $ for which the state trajectories of (\ref{Eq_38_A}) do not reach the target
set~$ \Omega_c $ and the approximate cost defined above equals $ 1 - \varepsilon $. It is hence reasonable first
to introduce an auxiliary problem, whose solution can then serve as an initial guess for an iterative algorithm
applied to the main problem.

Consider the characteristic system rewritten in reverse time~$ \tau $, that is,
\begin{equation}
\left\{ \begin{aligned}
& \frac{\mathrm{d} \hat{x}^*(\tau)}{\mathrm{d} \tau} \:\: = \:\: -\mathrm{D}_p H(\hat{x}^*(\tau), \, \hat{u}^*(\tau), \,
\hat{p}^*(\tau), \, \tilde{p}^*) \:\: = \:\: -f(\hat{x}^*(\tau), \, \hat{u}^*(\tau)), \\
& \frac{\mathrm{d} \hat{p}^*(\tau)}{\mathrm{d} \tau} \:\: = \:\: \mathrm{D}_x H(\hat{x}^*(\tau), \, \hat{u}^*(\tau), \,
\hat{p}^*(\tau), \, \tilde{p}^*) \\
& \qquad\quad \:\,
= \:\: (\mathrm{D}_x f(\hat{x}^*(\tau), \, \hat{u}^*(\tau)))^{\top} \: \hat{p}^*(\tau) \:\, + \:\,
\tilde{p}^* \: \mathrm{D}_x g(\hat{x}^*(\tau), \, \hat{u}^*(\tau)), \\
& \hat{u}^*(\tau) \: \in \: U^*(\hat{x}^*(\tau), \, \hat{p}^*(\tau), \, \tilde{p}^*), \\
& \tau \: \in \: I(\hat{x}^*(0), \, \hat{p}^*(0), \, \tilde{p}^*),
\end{aligned} \right.  \label{Eq_48}
\end{equation}
where $ \: I(\hat{x}^*(0), \, \hat{p}^*(0), \, \tilde{p}^*) \: $ is the maximum extendability time interval with zero
left endpoint during which the related solutions emanating from $ \: (\hat{x}^*(0), \, \hat{p}^*(0)) \: $ satisfy
$ \: \hat{x}^*(\tau) \, \in \, G, \: $ and take also
\begin{equation}
\hat{x}^*(0) \, \in \, \Omega_{c_1}, \quad
\hat{p}^*(0) \: \in \: \mathrm{N}(\hat{x}^*(0); \, \Omega_{c_1})  \label{Eq_49}
\end{equation}
for some parameter~$ c_1 \in (0, c] $. The conditions~(\ref{Eq_49}) come from Pontryagin's principle for the exit-time
optimal control problem with the target set
\begin{equation}
\Omega_{c_1} \:\, \stackrel{\mathrm{def}}{=} \:\, \left\{ x \in \bar{\Omega} \, \colon \,
V_{\mathrm{loc}}(x) \leqslant c_1 \right\} \:\, \subseteq \:\, \Omega_c  \label{Eq_50}
\end{equation}
instead of $ \Omega_c $. The aim is to get to a selected state~$ x_0 \in G $ as close as possible, which leads to
a shooting problem. The level~$ c_1 $ is allowed to be less than $ c $ in order to make the shooting more robust,
i.\,e., to increase the possibility that the resulting initial guess for the main optimization problem with
$ x_0 \in \mathcal{D}_0 $ generates a forward-time characteristic state trajectory reaching the original target
set~$ \Omega_c $ within the fixed time interval~$ [0, T_{\max}] $. Note also that a similar reduction of a terminal
sublevel set is used in \cite{MichalskaMayne1993} for increasing the robustness of a receding horizon stabilization
algorithm.

Additional properties need to be imposed.

\begin{assumption}  \label{Ass_29}
$ c_1 \in (0, c] $ is a constant. Item~{\rm 4} of Assumption~{\rm 2.7} holds when $ c $ is replaced with
$ c_1 $ and $ C_3 $ remains the same. Item~{\rm 4} of Assumption~{\rm 2.15} holds when $ c $ is replaced with
$ c_1 $ and $ C_6 $ is replaced with $ C'_6 \in (0, C_6] $.
\end{assumption}

Denote also
\begin{equation}
l_{c_1} \:\, \stackrel{\mathrm{def}}{=} \:\, \left\{ x \in \bar{\Omega} \, \colon \,
V_{\mathrm{loc}}(x) = c_1 \right\} \:\, = \:\, \partial \Omega_{c_1}.  \label{Eq_51}
\end{equation}

\begin{assumption}  \label{Ass_30}
The local CLF~$ V_{\mathrm{loc}}(\cdot) $ is continuously differentiable in $ \Omega ${\rm ,} and{\rm ,} for
any direction~$ \xi \in \mathbb{R}^n $ with $ \| \xi \| = 1 ${\rm ,} there exists a unique state
\begin{equation}
x^*_{\mathrm{term}}(\xi) \:\: \in \:\: l_{c_1} \:\, \cap \:\, \{ \lambda \xi \: \colon \: \lambda > 0 \}.  \label{Eq_51_2}
\end{equation}
\end{assumption}

\begin{remark}  \label{Rem_31}  \rm
Let Assumptions~2.1, 2.7 and \ref{Ass_29}  hold. It is not difficult to verify that
Assumption~\ref{Ass_30} also holds if, for example, the following conditions are fulfilled:
\begin{itemize}
\setlength\itemsep{0em}
\item  $ \Omega $ is convex;
\item  $ V_{\mathrm{loc}}(\cdot) $ is continuously differentiable and convex in $ \Omega $ (the convexity of
$ V_{\mathrm{loc}}(\cdot) $ implies the convexity of its sublevel sets, such as $ \Omega_c $ and
$ \Omega_{c_1} $);
\item  $ V_{\mathrm{loc}}(\cdot) $ satisfies
$$
\left< x, \, \mathrm{D} V_{\mathrm{loc}}(x) \right> \: > \: 0 \quad \forall x \in l_{c_1}
$$
(this holds in particular for quadratic local CLFs).
\end{itemize}
\qed
\end{remark}

In line with Assumptions~\ref{Ass_29},~\ref{Ass_30} and the Hamiltonian vanishing condition in Pontryagin's
principle, the initial data~(\ref{Eq_49}) for the reverse-time characteristic system~(\ref{Eq_48}) is taken as
\begin{equation}
\arraycolsep=1.5pt
\def\arraystretch{1.5}
\begin{array}{c}
\hat{x}^*(0) \: = \: x^*_{\mathrm{term}}(\xi) \: \in \: l_{c_1}, \quad \xi \in \mathbb{R}^n, \quad \| \xi \| = 1, \\
\hat{p}^*(0) \,\, = \,\, \varkappa \: \mathrm{D} V_{\mathrm{loc}}(\hat{x}^*(0)), \quad \varkappa \geqslant 0, \\
\mathcal{H}(\hat{x}^*(0), \: \hat{p}^*(0), \: \tilde{p}^*) \,\, = \,\, 0, \quad \tilde{p}^* \geqslant 0.
\end{array}  \label{Eq_52}
\end{equation}
By adopting the normalization condition $ \: \| (\hat{p}^*(0), \, \tilde{p}^*) \| \: = \: 1, \: $ one obtains
\begin{equation}
\tilde{p}^* \in [0, 1], \quad
\varkappa \: = \: \frac{\sqrt{1 - (\tilde{p}^*)^2}}{\| \mathrm{D} V_{\mathrm{loc}}(\hat{x}^*(0)) \|} \, ,
\label{Eq_53}
\end{equation}
and $ \: \mathrm{D} V_{\mathrm{loc}}(\hat{x}^*(0)) \, \neq \, 0_n \: $ due to the infinitesimal decrease condition on
the local CLF and $ \: \inf \: \{ g(x, u) \: \colon \: x \in l_{c_1},  \:\, u \in U \} \,\, > \,\, 0 \: $ (recall
Assumptions~2.7, 2.15 and \ref{Ass_29}). The latter properties also yield that, for any
$ \xi \in \mathbb{R}^n $ with $ \| \xi \| = 1 $, the function
\begin{equation}
[0, 1] \: \ni \: \tilde{p} \:\: \longmapsto \:\: \mathcal{H} \left( x^*_{\mathrm{term}}(\xi), \:\:
\frac{\sqrt{1 - \tilde{p}^2}}{\| \mathrm{D} V_{\mathrm{loc}}(x^*_{\mathrm{term}}(\xi)) \|} \,\,
\mathrm{D} V_{\mathrm{loc}}(x^*_{\mathrm{term}}(\xi)), \:\: \tilde{p} \right)  \label{Eq_53_2}
\end{equation}
is negative at $ \tilde{p} = 0 $ and positive at $ \tilde{p} = 1 $, i.\,e., at least one root~$ \tilde{p}^* $ exists
and satisfies
\begin{equation}
0 < \tilde{p}^* < 1  \label{Eq_54}
\end{equation}
(which in particular excludes the abnormal case~$ \tilde{p}^* = 0 $). Possible multiple roots on $ (0, 1) $ can be numerically
bracketed, for example, by using the \texttt{zbrak} routine from \cite[\S 9.1]{PressNR2007}. Each of them is further handled,
and a root with the best shooting performance is finally selected.

The shooting goal is to minimize the lowest deviation
\begin{equation}
\min_{\tau \:\: \in \:\: [0, \, T_{\max}] \:\, \cap \:\, I(\hat{x}^*(0), \,\, \hat{p}^*(0), \,\, \tilde{p}^*)} \:
\left\| \hat{x}^*(\tau) \, - \, x_0 \right\|^2  \label{Eq_55}
\end{equation}
from a state~$ x_0 \in G $ over the solutions of the reverse-time characteristic Cauchy problems~(\ref{Eq_48}), (\ref{Eq_52}),
(\ref{Eq_53}). One therefore arrives at optimizing over the vectors~$ \xi $ on the unit sphere in $ \mathbb{R}^n $, which can be
parametrized as follows:
\begin{equation}
\left\{ \begin{aligned}
& \xi_1 \: = \: \prod_{i = 1}^{n - 1} \sin \, \theta_i, \\
& \xi_j \: = \: \cos \, \theta_{j - 1} \, \prod_{i = j}^{n - 1} \sin \, \theta_i, \quad
j \, = \, \overline{2, \, n - 1}, \\
& \xi_n \: = \: \cos \, \theta_{n - 1}, \\
& 0 \leqslant \theta_1 < 2 \pi, \quad 0 \leqslant \theta_j \leqslant \pi, \quad
j \, = \, \overline{2, \, n - 1}.
\end{aligned} \right.  \label{Eq_56}
\end{equation}
The periodicity in the angles $ \, \theta_i $, $ i = \overline{1, n}, \, $ allows for performing unconstrained optimization
over them. Possible multi-extremality in the auxiliary problem can be treated by applying an iterative optimization method to
a fixed number of random initial guesses generated according to the uniform angles distribution.

If $ \: (\hat{x}^*(\cdot), \, \hat{p}^*(\cdot), \, \tilde{p}^*) \: $ is an optimal shooting characteristic and $ \tau' $ is
a minimizer in (\ref{Eq_55}), then the extended adjoint vector
\begin{equation}
\left( \frac{\hat{p}^*(\tau')}{\tilde{p}^*} \, , \: 1 \right)  \label{Eq_57}
\end{equation}
(normalized so as to make the last coordinate equal to $ 1 $) can specify the initial guess for an iterative optimization method
applied to the main problem for forward-time characteristics. The normalization in (\ref{Eq_57}) allows for optimizing with
respect to $ p_0 \in \mathbb{R}^n $ in the main problem, as well as for representing the gradient~$ \mathrm{D} V(\cdot) $
of the original value function along optimal characteristics in case~$ x_0 \in \mathcal{D}_0 $ ($ V(x_0) < +\infty $) directly via
the adjoint variable (costate).

If $ p_0 $ is an optimal initial costate in the main problem (with $ \tilde{p}^* = 1 $) for an initial state~$ x_0 $,
the optimal control action at this state is chosen from $ \, U^*(x_0, p_0, 1) $.

Even if $ x_0 \in \mathcal{D}_0 $ and the optimal cost~(\ref{Eq_55}) in the auxiliary shooting problem is rather small,
the related initial guess~(\ref{Eq_57}) for the main problem and also some neighboring adjoint vectors might still
lead to a forward-time characteristic state trajectory that emanates from $ x_0 $ but does not reach the target
set~$ \Omega_c $ within the time interval~$ [0, T_{\max}] $. In this situation, the shooting is not accurate enough, but
the following technique for evaluating $ v(x_0) $ may be helpful. Select two sufficiently small positive
parameters~$ \delta_1, \delta_2 $. If the square root of the optimal shooting cost~(\ref{Eq_55}) is less than $ \delta_1 $
and $ \, \hat{x}_0 = \hat{x}^*(\tau') \, $ is the closest state to $ x_0 $ on the optimal shooting characteristic, then
one can solve the main optimization problem (with $ \tilde{p}^* = 1 $) for the initial state~$ \hat{x}_0 $ and use
the corresponding value~$ v(\hat{x}_0) $ and optimal initial costate~$ \hat{p}_0 $ in order to approximate~$ v(x_0) $.
More precisely, if $ \: \| \hat{x}_0 - x_0 \| \, < \, \delta_1 $ and
$ \:  v(\hat{x}_0) \, = \, 1 - e^{-V(\hat{x}_0)} \, < \, 1 - \varepsilon, \: $ consider the first-order estimate
\begin{equation}
\arraycolsep=1.5pt
\def\arraystretch{1.5}
\begin{array}{c}
V_1(x_0) \:\: \stackrel{\mathrm{def}}{=} \:\: V(\hat{x}_0) \: + \: \left< \hat{p}_0, \, x_0 - \hat{x}_0 \right>, \quad
V(\hat{x}_0) \: = \: -\ln \, (1 \, - \, v(\hat{x}_0)), \\
v_1(x_0) \: \stackrel{\mathrm{def}}{=} \: 1 \, - \, e^{-V_1(x_0)},
\end{array}  \label{Eq_58}
\end{equation}
and, if also $ \: 0 \, \leqslant \, v_1(x_0) \, < \, 1 - \varepsilon \: $ and $ \:  |v_1(x_0) \, - \, v(\hat{x}_0)| \: < \: \delta_2, \: $
take $ \: v(x_0) \, \approx \, v_1(x_0) $. If at least one of the tested conditions does not hold, put
$ \: v(x_0) \, \approx \, 1 - \varepsilon $. Moreover, the sought-after control action at the state~$ x_0 $ is selected from
$ \, U^*(x_0, \hat{p}_0, 1) $.

According to Theorem~2.20, $ v(\cdot) $ is differentiable almost everywhere in $ \mathcal{D}_0 $. If it is differentiable at
$ \hat{x}_0 $ and some convex open neighborhood of $ \hat{x}_0 $ containing $ x_0 $ lies in $ \mathcal{D}_0 $, then (\ref{Eq_58}) indeed
gives a first-order approximation of $ v(x_0) $, because the costate~$ \hat{p}_0 $ (for which $ \tilde{p}^* = 1 $)
represents~$ \mathrm{D} V(\hat{x}_0) $.

\subsection{Using direct approximation methods for optimal open-loop control problems}
\label{subsect_direct_methods}

If the field of the solutions of the characteristic Cauchy problems introduced in Theorem~2.28 has a complicated
structure (which takes place for a wide class of high-dimensional nonlinear control systems), one may face significant
difficulties when trying to numerically solve the characteristics based and possibly multi-extremal optimization problem.
In particular, it may be very difficult to achieve a suitable shooting accuracy in the auxiliary problem for the reverse-time
characteristics. In this situation, it is reasonable to compute the value function and optimal control action at any selected
state by using so-called direct approximation methods, which are implemented in special software, such as
GPOPS--II~\cite{PattersonRao2014,GPOPS_II_User_guide}, ICLOCS2~\cite{ICLOCS2}, PSOPT~\cite{Becerra2010,PSOPT_Manual},
BOCOP~\cite{BonnansMartinon2017}, and ACADO~\cite{Houska2011,ACADO_Manual}. These methods involve direct transcriptions of
infinite-dimensional optimal open-loop control problems to finite-dimensional nonlinear programming problems via
discretizations in time applied to state and control variables, as well as to dynamical state equations. Compared to
indirect frameworks building on Pontryagin's principle and the method of characteristics (such as the considerations of
Subsection~\ref{subsect_char_based_method}), the direct approximation techniques are in principle less precise and less
justified from the theoretical point of view, but often more robust with respect to initialization and more straightforward
to use.

Besides, such aforementioned optimal control solvers as GPOPS--II, ICLOCS2, PSOPT, and BOCOP employ direct collocation
methods and can even provide costate estimates via the Karush--Kuhn--Tucker multipliers of the nonlinear programming
problems (see also \cite{Benson2006,Garg2011}). The range of applicability of these estimates is a relevant subject of
future research.

\subsection{Incorporating the approach in model predictive control schemes and using sparse grids}

Section~4 discusses the advantages of our curse-of-dimensionality-free approach to CLF approximation and feedback stabilization.
They allow for incorporating the approach in online stabilization algorithms based on model predictive control (MPC) methodologies
(a comprehensive introduction to the latter as well as various applications can be found in
\cite{MichalskaMayne1993,ChenAllgower1998,Fontes2001,Jadbabaie2005,Wang2009,GrunePannek2017}). These methodologies can be
implemented even if the original deterministic system is perturbed by a stochastic noise with a small intensity (as in
\cite[Example~5.3]{YegorovDower2017} or \cite[Subsections~5.2 and 5.3]{KangWilcox2017}).

For example, such an online stabilization algorithm can be formulated as follows. Fix a finite time interval~$ [0, T] $ and its
partition
$$
0 \, = \, t_0 \, < \, t_1 \, < \, \ldots \, < \, t_{N - 1} \, < \, t_N \, = \, T.
$$
A particular case is the uniform time grid with equal steps $ \: t_{i + 1} - t_i \, = \, T / N $,
$ \, i \, = \, \overline{0, N - 1} $. One can use the piecewise constant time-dependent control strategy that is recomputed at
each instant $ \: t = t_i $, $ \: i \, \in \, \{0, 1, \ldots, N - 1 \}, \: $ as the stabilizing control action evaluated at
the corresponding current state~$ x(t_i) $ via our approach.

In a real-time implementation, the time for computing a new control action is not negligible, especially if the current
state lies outside the sublevel set~$ \Omega_c $ of the local CLF~$ V_{\mathrm{loc}}(\cdot) $ (so that the related
exit-time optimal open-loop control problem has to be numerically solved). Delays in the control switches should then
take place. Since the control actions are evaluated by using the states~$ \: x(t_i) $,
$ \, i \, = \, \overline{0, N - 1}, \: $ the time steps $ \: t_{i + 1} - t_i $,
$ \, i \, = \, \overline{0, N - 1}, \: $ have to be selected greater than a~priori estimates for those delays.

The online control computation can be made faster though less accurate, if one carries out the following scheme that
employs a sparse grid framework (such as those described in \cite{KangWilcox2017,Garcke2012}):
\begin{itemize}
\item  take a bounded region~$ \Pi $ (e.\,g., a parallelepiped) in the state space, so that $ \Omega_c \subset \Pi $, and
generate an a~priori sparse grid~$ \Sigma $ on $ \Pi $;
\item  perform the offline evaluation of the optimal costates or the optimal control actions at the nodes of
$ \Sigma $, and store the resulting offline data in advance (we use the convention that the optimal costate and
the optimal control action at a state~$ x \in \Omega_c $ are defined as the gradient~$ \mathrm{D} V_{\mathrm{loc}}(x) $
of the local CLF and the value~$ u_{\mathrm{loc}}(x) $ of the related locally stabilizing feedback, respectively);
\item  the level~$ c $ for the target set~$ \Omega_c $ should be chosen sufficiently small in order to reduce
the sparse grid interpolation errors caused by the discontinuity of the optimal costate and the optimal feedback
control strategy on the boundary~$ l_c = \partial \Omega_c $;
\item  for a state~$ \: x(t_i) $, $ \: i \, \in \, \{0, 1, \ldots, N - 1 \}, \: $ on an online controlled
trajectory, obtain the next control action (i) as $ u_{\mathrm{loc}}(x(t_i)) $ if $ x(t_i) \in \Omega_c $,
(ii) via the interpolation from the offline data on the sparse grid~$ \Sigma $ if
$ \, x(t_i) \, \in \, \Pi \setminus \Omega_c, \, $ and (iii) by solving the exit-time optimal open-loop
control problem for the initial state~$ x(t_i) $ if $ x(t_i) \notin \Pi $; 
\item  since the control evaluation outside $ \Pi $ requires in principle more time than that inside $ \Pi $,
it may be reasonable to use an adaptive time grid~$ \: t_i $, $ \, i \, = \, \overline{0, N'}, \: $ such that
$$
\arraycolsep=1.5pt
\def\arraystretch{1.5}
\begin{array}{c}
t_0 \, = \, 0, \\
t_{i + 1} \:\, = \:\, \min(t_i + \Delta, \: T) \quad \mbox{if} \:\,\, x(t_i) \notin \Pi, \\
t_{i + 1} \:\, = \:\, \min(t_i + \Delta', \: T) \quad \mbox{if} \:\,\, x(t_i) \in \Pi, \\
i \, = \, \overline{0, N' - 1}, \quad t_{N' - 1} \, < \, T, \quad t_{N'} \, = \, T, \\
\mbox{$ \Delta $ and $ \Delta' $ are constant steps satisfying $ \: 0 \, < \, \Delta' \, < \, \Delta $}, \\
\mbox{the index~$ N' $ as well as the intermediate time grid nodes are therefore adaptive}.
\end{array}
$$
\end{itemize}
Note that the sparse grid extrapolation outside $ \Pi $ is essentially less accurate than the interpolation inside $ \Pi $.
It is hence recommended to obtain the online control actions outside $ \Pi $ in a more computationally expensive way,
e.\,g., as described in Subsections~\ref{subsect_char_based_method},~\ref{subsect_direct_methods}.

The study~\cite{KangWilcox2017} involving sparse grid and MPC techniques has served as a primary motivation for
the proposed scheme. Applications in \cite{KangWilcox2017} included certain optimal control problems with fixed
finite horizons and without control constraints, and the offline sparse grid data was prepared by solving
characteristic boundary value problems numerically. However, as discussed in Subsection~2.3, the latter may sometimes
have multiple solutions, not all of which are optimal, and the framework of Subsection~\ref{subsect_char_based_method}
therefore deals with characteristic Cauchy problems, while the direct approximation methods mentioned in
Subsection~\ref{subsect_direct_methods} do not rely on characteristics.

As was also emphasized in \cite{KangWilcox2017}, the actual sparse grid interpolation errors may be acceptable for some
models if the state space dimension is not too high and if the sought-after functions are smooth enough, although
the theoretical error estimates are rather conservative. Moreover, the interpolation of costates is preferable to that of
feedback control laws if the latter are expected to have a sharper behavior. Nevertheless, the range of applicability of
sparse grids to solving feedback control problems has to be further investigated.

\end{document}